\documentclass[leqno,10pt]{amsart} %
\usepackage{amssymb}
\usepackage{amstext}
\usepackage{graphicx}
\usepackage{tikz}

\usepackage[left=1.0in, right=1.0in, top=1.25in, bottom=1.25in]{geometry}
\usepackage{graphicx}
\usepackage{amsthm}
\usepackage{hyperref}

\newcommand{\grb}[1]{\raisebox{-.8cm}{\includegraphics[height=2cm]{UP#1.pdf}}}
\newcommand{\grc}[1]{\raisebox{-1.3cm}{\includegraphics[height=3cm]{UP#1.pdf}}}
\newcommand{\grd}[1]{\raisebox{-1.8cm}{\includegraphics[height=4cm]{UP#1.pdf}}}
\newcommand{\gre}[1]{\raisebox{-2.3cm}{\includegraphics[height=5cm]{UP#1.pdf}}}

\theoremstyle{plain} 
\newtheorem{theo}{\indent\sc Theorem}[section]
\newtheorem{lemm}[theo]{\indent\sc Lemma}
\newtheorem{cor}[theo]{\indent\sc Corollary}
\newtheorem{prop}[theo]{\indent\sc Proposition}

\theoremstyle{definition} 
\newtheorem{defi}[theo]{\indent\sc Definition}

\newcommand{\R}{\mathbb R}
\newcommand{\C}{\mathbb C}
\newcommand{\N}{\mathbb N}

\newcommand{\W}{\Lambda}

\newcommand{\Pl}{\mathcal P}

\newcommand{\ca}{\mathcal{C}}
\newcommand{\AW}{\mathcal{A}\Lambda}
\newcommand{\AC}{\mathcal{A}}
\newcommand{\Irr}{\text{Irr}(\ca)}
\newcommand{\Fus}{\mathbb{C}[\Irr]}
\newcommand{\tAW}{\widetilde{\AW}}
\newcommand{\AP}{\mathcal{AP}}
\newcommand{\A}{\mathcal{A}}

\begin{document}

\title[Annular Representation Theory for Rigid C*-Tensor Categories]{Annular Representation Theory for Rigid C*-Tensor Categories}
\author[Shamindra Kumar Ghosh]{\rm Shamindra Kumar Ghosh}
\author[Corey Jones]{\rm Corey Jones }
\thanks{The second author was supported in part by NSF Grant DMS-1362138.}


%%%%%%%%%%%% Authors addresses %%%%%%%%%%%%%
\address[Shamindra Kumar Ghosh]{
India Statistical Institute\endgraf
Stat-Math Unit (Kolmogorov Bhavan)\endgraf
203 B. T. Road\endgraf
Kolkata 700108\endgraf
India}

\email{shamindra.isi@gmail.com}

\address[Corey Jones]{
Vanderbilt University\endgraf
Department of Mathematics\endgraf
1326 Stevenson Center\endgraf
Nashville\endgraf
TN 37240\endgraf
USA}
\email{corey.m.jones@vanderbilt.edu}

%%%%%%%%%%%%%%%%%%%%%%%%%%%%%%%%%%%%%%%%%
\maketitle

\begin{abstract} 

We define annular algebras for rigid $C^{*}$-tensor categories, providing a unified framework for both Ocneanu's tube algebra and Jones' affine annular category of a planar algebra.  We study the representation theory of annular algebras, and show that all sufficiently large (full) annular algebras for a category are isomorphic after tensoring with the algebra of matrix units with countable index set, hence have equivalent representation theories.  Annular algebras admit a universal $C^{*}$-algebra closure analogous to the universal $C^{*}$-algebra for groups.  These algebras have interesting corner algebras indexed by some set of isomorphism classes of objects, which we call centralizer algebras.  The centralizer algebra corresponding to the identity object is canonically isomorphic to the fusion algebra of the category, and we show that the admissible representations of the fusion algebra of Popa and Vaes are precisely the restrictions of arbitrary (non-degenerate) $*$-representations of full annular algebras.  This allows approximation and rigidity properties defined for categories by Popa and Vaes to be interpreted in the context of annular representation theory.  This perspective also allows us to define ``higher weight'' approximation properties based on other centralizer algebras of an annular algebra.  Using the analysis of annular representations due to Jones and Reznikoff, we identify all centralizer algebras for the $TLJ(\delta)$ categories for $\delta\ge 2$.   

\end{abstract}

\begin{section}{Introduction}

Rigid $C^{*}$-tensor categories provide a unifying language for a variety of phenomena encoding ``quantum symmetries''.  For example, they appear as the representation categories of Woronowicz' compact quantum groups, and as ``gauge symmetries'' in the algebraic quantum field theory of Haag and Kastler.  Perhaps most prominently, they arise as categories of finite index bimodules over operator algebras, taking center stage in Jones' theory of subfactors.  The construction and classification of these categories is a very active area of research.  Much of the work in this area has been focused on \textit{unitary fusion categories}, which are rigid $C^{*}$-tensor categories with finitely many isomorphism classes of simple objects.  Categories with infinitely many isomorphism classes of simple objects are called \textit{infinite depth}, and the primary examples come from either discrete groups, representation categories of compact quantum groups, or general categorical constructions, such as the free product, with finite depth examples.  

Infinite depth categories may exhibit interesting analytical behavior analogous to infinite discrete groups.  Approximation and rigidity properties such as amenability, the Haagerup property, and property (T) can be defined for discrete groups in terms of the behavior of sequences of positive definite functions converging to the trivial representation, or equivalently through the properties of the Fell topology on the space of irreducible unitary representations near the trivial representation. In particular, approximation properties guarantee the existence of ``small'' representations converging to the trivial representation, while property (T) asserts that the trivial representation is isolated in the Fell topology.

Following the analogy with groups in the subfactor context, Popa introduced concepts of analytical properties such as amenability, the Haagerup property, and property (T) for finite index inclusions of $II_1$ factors \cite{Po0}, \cite{Po1}, \cite{Po3}, \cite{Po4}.  For a finite index subfactor $N\subseteq M$, Popa introduced the symmetric enveloping inclusion $T\subseteq S$ (see \cite{Po4}). One can view $S$ as a sort of crossed product of $T$ by the category of $M-M$ bimodules appearing in the standard invariant of $N\subseteq M$.  Then one can use sequences of UCP maps $\psi_{n}: S\rightarrow S$ which are $T$-bimodular in place of positive definite functions to define approximation and rigidity properties, with the identity map replacing the trivial representation.  Alternatively, one can use $S-S$ bimodules generated by $T$ central vectors in place of unitary representations.  While these definitions apriori depend on the subfactor $N\subseteq M$, Popa showed that in fact these definitions depend only on the standard invariant of the subfactor.  If the subfactor comes from a group either through the group diagonal construction or the Bisch-Haagerup construction, Popa (\cite{Po3}, \cite{Po4}) and Bisch-Popa (\cite{BiPo}), Bisch-Haagerup (\cite{BH}) respectively, showed that the subfactor has an analytical property if and only if the group does, ensuring that these are in fact the right definitions for these properties in the subfactor setting.

Recently in a remarkable paper, Popa and Vaes show how to extend these definitions to arbitrary rigid $C^{*}$-tensor categories without reference to an ambient subfactor \cite{PV}.  The \textit{fusion algebra} of a category is the complex linear span of isomorphism classes of simple objects, with multiplication given by the fusion rules.  Popa and Vaes define a class of \textit{admissible} representations of the fusion algebra, which take the place of unitary representations of groups.  Approximation and rigidity properties have natural definitions in this setting, and they show that in the case $\ca$ is the category of $M-M$ bimodules for a finite index subfactor $N\subseteq M$, the category has the property if and only if the subfactor does.

One of the goals of this paper is to understand the admissible representation theory of Popa and Vaes as the ordinary representation theory of another algebraic object, namely Ocneanu's tube algebra.  The \textit{tube algebra} $\AC$ is an associative $*$-algebra associated to a rigid $C^{*}$-tensor category $\ca$, introduced by Ocneanu \cite{O}.  In the fusion case this is a finite dimensional semi-simple algebra.  This algebra's significance stems from the fact that irreducible representations of this algebra are in 1-1 correspondence with simple objects in the Drinfeld center $Z(\ca)$.  $Z(\ca)$ is always a modular tensor category, making it of great interest for applications in topological quantum field theory.  Computing the tube algebra provides an algorithmic (though sometimes quite complicated) approach to finding the combinatorial data for $Z(\ca)$ from the combinatorial data of $\ca$.

One approach to studying tensor categories is the \textit{planar algebra} formalism, introduced by Jones in \cite{Jo2}.  A planar algebra packages all the data of a rigid $C^{*}$-tensor category into a single algebraic object, described by planar pictures drawn in disks, along with an action of the operad of planar tangles.  This approach has been very useful, both technically and conceptually, leading to significant progress in both the classification and construction of new examples \cite{JoMorSny}.  Jones introduced the \textit{annular category} of a planar algebra in \cite{Jo4}, with the intention of providing obstructions to the existence of planar algebras with certain principal graphs.  This has been quite successful and is a fundamental technique in the classification of subfactor planar algebras of small index.  A slightly bigger object, the \textit{affine annular category} of the planar algebra was introduced and studied in \cite{Jo3}.  The affine annular category of a planar algebra is obtained by drawing pictures in the interior of annuli rather than disks and applying only local relations.   It was shown in \cite{DGG} that the tensor category of finite dimensional Hilbert space representations of the affine annular category is braided monoidal equivalent to the Drinfeld center of the projection category of the planar algebra.  A similar result in the TQFT setting was shown by Walker \cite{W1}.

It is therefore not surprising that the affine annular category of a planar algebra and the tube algebra of the underlying category have equivalent representation theories, since the category of finite dimensional representations of both algebras are equivalent to the Drinfeld center.  In this paper, we introduce \textit{annular algebras} $\AW$, with weight set $\W\subseteq [Obj(\ca)]$.  Choosing $\W:=\Irr$ yields the tube algebra of Ocneanu, denoted $\AC$, while choosing $\W$ based on a planar algebra description yields the affine annular category $\AP$ of  Jones.  We show that all sufficiently large (full) annular algebras are isomorphic after tensoring with the $*$-algebra of matrix units with countable index set, hence have equivalent representation theories, unifying the two perspectives and providing a means of translating results from planar algebras to the tube algebra in a direct way.  

With a unified perspective in hand, we investigate annular algebras of a $C^{*}$-tensor category.  For each object $k\in \W$, there is a corner of the annular algebra denoted $\AW_{k,k}$ which is a unital $*$-algebra.  If we denote the identity object by $0$, then $\AW_{0,0}$ is canonically $*$-isomorphic to the fusion algebra of $\ca$.  We show that admissible representations of the fusion algebra in the sense of Popa and Vaes are precisely representations of the fusion algebra which are restrictions of $*$-representations of the tube algebra (or any full annular algebra).  This allows us to put context to the admissible representations of \cite{PV} in a natural way.  Inspired by the work of Brown and Guentner \cite{BG}, we can also define analytical properties for arbitrary weights $k\in \W$ by studying the admissible representations of the algebra $\AW_{k,k}$. 

We remark that shortly before the original version of this paper appeared, Neshveyev and Yamashita showed that the admissible representations of Popa and Vaes arise from objects in $Z(\text{ind-}\ca)$ \cite{NY2}.  Shortly after our paper appeared, Stefaan Vaes pointed out that representations of the tube algebra are in bijective correspondence with objects in $Z(\text{ind-}\ca)$, completing the circle between the three different points of view.  A detailed discussion of this correspondence will appear in a paper currently in preparation by S. Popa, D. Shlyakhtenko, and S. Vaes.

The $C^{*}$-algebras that appear as corners of the tube algebra in the Temperley-Lieb-Jones categories $TLJ(\delta)$ for $\delta\ge 2$ are unital, abelian $C^{*}$-algebras hence isomorphic to the continuous functions on compact Hausdorff spaces.  The spaces appear to be rather interesting.  Let $\delta\ge 2$.  We define the following topological spaces:

For $k$ even, $k>0$, $X_{k}:=\grb{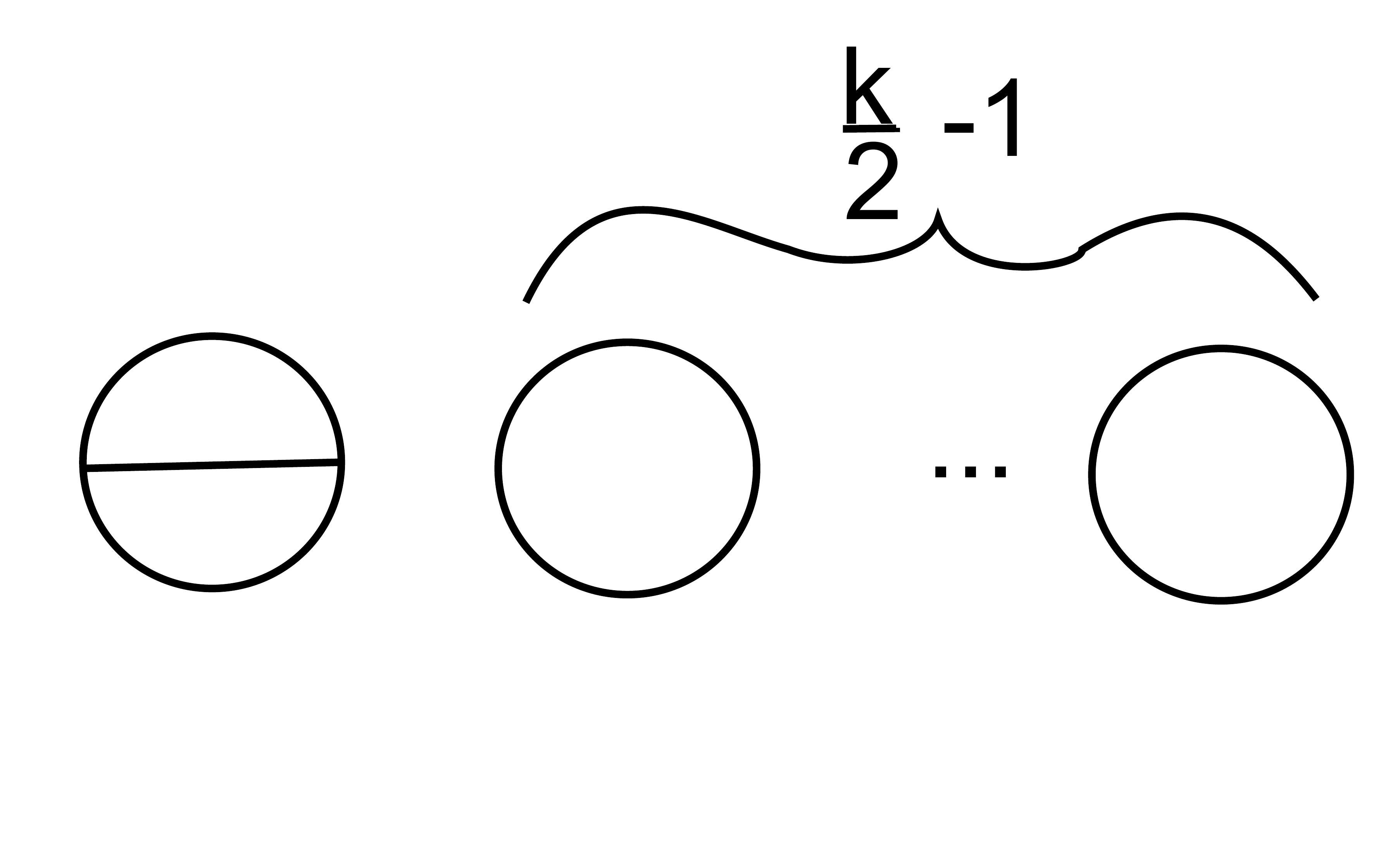}$.  For $k$ odd, define $X_{k}:=\grb{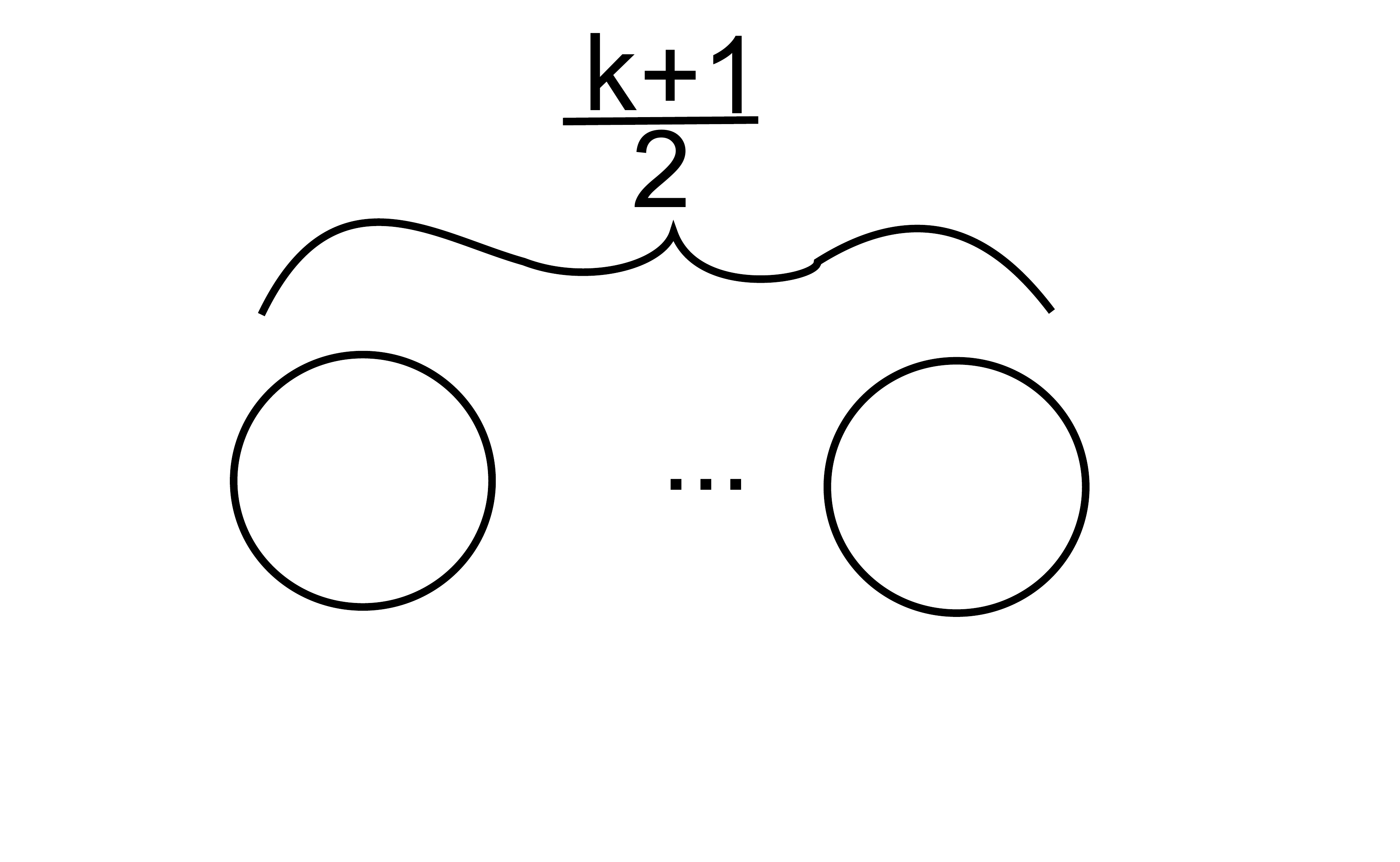}$

For $k$ even, $k>0$ define $Y_{k}:=\grb{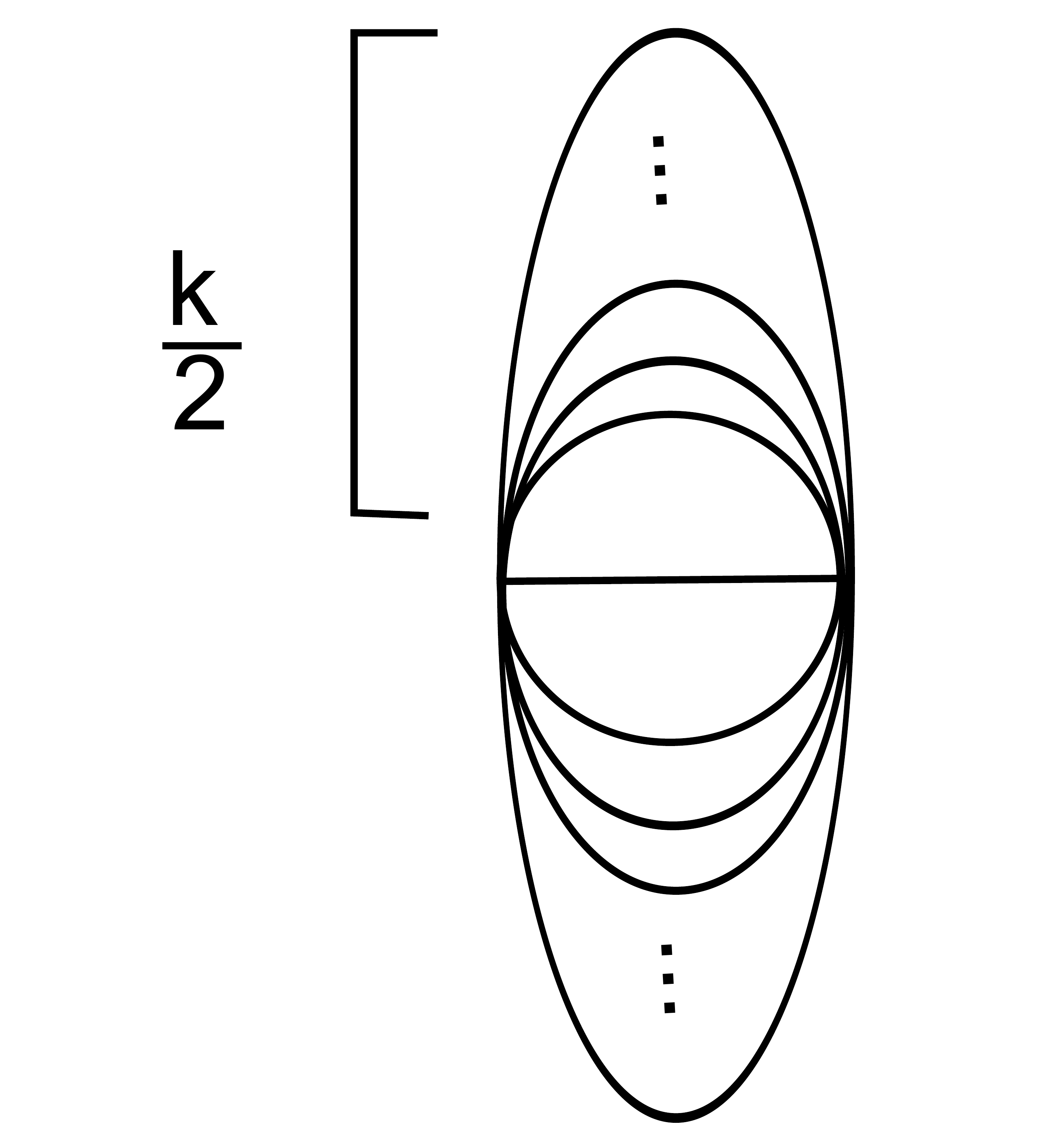}$.  For $k$ odd, define $Y_{k}:=\grb{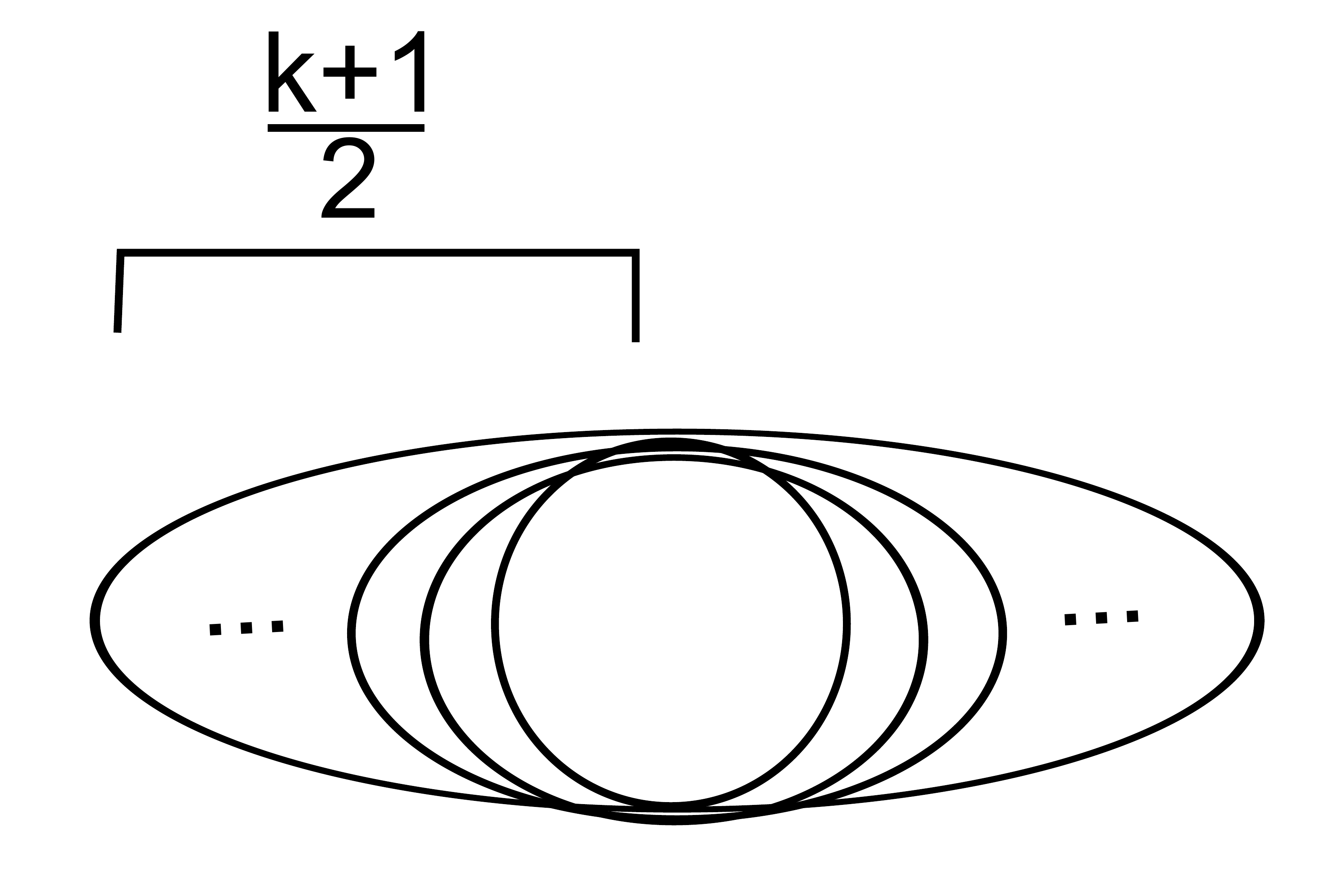}$

We define $X_{0}=Y_{0}:=[-\delta,\delta]$.

\medskip

\ \  We let $\A$ denote the tube algebra of the $TLJ(\delta)$ categories, and $\A_{k,k}$ be the corner corresponding to the $k^{th}$ Jones-Wenzl idempotent.  In this paper we prove the following:

\begin{theo} If $\delta>2$, then $C^{*}(\A_{k,k})\cong C(X_{k})$.  If $\delta=2$, $C^{*}(\A_{k,k})\cong C(Y_{k})$.
\end{theo}

\medskip

This yields a topological characterization of the centralizer algebras of the tube algebras and hopefully will yield a deeper topological understanding of the category $Rep(ATL)$.  This result highlights a key point that was uncovered by Jones-Reznikoff in \cite{Jo3}:  in terms of annular representation theory, the category $TLJ(2)$ is non-generic.  Here we see that the space arising is topologically distinct from the case $\delta>2$, and thus we see that the algebras $C^{*}(\A_{k,k})$ distinguish these two cases.  In general, it seems that the ``non-smooth" points arise precisely from the existence of an actual unitary half-braiding.  In the $\delta=2$ case, the many ``non-smooth" points are the result of the standard braidings on $TLJ(\delta)$ being unitary (whereas they are not for $\delta>2$).

 The structure of the paper is as follows: In Section $2$ we briefly review rigid $C^{*}$-tensor categories. In Section $3$, we define annular algebras over $\ca$, in particular the tube algebra.  In Section $4$ we describe the basic annular representation theory and the universal norm for the tube algebra.  In Section $5$ we present our analysis of some examples, in particular $G$-Vec for a discrete group $G$ and $TLJ(\delta)$.  Section $6$ discusses approximation and rigidity properties and the relationship to the work of Popa and Vaes, as well as our definitions of analytical properties for arbitrary weights.  

\subsection{Acknowledgements}
We would like to thank Dietmar Bisch, Arnaud Brothier, Vaughan Jones, Jesse Peterson and Ved Gupta for many useful discussions on these topics, and Ben Hayes for his suggestion to look at the work of Brown and Guentner for higher weight approximation and rigidity properties.  We are grateful to Stefaan Vaes for his helpful comments, and pointing out to us the equivalence between $Rep(T\ca)$ and $Z(\text{ind-}\ca)$. We thank Makoto Yamashita for his correspondence, and pointing out the paper \cite{Pu} to us, as well as identifying an incorrect statement in an earlier version of this paper.  The second author was supported in part by NSF Grant DMS-1362138.

\end{section}

\begin{section}{Preliminaries: Rigid $C^{*}$-Tensor Categories}

In this paper we will be concerned with semi-simple $C^{*}$-categories with strict tensor functor, simple unit and duals.  We also assume that $\ca$ has at most countably many isomorphism classes of simple objects.  This type of rigid $C^{*}$-tensor category is by far the most frequently studied.  We will briefly elaborate on the meaning of each of these adjectives.

A $C^{*}$-category is a $\C$-linear category $\ca$, with each morphism space $Mor(X,Y)$ a Banach space satisfying $\| fg\|\le \|f \| \| g\|$ together with a conjugate-linear, involutive, contravariant functor $*:\mathcal{C}\rightarrow \ca$ which fixes objects and satisfies the $C^{*}$-property, namely $||f^{*}f||=||ff^{*}||=||f||^{2}$ for all morphisms $f$.  This makes each endomorphism algebra $Mor(X, X)$ into a $C^{*}$-algebra, and we also require that for all $f\in Mor(X,Y)$, $f^{*}f$ is positive in $Mor(X,X)$ for all objects $X,Y$.  We say the category is \textit{semi-simple} if the category has direct sums, sub-objects, and each $Mor(X,Y)$ is finite dimensional.

A strict tensor functor is a bi-linear functor $\otimes: \ca \times \ca\rightarrow \ca$, which is associative and has a distinguished unit $id\in Obj(\ca)$ such that $X\otimes id=X=id \otimes X$.  In general, the strictness assumption is quite strong and most tensor categories arising naturally in mathematics do not satisfy this condition, but rather the more complicated pentagon and triangle axioms (see, for example, \cite{NT}, Chapter 2).  However, every tensor category is equivalent in the appropriate sense to a strict one, so it is convenient when studying categories up to equivalence to include this condition.

The category is \textit{rigid} if for each $X\in Obj(\ca)$, there exists $\overline{X}\in Obj(\ca)$ and morphisms $R\in Mor(id, \overline{X}\otimes X)$ and $\overline{R}\in Mor(id, X\otimes \overline{X})$ satisfying the so-called conjugate equations:

$$(1_{\overline{X}}\otimes \overline{R}^{*})(R\otimes 1_{\overline{X}})=1_{\overline{X}}\ \text{and}\ (1_{X}\otimes R^{*})(\overline{R}\otimes 1_{X})=1_{X}$$

We say two objects $X,Y$ are (unitarily) \textit{isomorphic} if there exists $f\in Mor(X,Y)$ such that $f^{*}f=1_{X}$ and $f f^{*}=1_{Y}$.  We call an object $X$ \textit{simple} if $Mor(X,X)\cong \mathbb{C}$.  We note that for any simple objects $X$ and $Y$, $Mor(X,Y)$ is either isomorphic to $\mathbb{C}$ or $0$. Two simple objects are isomorphic if and only if $Mor(X,Y)\cong \mathbb{C}$. Isomorphism defines an equivalence relation on the collection of all objects and we denote the equivalence class of an object by $[X]$, and the set of isomorphism classes of simple objects $\text{Irr}(\mathcal{C})$.

The semi-simplicity axiom implies that for any object $X$, $Mor(X,X)$ is a finite dimensional  $C^{*}$-algebra over $\mathbb{C}$, hence a multi-matrix algebra.  It is easy to see that each summand of the matrix algebra corresponds to an equivalence class of simple objects, and the dimension of the matrix algebra corresponding to a simple object $Y$ is the square of the multiplicity with which $Y$ occurs in $X$.  In general for a simple object $Y$ and any object $X$, we denote by $N^{Y}_{X}$ the natural number describing the multiplicity with which $[Y]$ appears in the simple object decomposition of $X$.  If $X$ is equivalent to a subobject of $Y$, we write $X\prec Y$.  We often write $X\otimes Y$ simply as $XY$ for objects $X$ and $Y$.

 For two simple objects $X$ and $Y$, we have that $[X\otimes Y]\cong \oplus_{Z} N^{Z}_{XY}[Z]$.  This means that the tensor product of $X$ and $Y$ decomposes as a direct sum of simple objects of which $N^{Z}_{XY}$ are equivalent to the simple object $Z$.  The $N^{Z}_{XY}$ specify the \textit{fusion rules} of the tensor category and are a critical piece of data.

The fusion algebra is the complex linear span of isomorphism classes of simple objects $\C[\text{Irr}(\ca)]$, with multiplication given by linear extension of the fusion rules.  This algebra has a $*$-involution defined by $[X]^{*}=[\overline{X}]$ and extended conjugate-linearly.  This algebra is a central object of study in approximation and rigidity theory for rigid $C^{*}$-tensor categories.  

For a more detailed discussion and analysis of the axioms of a rigid $C^{*}$-tensor category, see the paper of Longo and Roberts \cite{LR} and Chapter 2 of the book by Neshveyev and Tuset \cite{NT}.  For the discussion of $C^{*}$-tensor categories and their relationship with other notions of duality in tensor categories see the paper of Mueger \cite{Mu1}. 

In a rigid $C^{*}$-tensor category, we can define the statistical dimension of an object $d(X)=inf _{(R,\overline{R})} || R|| ||\overline{R}||$, where the infimum is taken over all solutions to the conjugate equations for an object $X$.  The function $d(\ .\ ):Obj(\ca)\rightarrow \mathbb{R}_{+}$ depends on objects only up to unitary isomorphism.  It is multiplicative and additive and satisfies $d(X)=d(\overline{X})$ for any dual of $X$.  We call solutions to the conjugate equations \textit{standard} if $||R||=||\overline{R}||=d(X)^{\frac{1}{2}}$, and such solutions are essentially unique.  For standard solutions of the conjugate equations, we have a well defined trace $Tr_{X}$ on endomorphism spaces $Mor(X,X)$ given by

 $$Tr_{X}(f)=R^{*}(1_{\overline{X}}\otimes f)R=\overline{R}^{*}(f\otimes 1_{\overline{X}})\overline{R}\in Mor(id,id)\cong \C$$  

This trace does not depend on the choice of dual for $X$ or on the choice of standard solutions.  We note that $Tr(1_{X})=d(X)$.  See \cite{LR} for details.

We will frequently use the well known \textit{graphical calculus} for tensor categories.  See, for example, Section 2.5 of \cite{Mu1} or \cite{Y12}.  We refer the reader to \cite{BHP} for the closely related planar algebra perspective.
\end{section}

\begin{section}{Annular Algebras}

The tube algebra $\AC$ of a rigid $C^{*}$-tensor category $\ca$ was introduced by Ocneanu in \cite{O} in the subfactor context.  This algebra has proved to be useful for computing the Drinfeld center $Z(\ca)$, since finite dimensional irreducible representations of $\AC$ are in one-to-one correspondence with simple objects of $Z(\ca)$ (see \cite{I}).  In general, arbitrary representations of $\AC$ are in one-to-one correspondence with objects in $Z(\text{ind-}\ca)$ studied by Neshveyev and Yamashita in \cite{NY2}, an observation due to Stefaan Vaes.

The (affine) annular category of a planar algebra was introduced by Jones in \cite{Jo3}, \cite{Jo4}, with the purpose of providing obstructions to the existence of subfactor planar algebras with certain principal graphs.  Since every planar algebra $\Pl$ with index $\delta$ contains the Temperley-Lieb-Jones planar algebra $TLJ(\delta)$, one can decompose $\Pl$ as a direct sum of irreducible representations of the annular $TLJ(\delta)$ category.  The irreducible representations of $TLJ(\delta)$ were completely determined by Jones \cite{Jo4} and Jones-Reznikoff \cite{Jo3}, yielding a useful tool for the classification program of subfactors.

Here we introduce a mild generalization of both the algebraic structures described above, which we call an annular algebra of the category.  It depends on a choice of objects in the category, and is flexible enough to include both Ocneanu's tube algebra and Jones' affine annular categories as special cases.  The tube algebra is in some sense a minimal example, while the affine annular category of a planar algebra is particularly suitable in the case when the category arises as the projection category of a planar algebra with a nice skein theoretic presentation.   Proposition 3.5 shows that any two ``sufficiently large'' annular algebras (a class which include both the above mentioned examples) have equivalent representation theories in a strong sense.  This result allows us to translate results of Jones-Reznikoff on the affine annular $TLJ(\delta)$ to the tube algebra setting in Section 5.

For a rigid $C^{*}$-tensor category $\ca$, choose a set of representatives $X_{k}\in k$ for each $k\in\Irr$.  Let $0\in \Irr$ denote the equivalence class of the tensor unit, and choose $X_{0}$ to be the strict tensor unit. 

Let $[Obj(\ca)]$ be the set of equivalence classes of objects in $\ca$.  Let $\W$ be a subset of $[Obj(\ca)]$. For each $i\in \W$, we choose a representative $Y_{i}\in i$.  Then we define the \textit{annular algebra with weight set $\W$}

$$\AW:=\bigoplus_{i,j\in \W,\ k\in \Irr} Mor(X_{k}\otimes Y_{i} , Y_{j}\otimes X_{k})$$

An element $x\in \AW$  is given by a sequence $x^{k}_{i,j}\in Mor(X_{k}\otimes Y_{i}, Y_{j}\otimes X_{k})$ with only finitely many terms non-zero.  For a simple object $\alpha$ and  and arbitrary object $\beta$, $Mor(\alpha, \beta)$ has a Hilbert space structure with inner product defined by $\eta^{*}\xi=\langle \xi, \eta\rangle  1_{\alpha}$.  Note that this inner product differs from the tracial inner product by a factor of $d(\alpha)$.

$\AW$ carries the structure of an associative $*$-algebra, with associative product $\cdot$ and $*$-involution $\#$ defined by

$$(x\cdot y)^{k}_{i,j}=  \sum_{s \in \Lambda , m, l\in \Irr}\ \sum_{V\in onb(X_k,\ X_m\otimes X_l)}(1_{j}\otimes V^{*}) ( x^{m}_{s,j}\otimes 1_{l})(1_{m}\otimes y^{l}_{i,s})(V\otimes 1_{i})$$

$$(x^{\#})^{k}_{i,j}=  (\overline{R}^{*}_{k}\otimes 1_{j}\otimes 1_{k}) (1_{k}\otimes (x^{\overline{k}}_{j,i})^{*}\otimes 1_{k})(1_{k}\otimes 1_{i}\otimes R_{k})$$

\medskip

\noindent where $R_{k}\in Mor(id, \overline{X}_{k}\otimes X_{k})$ and $\overline{R}_{k}\in   Mor(id, X_k\otimes \overline{X}_{k})$ are standard solutions to the conjugate equations for $X_{k}$.  In the first sum, $onb$ denotes an orthonormal basis with respect to our inner product, and we may have $onb(X_k,\ X_ m\otimes X_ l)=\varnothing$ if $X_{k}$ is not equivalent to a sub-object of $X_{m}\otimes X_{l}$.  We mention the above compact form for the definition was borrowed from Stefaan Vaes.  It is clear that the isomorphism class of this algebra does not depend on the choices of representatives $X_{k}$.  We often write the sequence of morphisms as a sum $\displaystyle x=\sum_{i,j\in \W,\ k\in \Irr} x^{k}_{i,j}$, where only finitely many terms are non-zero. 

We denote the subspaces $\AW^{k}_{i,j}:=Mor(X_{k}\otimes Y_{i}, Y_{j}\otimes X_{k})\subset \AW$, and $\AW_{i,j}=\bigoplus_{k\in \Irr} \AW^{k}_{i,j}$.  For each $m\in \W$, there is a projection $p_{m}\in \AW^{0}_{m,m}$ given by $p_{m}:=1_{m}\in Mor(id\otimes Y_{m}, Y_{m}\otimes id)\in \AW$.  In particular $(p_{m})^{k}_{i,j}=\delta_{k,0}\delta_{i,j}\delta_{j,m} 1_{m}$.  We see that $\AW_{i,j}=p_{j} \AW p_{i}$.   The corner algebras $\AW_{m,m}=p_{m}\AW p_{m}$ are unital $*$-algebras.  We call $\AW_{m,m}$ the \textit{weight m centralizer algebra}.  The motivation for the terminology comes from the case when $\ca$ is $G-Vec$ for a discrete group $G$. In this example $m\in \Irr$ corresponds to an element of the group $G$, and $\AW_{m,m}$ is isomorphic to the group algebra of the centralizer subgroup of the element $m$ (see Section $5.1$).

Suppose $\W$ contains the strict tensor identity, labelled as usual by $X_{0}$.  Recall the fusion algebra of $\ca$ is the complex linear span of isomorphism classes of simple objects $\C[\text{Irr}(\ca)]$.  Multiplication is the linear extension of fusion rules and $*$ is given on basis elements by the duality.  From the definition of multiplication in $\AW$, one easily sees the following:

\begin{prop}The fusion algebra $\C[\text{Irr}(\ca)]$ is $*$-isomorphic to $\AW_{0,0}$, via the map $[X_{k}]\rightarrow 1_{k}\in (X_{k}\otimes id, id\otimes X_{k})\in \AW^{k}_{0,0}$.
\end{prop}  

\begin{defi} The \textit{annular category with weight set $\W$} is the category where $\Lambda$ is the space of objects, and the morphism space from $k$ to $m$ is given by $\displaystyle \AW_{k,m} :=  \bigoplus_{j \in Irr (\mathcal C)} \mathcal A \Lambda^j_{k,m}$.  Composition is given by the restriction of annular multiplication.  
\end{defi}

The annular category and annular algebra basically contain the same information, so one can go between these two perspectives at leisure.  We feel the algebra perspective is slightly more convenient for the purpose of representation theory, however, any analysis of the algebra seems to always reduce to studying the centralizer algebras first, so the two points of view are not actually distinct in practice.  We remark that this category is \textit{not} a tensor category in general.

We introduce a bit of graphical calculus for annular algebras, extending the well known graphical calculus for tensor categories.  For $x\in Mor(X_k\otimes Y_n,\ X_m \otimes X_{k})$, we draw the picture \ \ \     $\grb{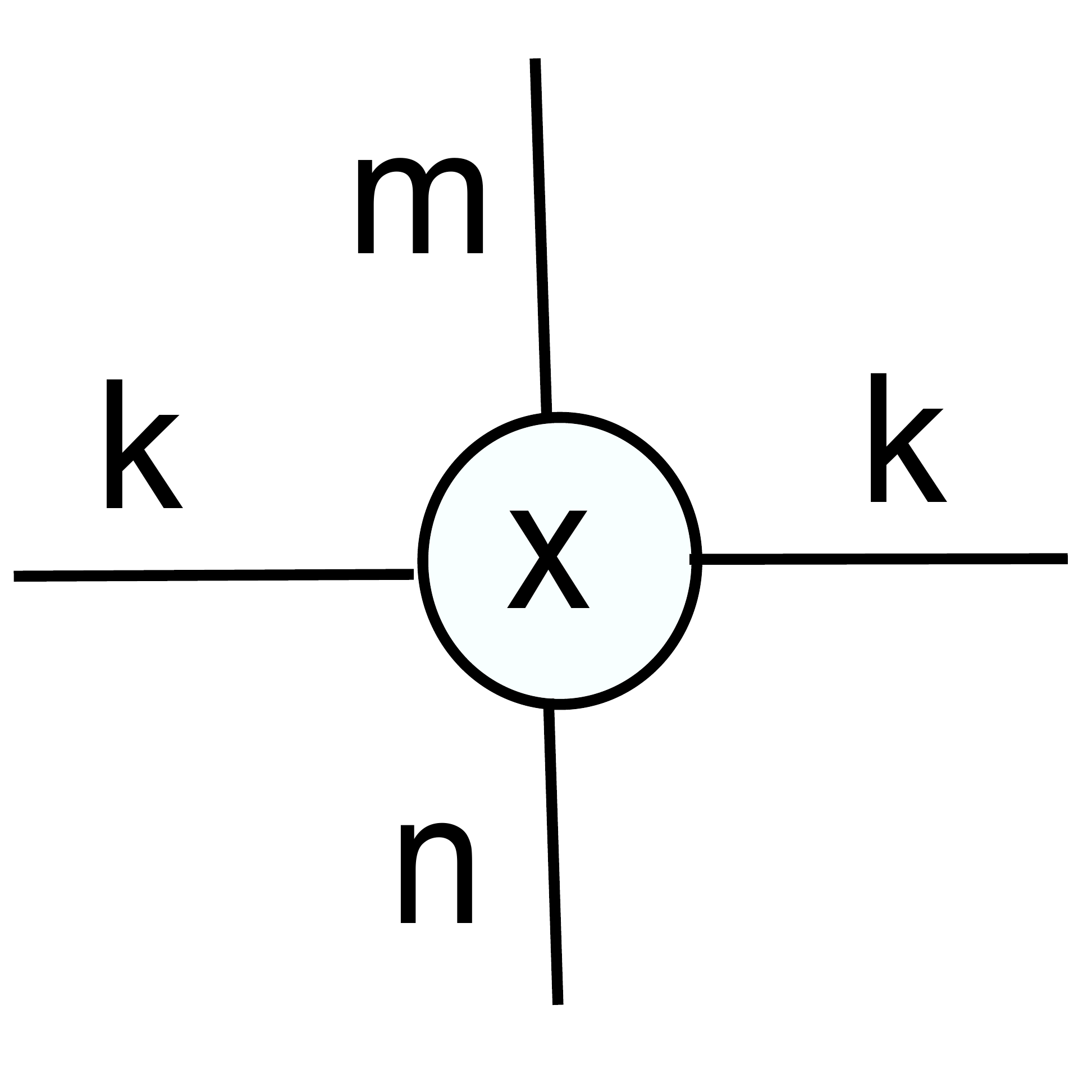}$.

Conversely, if we see such a picture with top, bottom and side strings labeled with a $x$, then $x$ will represent a morphism in the space obtained by pulling the left side string down to the bottom left and the right side string up to the top right. For example, the picture

$\grc{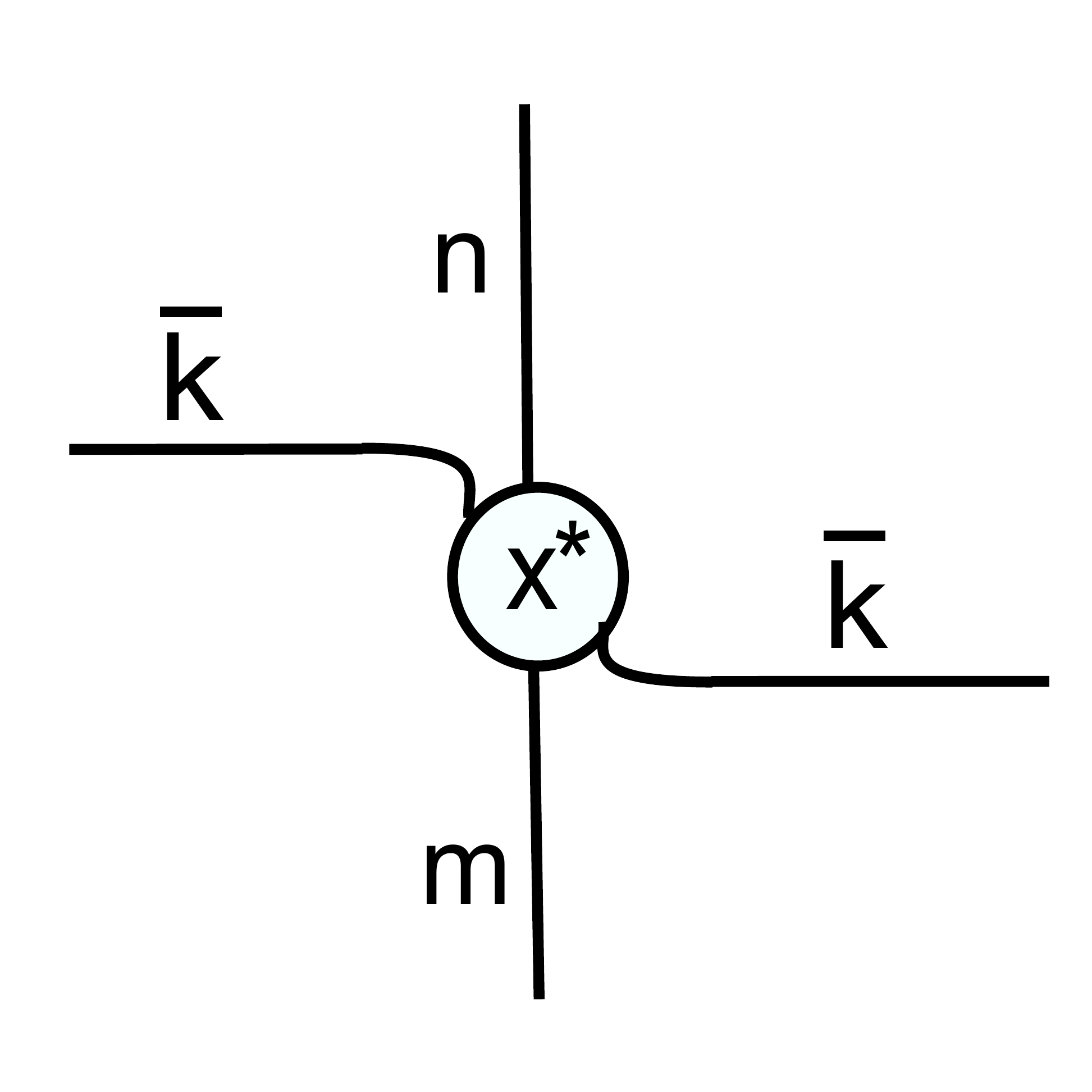}$\ \  represents the morphism $$x^{\#}=(\overline{R}^{*}_{\overline{k}}\otimes 1_{n}\otimes 1_{\overline{k}}) (1_{\overline{k}}\otimes (x)^{*}\otimes 1_{\overline{k}})(1_{\overline{k}}\otimes 1_{m}\otimes R_{\overline{k}})\in Mor(X_{\overline{k}}\otimes Y_m, Y_{n}\otimes X_{\overline{k}}),$$

 where $x\in Mor(X_{k}\otimes Y_{n}, Y_{m}\otimes X_{k})$ is described above.  As we shall see, this graphical calculus will be convenient for writing certain identities and equations that may take a large amount of space to write as compositions and tensor products of morphisms, but consist of a simple picture using this formalism.  We remark that diagrams having no side strings can be interpreted as  morphisms in the category, and our graphical calculus restricts to the standard graphical calculus for tensor categories.

Annular algebras have a positive definite trace, given by $\Omega(x):=\sum_{k\in \Irr} Tr_{k}(x^{0}_{k,k})$, where $Tr_{k}$ denotes the canonical (unnormalized) trace on $Mor(Y_{k}, Y_{k})$, and we canonically identify $Mor(id\otimes Y_{k}, Y_{k}\otimes id)\cong Mor(Y_{k}, Y_{k})$; positive definiteness of $\Omega$ can be deduced following the same line of arguments used in the proof of \cite[Proposition 3.7]{DGG}. If we let $tr_{k}(\ . \ ):= \frac{1}{d(X_{k})}Tr_{k}(\ . \ )$, we define $\omega$ by the same formula, replacing $Tr_{k}$ with $tr_{k}$.  It is easy to see that $\Omega$ is a tracial functional on $\AW$, while $\omega$ is not due to the normalization factor.  It will be convenient, however, to have both functionals at hand.

%We denote the subspaces $\AW^{k}_{i,j}:=Mor(X_{k}\otimes Y_{i}, Y_{j}\otimes X_{k})\subset \AW$.

%For each $k\in \W$, there is a projection $p_{m}\in \AW^{0}_{k,k}$ given by $p_{m}:=1_{m}\in Mor(id\otimes Y_{m}, Y_{m}\otimes id)\in \AW$.  In particular $(p_{m})^{k}_{i,j}=\delta_{k,0}\delta_{i,j}\delta_{j,m} 1_{m}$.  We see that the corner $\AW_{k,k}=p_{k}\AW p_{k}$ is a unital $*$-algebra, and $\displaystyle \AW_{k,k}\cong \bigoplus_{j\in \Irr} \AW^{j}_{k,k}$.   We call the corner $\AW_{k,k}$ the \textit{weight k centralizer algebra}.

%We define the \textit{annular category with weight set $\W$} as the category with objects $\W$, and $Mor(k,m):=\AW_{k,m}$.  Composition is given by the restriction of annular multiplication.  The annular category and annular algebra basically contain the same information, so one can go between these two perspectives at leisure.  We feel the algebra perspective is slightly more convenient from the representation theory point of view, however, any analysis of the algebra seems to always reduce to studying the centralizer algebras first, so the two points of view are not actually distinct in practice.

\begin{defi}\label{tubalg} The \textit{tube algebra} is the annular algebra with weight set $\Irr$.  We denote the tube algebra $\AC$.
\end{defi}

The tube algebra is the ``smallest'' annular algebra that contains all of the information of the annular representation theory of the category as described in the next section, and hence is the best for many purposes.  In fact, a sufficiently large arbitrary annular algebra is ``Morita equivalent'' to the tube algebra.  Our notion sufficiently large is given by the following definition:

\begin{defi} A weight set $\W \subseteq Obj(\ca)$ is \textit{full} if every simply object is equivalent to a sub-object of some $X_{k},\ k\in\W$.
\end{defi}
For a countable set $I$, let $F(I)$ denote the $*$-algebra spanned by the system of matrix units $\{ E_{i,j} \in B(l^2(I)) : i,j \in I \}$ with respect to the orthonormal basis $I$ in $l^2(I)$. Further, for sets $I,J$, we will denote the span of the system of matrix units $\{ E_{i,j} \in B(l^2(I) , l^2(J)) : i \in I, j \in J \}$ by $F(I,J)$.
\begin{prop} If $\W$ is full, then $F(I) \otimes \A\cong F(I) \otimes \AW$ as $*$-algebras.
\end{prop}

\begin{proof}

% If $\W$ is full, for each $i\in \Irr$, we can find some $k_{i}\in \W$ and a projection $q_{i}\in Mor(Y_{k_{i}}, Y_{k_{i}})\subseteq \AW_{k_{i},k_{i}}$ equivalent to the object $X_{i}$ (recall $\ca$ is semi-simple so is equivalent to its projection completion).   We note that the spaces $\AC^{m}_{i,j}:=Mor(X_{m}\otimes X_{i}, X_{j}\otimes X_{m})$ and $q_{j}\AW^{m}_{k_i, k_{j}} q_{i}:=Mor(X_{m}\otimes q_{i}, q_{j} \otimes X_{m})$ are isomorphic since $q_i\cong X_{i}$ and these objects are simple.  We here we choose an isometry $v_{i}:q_{i}\rightarrow X_{i}$ implementing the equivalence for each $i$, and we have an injective morphisms $\lambda: \AC_{i,j}\rightarrow \AW_{k_{i}, k_{j}}$, defined for $f\in \AC^{m}_{i,j}$ by
% $$\lambda(f):=(v_{j}^{*}\otimes 1_{m})f(1_{m}\otimes v_{i})\in \AW^{m}_{k_{i}, k_{j}}.$$
% It is easy to see that this isomorphism extends to a $*$-algebra isomorphism from $\AC$ onto a sub-algebra of $\AW$, which we call $\lambda$.   Distinct choices of $\lambda$ will differ by unit scalars on each morphism space, but it is easy to see that any such choice yields an isomorphism since each isometry yielding non-zero multiplication occurs with its adjoint.
 
We see abstractly that $\displaystyle \AW^{k}_{m,n}\cong \bigoplus_{s,t\in \Irr} Mor(X_t, Y_n)\otimes \AC^{k}_{s,t}\otimes \overline{Mor(X_{s}, Y_{m})}$ since an arbitrary element $f\in \AW^{k}_{m,n}$ can be decomposed uniquely as: $$\displaystyle f = \sum_{s,t\in \Irr} \sum_{\substack{V \in onb(X_t , Y_n) \\ W \in onb(X_s , Y_m)}}
% d(X_s) d(X_t)
\left[  (V V^{*} \otimes 1_{k}) f (1_{k}\otimes W W^{*}) \right]$$ where $onb(X_s , Y_m)$ is an orthonormal basis for $Mor(X_{s}, Y_{m})$ with respect to the inner product defined in the definition of annular algebras.  We see this decomposition does not depend on the choice of such a basis.  Thus, the isomorphism implemented by the decomposition is 

$$\displaystyle f \mapsto \sum_{s,t\in \Irr}  \sum_{\substack{V \in onb(X_t , Y_n) \\W \in onb(X_s , Y_m)}} V \otimes \left[  (V^{*} \otimes 1_{k}) f (1_{k}\otimes W ) \right] \otimes \overline{W};$$

This map has its inverse defined by taking $*$ in the third tensor component and then composing the morphisms in the obvious way.
 
  % Note that\\ \noindent for $\mathcal A \Lambda^j_{m,n} \ni f \leftrightarrow a \otimes g \otimes \overline{b} \in Mor(X_t, Y_n) \otimes \mathcal A^j_{s,t} \otimes \overline{Mor(X_s , Y_m)}$\\ \noindent  and $\mathcal A \Lambda^k_{{m'},{n'}} \ni f' \leftrightarrow a' \otimes g' \otimes \overline{b'} \in Mor(X_{t'}, Y_{n'}) \otimes \mathcal A^k_{s',t'} \otimes \overline{Mor(X_{s'} , Y_{m'})}$, we have $$f\cdot f' \mbox{ (in } \mathcal A \Lambda  {) } \; = \; \delta_{m,n'} \; \delta_{s,t'} \; \langle a', b \rangle \; \left[ (a \otimes 1_{jk}) \circ (g \cdot g') \circ (1_{jk} \otimes (b')^*)\right]$$ where $\langle \cdot , \cdot \rangle$ is the inner product on $Mor (X_s, Y_m)$ induced by the same trace, and $f^* \leftrightarrow b \otimes g^* \otimes \overline{a}$. So,
  
 If we let $B_{s,m}$ denote an orthonormal basis of $Mor(X_s, Y_m)$ for all $s\in Irr (\mathcal C)$, $m \in \Lambda$, then we have a vector space isomorphism $$ \mathcal A \Lambda^j_{m,n} \cong \bigoplus_{s,t \in Irr(\mathcal C)} M_{B_{t,n} \times B_{s,m}}(\C) \otimes \mathcal A^j_{s,t} \mbox{, namely}\ (V \otimes 1_j) \circ h \circ (1_j \otimes W^*) \leftrightarrow E_{V,W} \otimes h.$$ Moreover, multiplication and  $\#$ on the whole algebra $\mathcal A \Lambda$ correspond exactly with those on the matrix and the tube algebra parts.

Next, for $s \in \Irr$, we define the set $\displaystyle I_{s}:=\bigsqcup_{m\in \W} I \times B_{s,m}$.  We see that as a $*$-algebra we can identify  $\displaystyle F(I) \otimes \AW \cong \bigoplus_{m,n \in \Lambda} \bigoplus_{s,t \in \Irr} F(I) \otimes M_{B_{t,n} \times B_{s,m}}(\C) \otimes \AC_{s,t} \cong \bigoplus_{s,t \in \Irr} F(I_t , I_s) \otimes \AC_{s,t}$. Since $\Lambda$ is full, $I_{t}$ is non-empty, and we can identify it with $I$ for all $t \in \Irr$. Hence, it follows that $F(I) \otimes \AW \cong F(I) \otimes \A$ as $*$-algebras.
\end{proof}

 As we shall see in the next section, this correspondence allows us to pass between representations of $\AW$ and $\AC$ for any full weight set $\W$. Before studying representation theory, we describe another useful way to realize annular algebras as the quotient of a much bigger graded algebra.  For any weight set $\W$, we define
$$\widetilde{\AW}:=\bigoplus_{\alpha\in Obj(\ca),\ i,j\in \W} Mor(\alpha \otimes Y_{i},\ Y_{j}\otimes \alpha)$$

Notice that the direct sum is taken over $\W$ and \textit{all} objects in contrast with the definition for annular algebras.  As with annular algebras, however, $x\in \tAW$ is denoted by a collection $x^{\alpha}_{i,j}$ where $\alpha\in Obj(\ca)$ and $i,j\in \W$ with only finitely many non-zero term. $\tAW$ becomes an associative algebra with multiplication defined by: %This algebra has the structure of an associative graded $*$-algebra, defined by

$$(x\cdot y)^{\alpha}_{i,j}=  \sum_{s \in \W}\ \sum_{\beta ,\gamma \in Obj(\ca):\ \alpha=\beta \otimes \gamma)}( x^{\beta}_{s,j}\otimes 1_{\gamma})(1_{\beta}\otimes y^{\gamma}_{i,s}).$$

\noindent Note that associativity follows from strictness of our category.  For a $*$-structure, we need duals and standard solutions to the conjugate equations for every $\alpha \in Obj(\mathcal C)$ which are chosen once and for all in a consistent way.  A convenient notion for this purpose is a \textit{spherical structure} in the sense of \cite{Mu1}, Definition 2.6.  Such a choice for any rigid $C^*$-tensor category $\mathcal C$ is always possible by a result of Yamagami (see \cite{Y04}).  Thus we assume that we have chosen a spherical structure, which in particular picks a dual object (along with a standard solution to the conjugate equations) for each object in such a way that $\overline{\overline{\alpha}}=\alpha$.  Since $\tAW$ is built out of morphism spaces which already have a $*$, we will denote the $*$-structure here by $\#$ as in the annular algebra case which is defined as:

$$(x^{\#})^{\alpha}_{i,j}=  (\overline{R}^{*}_{\alpha} \otimes 1_{j} \otimes 1_{\alpha}) (1_{\alpha}\otimes (x^{\overline{\alpha}}_{j,i})^{*}\otimes 1_{\alpha})(1_{\alpha}\otimes 1_{i}\otimes R_{\alpha})$$

It is easy to check that $\#$ is a conjugate-linear, anti-isomorphic involution (by the definition of spherical structure).

We define the family of maps $\Psi^{\alpha}: Mor(\alpha \otimes Y_{i}, Y_{j}\otimes\alpha)\rightarrow \mathcal A \Lambda$ given by 
$$\displaystyle \Psi^{\alpha}(f)=\sum_{k \prec \alpha}\sum_{V\in onb(k, \alpha)}   (1_{j}\otimes V^{*})f(V\otimes 1_{i}).$$

Then the family of $\Psi^{\alpha}$ extends linearly to a surjective map $\Psi: \tAW\rightarrow \AW$.  It is also easy to see that $\Psi$ is a $*$-homomorphism.  Using basic linear algebra, one can see that $Ker(\Psi)$ is spanned by (not necessarily homogeneous) vectors of the form $f(s\otimes 1_{i}) - (1_{j}\otimes s)f\in \tAW$ for $f\in Mor(\alpha \otimes Y_{i}, Y_{j}\otimes \beta)$, and $s\in Mor(\beta, \alpha)$. 

We remark that the graphical calculus for annular algebras makes perfect sense in this setting, we simply allow side strings to be labeled by arbitrary objects.  In fact, we can now give a heuristic explanation for the words tube and annular associated to these algebras.

Take a diagram with top bottom and side strings as in our graphical calculus convention, and attach the bottom string to the inner disk of an annulus and the top strings to the boundary of the outer disk.  Then attach the side strings to each other around the ``bottom'' of the inner disk.  We allow isotopies in the interior of the annulus, so that the following pictures are equal:

 $$\gre{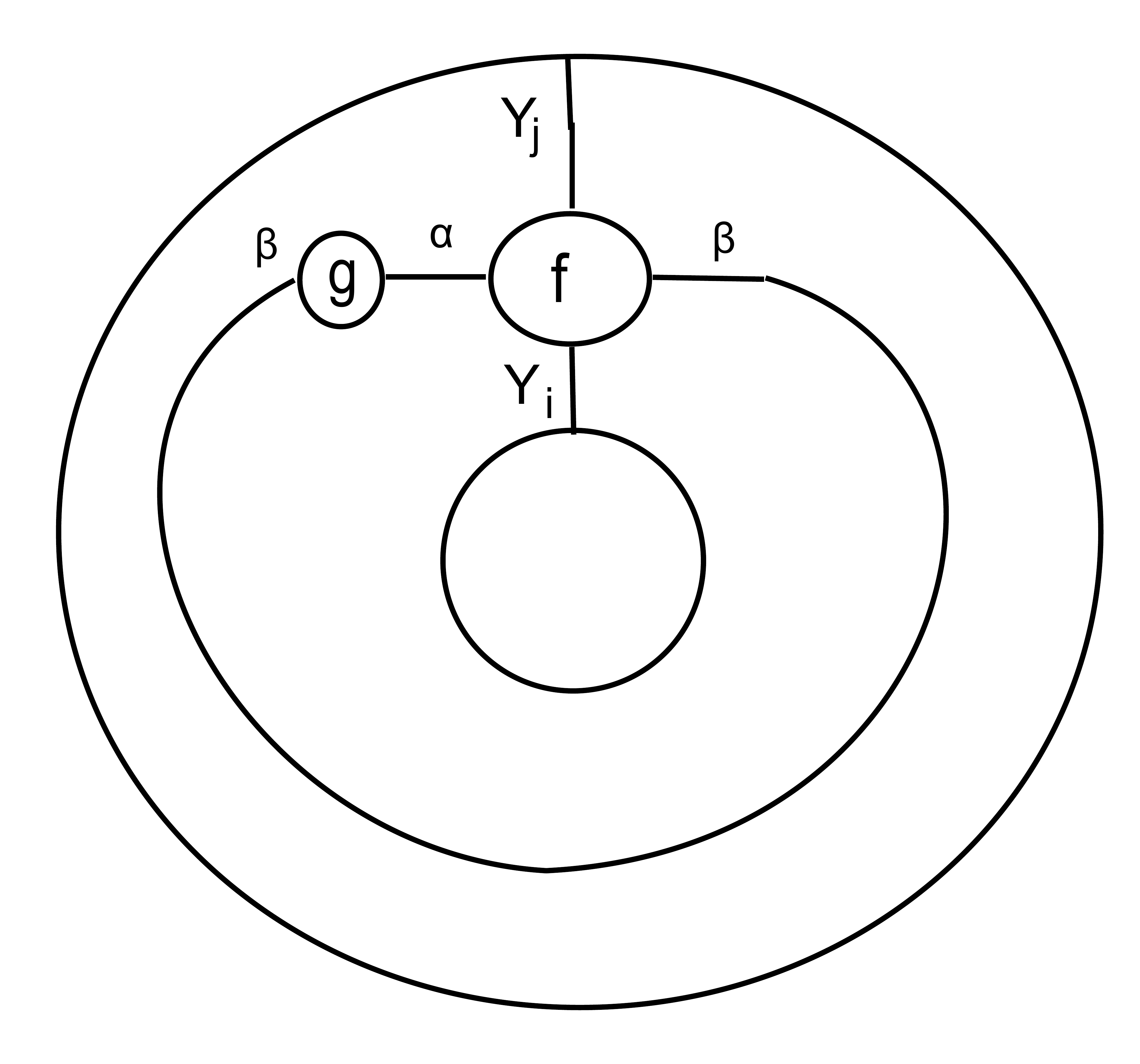}=\gre{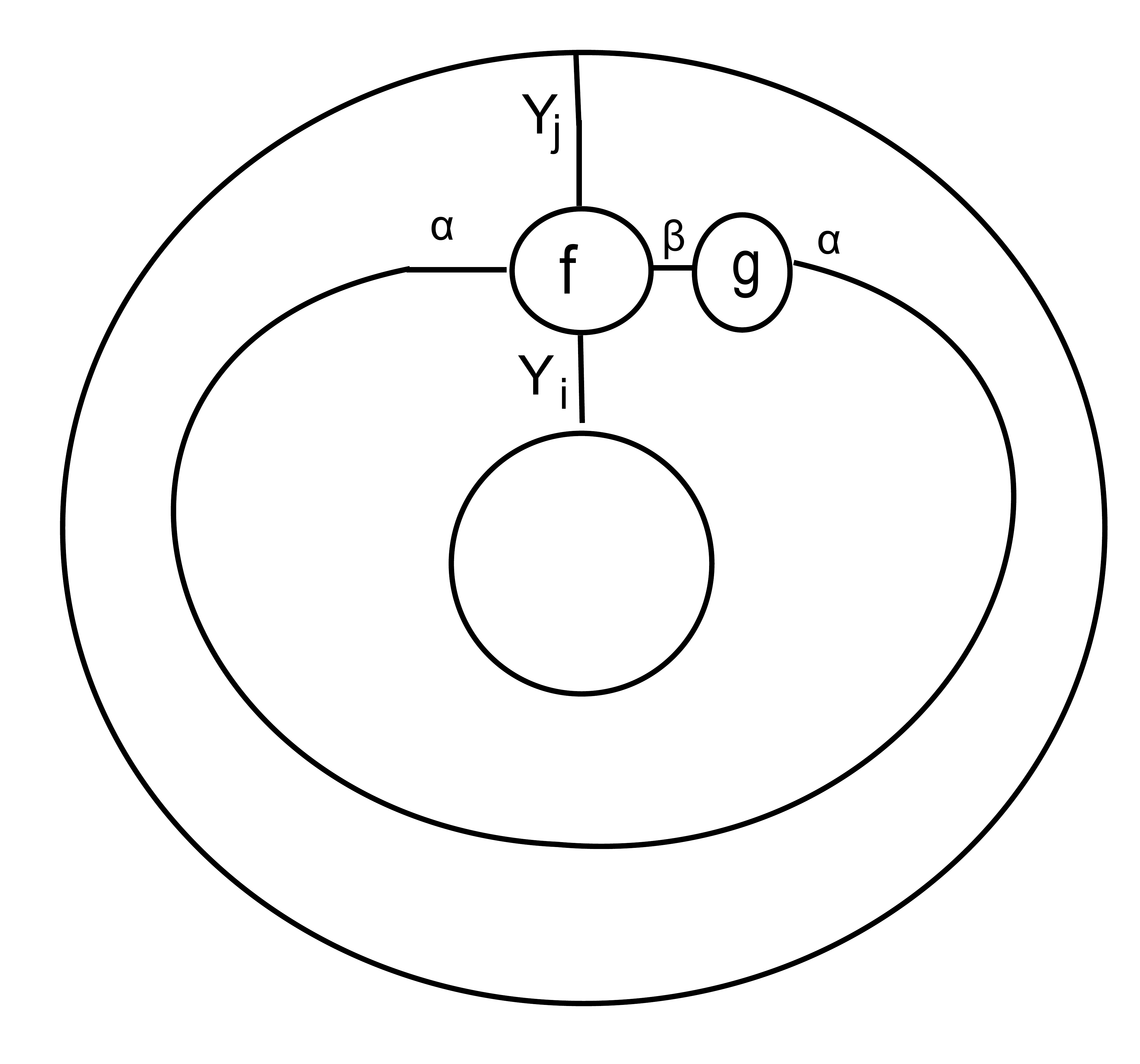}$$

This picture explains the kernel of the map $\Psi$.  Cutting the bottom string and returning to a rectangular picture, the difference of the resulting homs spans $Ker(\Psi)$. We also remark that composing such pictures and decomposing the identity on the side strings yields the multiplication structure we defined for annular algebras.  

Such pictures can be formalized in the setting of Jones' \textit{planar algebras}.  The result is Jones' affine annular category of a planar algebra.  If $\Pl$ is an (unshaded) planar algebra the affine annular category $\AP$ is the category with objects given by $\N$, and morphisms all annular tangles labeled by $\Pl$ subject to local relations.  For proper definitions, see \cite{Jo3}, \cite{Jo4} and \cite{DGG}. Composition of morphisms is given by composing annuli.  This category can be made into an algebra in the obvious way, which we also call $\AP$.  If we let $\ca:=Proj(\Pl)$ be the projection category of a planar algebra, choose the objects $\W:=\{1_{k}\in P_{k,k}\}_{k\in \N}\subseteq Obj(\ca)$.  Then it follows from \cite{DGG} that $\AW\cong \AP$.  We will see an example of this correspondence in section $5$ in our analysis of $TLJ(\delta)$ categories.  We refer the reader to \cite{Jo2} and \cite{BHP} for the definitions of planar algebras and the second reference for the projection category of a planar algebra.

We proceed to analyze the structure of the algebraic dual space of $\AW$.  For $i,j\in \W$, we define the space of \textit{commutativity constraints} by 
$$CC_{i,j}:=\{\prod_{\alpha\in Obj(\ca)} c_{\alpha}\ :\ c_{\alpha}\in Mor(\alpha \otimes Y_{i}, Y_{j}\otimes \alpha),\ \text{and for all}\ f\in Mor(\alpha, \beta),\ c_{\beta}(f\otimes 1_{i})=(1_{j}\otimes f)c_{\alpha}\}$$ 

See immediately that $c=(c_{\alpha})_{\alpha\in Obj(\ca)}$ is uniquely determined by the terms $(c_{i})_{i\in \Irr}$, and any such sequence determines a family. We also notice that the condition defining commutativity constraints is dual to the condition defining $Ker(\Psi)$.  This leads to the following observation:

\begin{prop} The algebraic dual of $\AW_{i,j}$ is canonically isomorphic to $CC_{j,i}$, where $c=(c_{k})_{k\in \Irr}$ acts on $f=\sum_{k\in \Irr} f_{k}$ with $f_{k}\in \AW^{k}_{i,j}$ by $c(f)=\sum_{k} Tr((1_j\otimes \overline{R}^{*}_{k})(f_{k}\otimes 1_{\overline{k}})(1_k\otimes c_{\overline{k}})(\overline{R}_{k}\otimes 1_j))$.
\end{prop}

  We encourage the reader to draw a picture of the above equation.  Here we view $f_{k}$ and $c_{k}$ as morphisms in $Mor (X_k\otimes Y_i,\ Y_j\otimes X_k)$ and $Mor (X_k\otimes Y_j,\ Y_i\otimes X_k)$ respectively, and composition is categorical (not annular) composition.

For more details on commutativity constraints, see \cite{DGG}.  We will only need them here in Section $6.2$ when using the description of $\boxtimes$ from \cite{DGG}.
\end{section}

\begin{section}{Representations}

The representation category $Rep(\AW)$  is simply the category of (non-degenerate) $*$-representations of $\AW$ as bounded operators on a Hilbert space.  We begin this section by showing that for a full weight set, $Rep(\AW)$ is equivalent to $Rep(\AC)$, removing the ambiguity of choosing a weight set in our discussions of representation theory.  The resulting representation category has interesting and useful applications.  It comes equipped with a tensor functor making it into a braided monoidal category.  It was shown in \cite{DGG} that the category of finite dimensional representations is (contravariantly) monoidally equivalent to the Drinfeld center, $Z(\ca)$.  In the case where $\Irr$ is finite, the tube algebra $\AC$ is finite dimensional.  Thus understanding its representation theory becomes a computable way of determining the categorical data of the Drinfeld center, and as far as we know is the most commonly used method for understanding $Z(\ca)$ (see \cite{I}, \cite{I2}).

Stefaan Vaes has observed that in general, $Rep(\AC)$ is (contravariantly) equivalent to the category $Z(\text{ind-}\ca)$ introduced and studied by Neshveyev and Yamashita in \cite{NY2}.  The ind-category is basically the ``direct sum completion" of $\ca$, defined by allowing arbitrary direct sums in $\ca$.  It is still a tensor category (though no longer rigid), hence one can apply the usual definitions to obtain a Drinfeld center.  A more detailed discussion of the correspondence between $Rep(\AC)$ and $Z(\text{ind-}\ca)$ will appear in a paper currently in preparation by Popa, Shlyakhtenko and Vaes. 

Another application of the representation theory is to provide natural definitions for approximation and rigidity properties such as amenability, the Haagerup property, and property (T) for rigid $C^{*}$-tensor categories.  One simply generalizes the corresponding definitions for groups given in terms of representation theory, using the trivial representation of $\AC$ (see Lemma 4.12) in place of the trivial representation for groups. 

The main technical difficulty we have to face is a universal bound on the norm of $\AC$ for non-degenerate $*$-representations.  We will see the combinatorial data of the category provides us with a satisfactory universal bound.  With this in hand, we can take arbitrary direct sums of representations, and construct a universal $C^{*}$-completion of $\AC$.  We begin with the formal definitions and immediate consequences.

\begin{defi} A \textit{non-degenerate representation} of an annular algebra $\AW$ is a $*$-homomorphism $\pi: \AW\rightarrow B(H))$ for some Hilbert space $H$ with the property that $\pi(\AW)\xi=0$ for $\xi\in H$ implies $\xi=0$.  We denote the category of non-degenerate representations with bounded intertwiners $Rep(\AW)$
\end{defi}

The non-degeneracy condition is minor.  An arbitrary $*$ representation decomposes as a direct sum of a non-degenerate subspace and a degenerate space, so we can restrict our attention to the non-degenerate piece.  For a non-degenerate representation $(\pi, H)$ and for $k\in \W$, we define $H_{k}:=\pi(p_{k})H \le H$, where $p_{k}$ is the identity projection in $\AW_{k,k}$ described above.  We easily see that $H\cong \oplus_{k\in \W} H_{k}$.  In this way, $\pi$ defines maps $\pi: \AW_{k,m}\rightarrow B(H_k, H_m)$.  Conversely, if we have a sequence of Hilbert spaces $\{H_{k}\}_{k\in \W}$ and a family of maps $\pi_{k,m}: \AW_{k,m}\rightarrow B(H_{k}, H_{m})$ compatible with multiplication and the $*$-structure on $\AW$, we can define a unique representation $\displaystyle \pi:\AW\rightarrow B(H)$ where $H:=\oplus_{k\in \W} H_k$.  It is often convenient to pass between these two pictures.

All representations we consider in this paper are non-degenerate.

\begin{theo} If $\W$ is full, then $Rep(\AW)\cong Rep(\AC)$ as additive categories.
\end{theo}

\begin{proof}  This follows from Proposition $3.5$.

\end{proof}

At this point, since we are mostly interested in representation theory, one might wonder why we bother considering annular algebras with arbitrary weight sets.  The reason is that many categories have a nice description with respect to some particular weight set.  For example, the planar algebras of V. Jones come equipped with a weight set indexed by the natural numbers and given by the number of strings on boundary components.  The resulting annular algebra is called the \textit{affine annular category} of the planar algebra, which is typically viewed as a category (see Definition 3.2) instead of an algebra  \cite{Jo4}, \cite{Jo3}, \cite{DGG}.   With this weight set, the structure of the annular algebra may become transparent via skein theory, and often has a simple description in terms of planar diagrams.  This is clearly illustrated in the $TLJ(\delta)$ categories which we discuss in the next section.  For these categories, the tube algebra at first glance may seem daunting, but applying Theorem 4.2, we can transport the classification of irreducible affine annular representations by Jones and Reznikoff (see \cite{Jo3}) from the planar algebra setting to the tube algebra setting.  This allows us to analyze the tube algebras for these categories, which appears to be quite difficult without these techniques.

In light of the above theorem, however, we lose little generality by focusing our attention on the tube algebra $\AC$.  All of the following results and proofs will be made for $\AC$, but can easily be translated to the more general setting of $\AW$ where $\Lambda$ is full.  The remainder of this section will focus on the demonstrating the existence of a universal $C^{*}$-algebra, denoted $C^{*}(\AC)$, which encodes the representation theory of $\AC$.  This universal $C^{*}$-algebra is directly analogous to and generalizes in some sense the universal $C^{*}$-algebra for groups.  In studying the algebra for groups, the notion of a positive definite function on the group is quite handy, and here we introduce a similar notion.  As we will see in the next section, the true analogy with groups is not with $\AC$ itself, but with the centralizer algebras $\AC_{k,k}$.  The corners are unital $*$-algebras with unit $p_{k}$, and hence have a positive cone.  One of the key points is that to encode the representation theory of the whole tube algebra requires us to extend this positive cone to include positive elements coming from ``outside'' $\AC_{k,k}$ itself.  In particular, we want elements of the form $f^{\#}\cdot f$ with $f\in \AC_{k,m}$ for arbitrary $m$ to be considered positive.  Thus any``local'' notion  of positive definite functions for the centralizer algebras needs to capture this kind of positivity.
 
 \begin{defi} For $k\in \Irr$, a linear functional $\phi: \AC_{k,k}\rightarrow \C$ is called a \textit{weight k annular state} if
 \begin{enumerate}
 \item
 $\phi(p_{k})=1$.
 \item
 $\phi(f^{\#}\cdot f)\ge 0$ for all $f\in \AC_{k,m}$ and $m\in \Irr$.
 \end{enumerate}
 
 We denote the collection of weight $k$ annular states $\Phi_{k}$ (for general $\W$, we denote this set by $\Phi\W_k$)
 \end{defi}

The goal now is to prove a $GNS$ type theorem, which takes a weight $k$ annular state and produces a unique ``k-cyclic'' representation of the whole tube algebra.  If $(\pi, H)\in Rep(\AC)$ and $\xi \in \pi(p_{k})H$ is a unit vector, then the functional $\langle \pi(\ .\ )\xi, \xi \rangle$ restricted to $\AC_{k,k}$ is a weight $k$-annular state.  We will show all weight $k$ annular states are of this form.  The positivity condition in the definition assures that when constructing a Hilbert space, the natural inner product will be positive semidefinite. The  only difficulty generalizing the usual $GNS$ construction is that $\AC$ does not already have a natural norm structure, so we cannot use positivity to assert boundedness of the tube algebra action as in the usual $C^{*}$-algebra GNS construction.  Our situation is analogous to groups, but even there, group elements must have norm $1$, so the action of an arbitrary element in the group algebra is bounded in the $L^{1}$ norm.  
 
 The trick will be to take an annular state and reduce boundedness of the tube algebra action to the situation of a positive linear functional on a finite dimensional $C^{*}$-algebra.  Recall the functional $\omega: \AC\rightarrow \C$ defined right before Definition \ref{tubalg}.
 
 \begin{lemm} Let $\displaystyle y\in \AC^{t}_{m,n}$ for $t\in \Irr$. Then, $\phi(x^{\#}\cdot y^{\#}\cdot y\cdot x)\le d(X_t)^{2}\omega(y \cdot y^{\#})\phi(x^{\#}\cdot x)$ for all $\phi\in \Phi_{k}$ and $x\in \AC_{k, m}$. 
 \end{lemm}
 
 \begin{proof}  Let $\displaystyle x=\sum_{j\in \Irr} x_{j}\in \AC_{k,m}$ where this sum is finite and each $x_{j}\in \AC^{j}_{k,m}$.   Then define the object $\alpha:=\oplus X_{j}$, where the $j$ here are the same $j$ in the description of $x$. Then viewing $x\in Mor(\alpha \otimes X_{k},\ X_{m}\otimes \alpha)$, we have $\Psi^{\alpha}(x)=x$.  Notice that since each $X_{j}$ has a chosen dual (the object chosen to represent the equivalence class of $\overline{X}_{j}$), this distinguishes a conjugate object $\overline{\alpha}$.  Let $\phi\in \Phi_{k}$.  We define a linear functional $\tilde{\phi}_ {x}$ on the finite dimensional $C^{*}$-algebra $End(X_t \otimes X_m \otimes \overline{X}_{t})$ by 
 
 $$\tilde{\phi}_{x}(\ .\ ):= \phi\circ \Psi^{\overline{t \alpha}\otimes t \alpha}\left(\grc{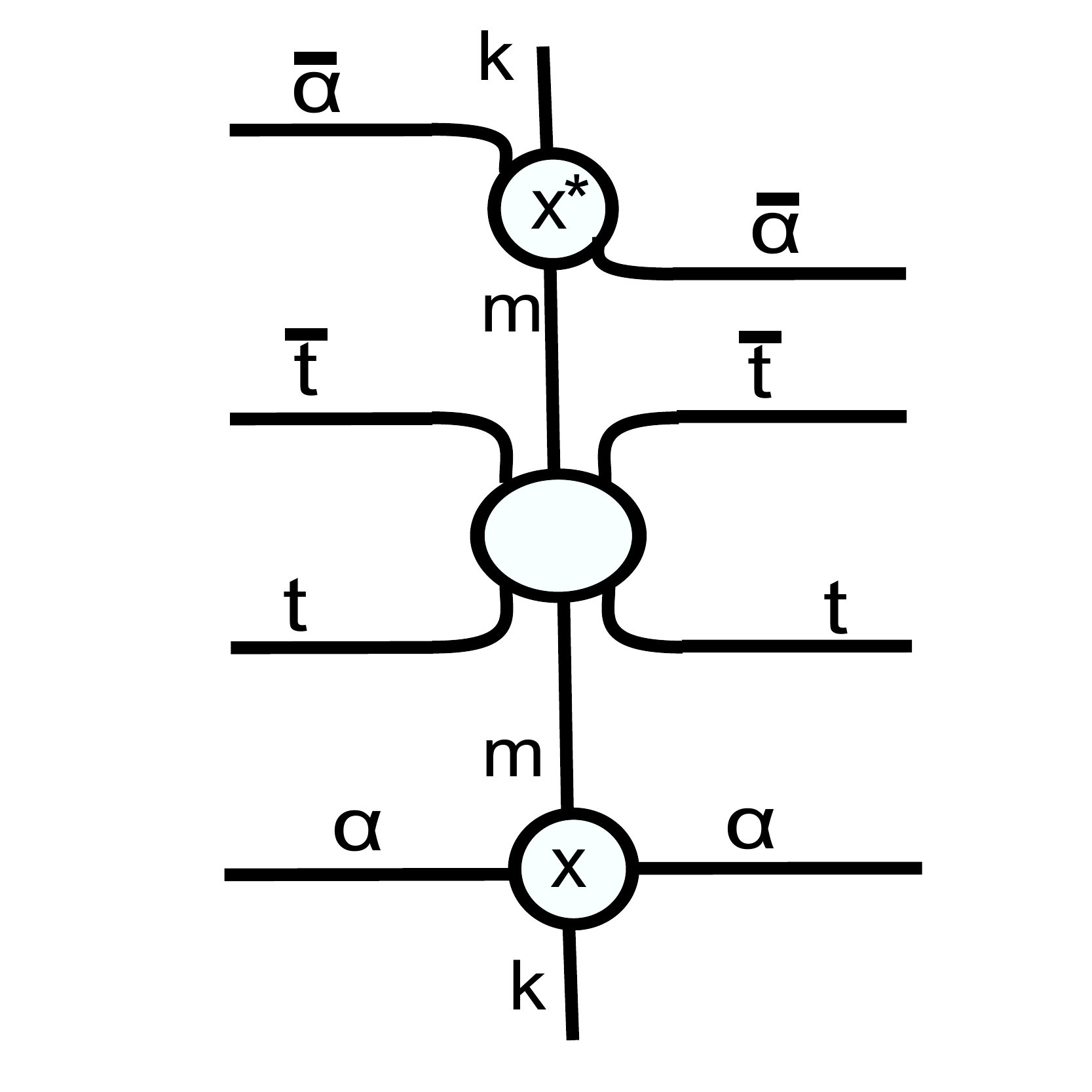}\right).$$

To evaluate $\tilde{\phi}_{x}(\ . \ )$ on a morphism $f\in End(X_t \otimes X_m \otimes \overline{X}_{t})$, we insert $f$ into the unlabeled disc in the above diagram and evaluate.   We claim that $\tilde{\phi}_{x}$ is a positive linear functional on the finite dimensional $C^{*}$-algebra $End(X_t \otimes X_m \otimes \overline{X}_{t})$.  For positive $w$ in this algebra, we see that 

$$\tilde{\phi}_{x}(w)=\sum_{j\in \Irr}\  \sum_{V \in onb({j}, tm\overline{t})} \tilde{\phi}_{x}(w^{\frac{1}{2}}VV^{*}w^{\frac{1}{2}})$$
$$= \!\!\!\!\!\! \sum_{j\in \Irr} \sum_{V\in onb({j}, tm\overline{t})} \!\!\!\!\!\! \phi \left( \! \left[\Psi^{t} \left( (V^{*} w^{\frac{1}{2}} \otimes 1_{t} ) (1_{t}\otimes 1_{m}\otimes R_{t}) \right) \! \cdot \! \Psi^{\alpha}(x)\right]^{\#} \cdot \left[\Psi^{t} \left((V^{*} w^{\frac{1}{2}} \otimes 1_{t} ) (1_{t}\otimes 1_{m}\otimes R_{t}) \right) \! \cdot \! \Psi^{\alpha}(x)\right] \! \right)$$
which is non-negative by definition of annular state.  Then by positivity of $\tilde{\phi}_{x}$, $$\tilde{\phi}_{x}(w)\le \| w\| \tilde{\phi}_{x}(1_{tm\overline{t}}) = \| w \| \phi \circ \Psi^{\overline{t \alpha} \otimes t \alpha} \left( \grd{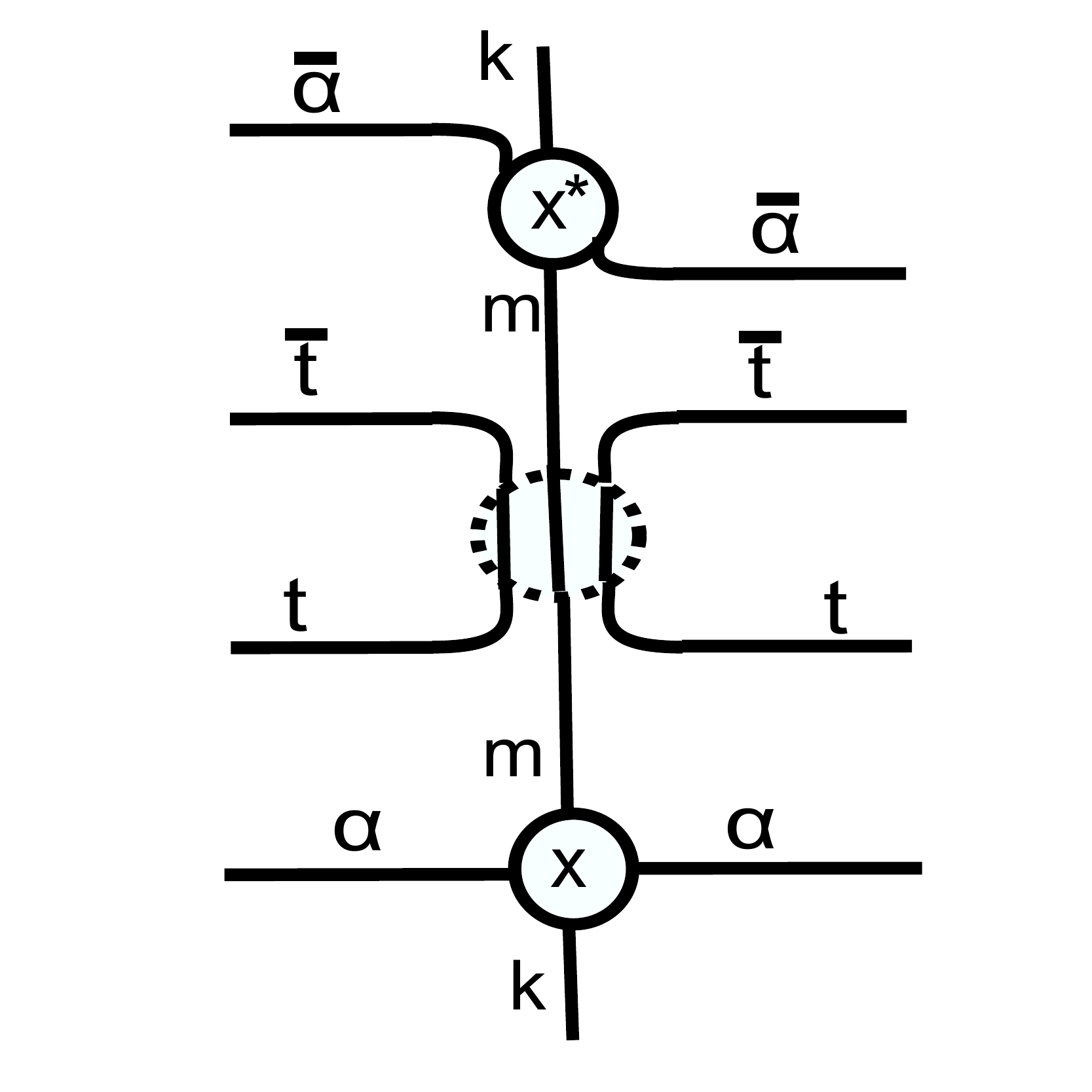} \right) = \|w\| d(X_t) \phi(x^{\#}\cdot x).$$  In the last equality we use the ``annular relation'' describing the kernel of $\Psi$ to pull the side $t$-cap from the left around to the right, yielding a closed $t$-circle hence a factor of $d(X_t)$.  Now for $y\in (X_t\otimes X_m, X_n\otimes X_t)$, consider the morphism $\tilde{y}:=(1_{n}\otimes \overline{R}^{*}_{t}) ( y \otimes 1_{\overline{t}}) \in Mor(X_t \otimes X_m \otimes \overline{X}_t ,\ X_n)$.  Then $\tilde{y}^{*}\tilde{y}\in End(X_t \otimes X_m \otimes \overline{X}_{t})$, and we see that
 
$$\phi(x^{\#}\cdot y^{\#} \cdot y\cdot x)=\tilde{\phi}_{x}(\tilde{y}^{*}\tilde{y})\le d(X_t )\|\tilde{y}^* \tilde{y}\| \phi(x^{\#}\cdot x) = d(X_t )\|\tilde{y} \tilde{y}^*\| \phi(x^{\#}\cdot x) = d(X_t )^{2}\omega(y \cdot y^\#)\phi(x^{\#}\cdot x)$$
 
For the last inequality, note that $\tilde{y} \tilde{y}^*$ is a scalar times $1_n$ ($X_n$ being simple), so to find that scalar we  apply the categorical trace and compare with $\omega(y \cdot y^\#)$, yielding the required result (we recommend the reader draw a picture here).
\end{proof}

We note the proof of this lemma has obvious modifications for general annular algebras associated to full weight sets.

Now,  if $\phi\in \Phi_{k}$, we define a sesquilinear form on the vector space $\hat{H}_{\phi}:=\bigoplus_{m\in \Irr}\AC_{k,m}$ by $\langle x , y \rangle_{\phi}:=\phi(y^{\#}\cdot x)$.  By definition this form is positive semi-definite.   Furthermore, this vector space has a natural action of $\AC$ by left multiplication.  We construct a Hilbert space by taking the quotient by the kernel of this form and completing, which we denote $H_{\phi}$.  Recall an arbitrary $y\in \AC$ can be written $\displaystyle y=\sum_{m,n, j\in \Irr} y^{j}_{m,n}$ where this sum is finite and each $y^{j}_{m,n}\in \AC^{j}_{m,n}$.  By the previous lemma, each $y^{j}_{m,n}$ preserves the kernel of the form and is bounded, therefore we have $\pi_{\phi}(y^{j}_{m,n})\in B(H_{\phi})$.  Extending linearly, $\pi_{\phi}: \AC\rightarrow B(H)$ is a (non-degenerate) $*$-representation of the tube algebra.
  
\begin{cor} A functional $\phi: \AC_{k,k}\rightarrow \mathbb{C}$ is in $\Phi_{k}$ if and only if there exists a non-degenerate $*$-representation $(\pi, H)$ of $\AC$, and a unit vector in $\xi\in \pi(p_{k})H$, such that $\phi(x)=\langle \pi(x)\xi, \xi\rangle$.  Furthermore the sub-representation on $H_{\xi}:=[\pi(\AC)\xi]\subseteq H$ is unitarily equivalent to the representation $H_{\phi}$ described above.
\end{cor}

Continuing the analogy with groups, we notice that Lemma $4.4$ provides us with a bound similar to the $L^{1}$-norm for groups.  Since an arbitrary element in the tube algebra will have its norm bounded by the constant in Lemma $4.4$ in any representation, we can take arbitrary direct sums of representations.  This allows us to define a universal representation, and a corresponding universal $C^{*}$-algebra.

\begin{defi} \ \ 

\begin{enumerate}
\item 
The \textit{universal representation} of the tube algebra is given by $\displaystyle (\pi_{u}, H_{u}):=\bigoplus_{k\in \Irr,\ \phi\in \Phi_{k}}(\pi_{\phi}, H_{\phi})$. 
\item
The \textit{universal norm} on $\AC$ is given by $\| x\|_{u}:=\|\pi_{u}(x)\|$.
\item
The \textit{universal $C^{*}$-algebra} is the completion $C^{*}(\AC):=\overline{\pi_{u}(\AC)}^{{\| \|}}$.  
\end{enumerate}
\end{defi}

Note that non-degenerate $*$-representations of $\AC$ are in 1-1 correspondence with non-degenerate, bounded $*$-representations of $C^{*}(\AC)$.  Note that the universal norm is finite (so that such an infinite direct sum exists), follows from Lemma $4.4$.   We record the consequences of Lemma 4.4 for the universal norm in the following corollary:

\begin{cor} Let $\displaystyle \sum_{j,k,m\in \Irr} x^{j}_{k,m}\in \AC$.  Then $\displaystyle 0< ||x||_{u}\le \sum_{j,m,n\in \Irr} d(X_{j}) \omega(x^{j}_{m,n}\cdot (x^{j}_{m,n})^{\#} )^{\frac{1}{2}}$.
\end{cor}

\begin{proof}
The bound on the right follows from Lemma $4.4$.  The strict positivity of the universal norm follows from the fact that $\omega$ is a positive definite functional on $\AC$ and $\omega |_{\AC_{k,k}}$ is a weight $k$ annular state for all $k\in \Irr$.
\end{proof}

We now turn our attention back to the centralizer algebras $\AC_{k,k}$.  We want to study the representation theory of these unital $*$-algebras, under the restriction that the representations must ``come from'' a tube algebra representation.  The reason for studying these representations is that while we are interested in the whole algebra $\AC$ and its representation theory, often we are able to understand the centralizer algebras and their admissible representations with much greater ease.  The following proposition is an easy corollary of the GNS construction:

\begin{cor} Let $k\in \Irr$, and let $(\pi_{k}, H_{k})$ be a non-degenerate $*$-representation of $\AC_{k,k}$.  The following are equivalent:
\begin{enumerate}
\item
Every vector state in $(\pi_{k}, H_{k})$ is weight $k$ annular state.
\item
$||\pi_{k}(x)||\le || x||_{u}$ for all $x\in \AC_{k,k}$.
\item
$(\pi_{k}, H_{k})$ extends to a continuous representation of the unital $C^{*}$-algebra $p_{k} C^{*}(\AC) p_{k}$.
\item
There exists a representation $(\pi, H)$ of $\AC$ such that $(\pi, H)|_{\AC_{k,k}}$ is unitarily equivalent $(\pi_{k}, H_{k})$.
\end{enumerate}
 \end{cor}
 
 \begin{proof} $(1)$ implies $(2)$ implies $(3)$ follows from the above discussion.  For $(3)$ implies $(4)$, we construct the representation $(\pi, H)$ in a manner analogous to the $GNS$ construction.  We see that $p_{m}\AC_{k,m}p_{k}$ provides a Hilbert $C^{*}$-bimodule for the corner algebras $p_{m}C^{*}(\AC)p_{m}$ and $p_{k}C^{*}(\AC)p_{k}$ for all $m\in \Irr$ with the obvious left and right inner products.  By standard Hilbert $C^{*}$-bimodule theory, we have an induced representation $(\pi_{m}, H_{m})$ of $p_{m} C^{*}(\AC) p_{m}$, where $H_{m}$ is the Hilbert space completion of $p_{m}\AC p_{k}\otimes H_{k}$ with respect to the induced inner product $\langle f\otimes \xi, g\otimes \eta\rangle_{m}:=\langle \pi_{k}(g^{\#}\cdot f)\xi, \eta \rangle_{k}$.  By bimodule theory, $H:=\bigoplus_{m\in \Irr} H_m$ carries a $*$-representation, $\pi$, of $\AC$. $(4)$ implies $(1)$ follows from the $GNS$ reconstruction result.
 
 \end{proof}
 
\begin{defi} A representation satisfying the equivalent conditions of the previous corollary is called a \textit{weight $k$ admissible representation}.  
\end{defi}

Admissible representations can be seen simply as representations of the centralizer algebras which are restrictions of representations of the whole tube algebra.  Alternatively, they are representations of the corner algebras which induce representations of the whole tube algebra.  Understanding admissible representations for all weights allows us to understand representations of the whole tube algebra.  Since the norm in weight $k$ admissible representations is bounded by the universal norm for $\AC_{k,k}$, one can construct a universal $C^{*}$-algebra completion $C^{*}(\AC_{k,k})$.  From the above proposition, it is clear that $C^{*}(\AC_{k,k})\cong p_{k} C^{*}(\AC) p_k$.  

We remark that proposition $3.5$ implies $C^{*}(\AC)\otimes K\cong C^{*}(\AW)\otimes K$ where $K$ is the $C^{*}$-algebra of compact operators on a seperable Hilbert space.

We end this section with two canonical examples of a non-degenerate $*$-representation of $\AC$ that always exists for all categories.  The first, the so-called left regular representation, is analogous to the left regular representation for groups (though not strictly analogous as we shall see!).  The second, the so-called ``trivial representation'' is rather non-trivial, but serves a similar role to the trivial representation in group theory for approximation and rigidity properties.

\begin{defi}
The \textit{left regular representation} has Hilbert space $L^{2}(\AC, \omega)$, and action $\pi_{\omega}$ given by left multiplication.  
\end{defi}

That the action here is bounded follows from the fact that $\omega|_{\AC_{k,k}}$ is an annular weight $k$ state, hence every vector state in $\pi_{\omega}(p_{k})L^{2}(\AC, \omega)$ is in $\Phi_{k}$.  Applying Lemma 4.4 yields the boundedness.

\medskip

Recall in the previous section that we had a canonical isomorphism $\C[\Irr]\cong \AC_{0,0}$.

\begin{lemm}  The one dimensional representation of $\AC_{0,0}$ defined by the character $1_{\ca}([X])=d(X)$ for all $X\in \Irr$, is a weight $0$ annular state.
\end{lemm}

\begin{proof} Let $\delta_{\alpha}$ denote the map canonically identifying $Mor(\alpha \otimes id,\ id\otimes \alpha)$ with $Mor(\alpha, \alpha)$ for all objects $\alpha$ for all objects $\alpha$.  Since $\AC^{k}_{0,0}:=Mor(X_{k}\otimes id,\ id\otimes X_{k})$, we  have a map $\displaystyle \delta:=\bigoplus_{k\in \Irr} \delta_{k}: \AC_{0,0}\rightarrow \bigoplus_{k\in \Irr} Mor(X_{k}, X_{k})$. Now we can see $1_{\ca}(x)=Tr(\delta(x))$, where $Tr:=\oplus_{k\in \Irr} Tr_{k}$.   Furthermore, one can check that for $x\in Mor(\alpha \otimes id,\ id\otimes \alpha)$, $1_{\ca}(\Psi^{\alpha}(x))=Tr_{\alpha}(\delta_{\alpha}(x))$.

For $\displaystyle x=\sum_{j\in \Irr} x^{j}_{0,m}\in \AC_{0,m}$, setting $\alpha:=\oplus X_{j}$ where the $j$ appear in the sum for $x$, we have $1_{\ca}(x^{\#}\cdot x)=1_{\ca}(\Psi^{\overline{\alpha}\alpha}(x^{\#}\cdot x))=Tr(\delta_{\overline{\alpha}\alpha}(x^{\#}\cdot x))=0$ for all $m\ne 0$ in $\Irr$ by sphericality of the trace, since $Mor(id,\ X_{m})=\{0\}$ for $m\ne 0$.  Therefore it suffices to check $1_{\ca}(x^{\#}\cdot x)\ge 0$ for $x\in \AC_{0,0}$, which follows since $1_{\ca}$ is a $*$-homomorphism.

\end{proof}

We note that for $k\in \Irr$, $k\ne 0$, $\pi_{1_{\ca}}(p_{k})=0$.  Thus all ``higher weight'' spaces in the trivial representation are $0$, so that in fact $1_{\ca}$ is a character on $\AC$

\begin{defi} The \textit{trivial representation} of $\AC$ is the one dimensional representation $1_{\ca}$ of $\AC$.

\end{defi}

The trivial representation will play a similar role in our representation theory to the trivial representation in the theory of groups.  

\end{section}

\begin{section}{Examples}

We will now analyze the tube algebra and, in particular, the centralizer algebras for two classes of categories: $G$-graded vector spaces for a discrete group $G$ , and the Temperly-Lieb-Jones categories.   In the $G$-$Vec$ case, we see the centralizer algebras are exactly the groups algebras of centralizer subgroups of elements.

\begin{subsection}{$G$-Vec}

Let $G$ be a discrete group and let $\ca$ be the category of $G$-graded vector spaces with trivial associator.  The tube algebra of this example is known, and is one of the earliest examples of a tube algebra, though we were unable to find the first description of it. The tube algebra in this case is essentially the Drinfeld double of the Hopf algebra $\C[G]$, which was one of the motivating examples in the definition of the Drinfeld center.  This example is typically presented in the case of finite groups, while here we consider discrete groups in general.  

Simple objects in $\ca$ are one-dimensional vector spaces indexed by elements of a group, and we identify $\Irr$ with the group $G$.  The tensor product corresponds to group multiplication, and duality corresponds to inverses of group elements.  To be clear, we are actually using a ``strictified'' version of the category, where $X\otimes Y=XY$ for $X,Y\in G$, with equality instead of isomorphism of objects.

 For $X,Y,Z\in G$,  by Frobenius reciprocity $\AC^{Z}_{X,Y}\cong Mor(X, \overline{Z}YZ)$ which is $1$ dimensional if $X=Z^{-1}YZ$ as group elements, and $0$ otherwise. Thus in the tube category language, there is a non-zero hom between $X,Y$ iff $X$ is conjugate to $Y$.  If we set $Conj(G):=\{\text{conjugacy classes of}\ G\}$, then we have a first decomposition $\displaystyle \AC \cong \bigoplus_{\Gamma\in Conj(G)} \A_{\Gamma}$, where $\A_{\Gamma}:=\bigoplus_{X,Y\in \Gamma} \AC_{X,Y}$.

 \ \ Thus it suffices to determine the structure of $\A_{\Gamma}$ for each conjugacy class $\Gamma$.   For $X\in G$, $\AC_{X,X}:=\bigoplus_{Y\in Z_{G}(X)}\AC^{Y}_{X,X}$, where $Z_{G}(X)$ is the centralizer subgroup of $X$ in $G$.  Since each $\AC^{Y}_{X,X}=Mor(YX,XY)$ is non-zero if and only if $XY=YX$, we can identify this space with $Mor(YX, YX)$ which in turn is isomorphic to $\C$.  Thus we have a natural vector space isomorphism $\alpha: \AC_{X,X}\cong \C[Z_{G}(X)]$.  Furthermore, it is easy to check that this is a $*$-algebra isomorphism. More specifically for $Y\in Z_{G}(X)$, we can choose $f^{Y}_{X}\in \AC^{Y}_{X,X}=Mor(YX, YX)$ to be the identity in the later morphism space.  Then we have from the tube algebra multiplication $f^{Y}_{X} \cdot f^{Z}_{X}=f^{YZ}_{X}\in \AC_{X,X}$, and $\#$ corresponds to inverses.  Now, for each $X,Y\in \Gamma$, $Z_{G}(X)\cong Z_{G}(Y)$.  In fact these are conjugate by any group element that conjugates $Y$ to $X$.  The number of possible conjugators from $X$ to $Y$ is $|Z_{G}(X)|$.  It is now easy to see that $\A_{\Gamma}\cong\C[Z_{G(X)}]\otimes B_{0,0}(\ell^{2}(\Gamma))$, where $ B_{0,0}(\ell^{2}(\Gamma))$ is the algebra of finite rank operators on the Hilbert space $\ell^{2}(\Gamma)$.  The diagonal copies of $Z_{G}(X)$ are the  $\AC_{X,X}$, and the matrix unit copies are given by $\AC_{X,Y}$.

\ \ We have the following claim:  Let $X\in \Irr \cong G$, and let $Z_{G}(X)$ be the centralizer subgroup of $X$ in $G$.  Then if $(\pi, H)$ is a unitary representation of $Z_{G}(X)$, then $(\pi, H)$ extends to a representation of $\A_{\Gamma}$, where $\Gamma$ is the conjugacy class of $X$.  To see this we simply note that since $\A_{\Gamma}\cong \C[ Z_{G}(X)]\otimes B_{0,0}(\ell^{2}(\Gamma))$, we can define the Hilbert space $H_{\Gamma}:=H\otimes \ell^{2}(\Gamma)$, with the obvious action.  It is clear that this is a $*$ representation by bounded operators of $\A_{\Gamma}$.  Therefore $$C^{*}_{u}(\AC_{X,X})\cong C^{*}_{u}(Z_{G}(X))$$

\end{subsection}

\begin{subsection}{TLJ Categories}

The Temperley-Lieb-Jones categories $TLJ(\delta)$ for $\delta\ge 2$ are equivalent to the categories $Rep(SU_{-q}(2))$, where $\delta=q+q^{-1}$ for $q$ a positive real number.  They provide a fundamental class of infinite depth rigid $C^{*}$-tensor categories.   They also provide examples of categories that have a nice planar algebra description and a nice categorical description simultaneously.  To describe them, fix a positive real number $\delta\ge 2$.  Then there is a unique $q\in \R$ such that $q+q^{-1}=\delta$.  We can then define for $n\in \N$,  $[n]_{q}=\frac{q^{n}-q^{-n}}{q-q^{-1}}$ if $q\ne 1$, and $[n]_{1}=n$.  

The rigid $C^{*}$-tensor category $TLJ(\delta)$ consists of:
\begin{enumerate}
\item
Self dual simple objects indexed by natural numbers, with $0$ indexing the identity.
\item
$d(k)=[k+1]_{q}$
\item
$k\otimes m\cong (k+m)\oplus (k+m-2)\oplus\dots\oplus |k-m|$
\end{enumerate}

For the rest of this section, we use $[n]$ to denote $[n]_{q}$, assuming $q$ is fixed.  The above properties are merely a summary of some relevant categorical data. These categories have much more structure than this, for example there are complicated $6$-j symbols, and these categories naturally have a braiding (non-unitary unless $q=1$).  These categories also can be realized as the projection categories of particularly nice planar algebras.

Define the unoriented, unshaded planar algebra $TL(\delta)$ as follows: 
\begin{enumerate}
\item
$P_{0}\cong \C$
\item
$ P_{2n+1}=0$
\item
$P_{2n}:=$ Linear span of disks with $2n$ boundary points with strings connecting boundary points
\item
 strings do not cross
\item
 All boundary points are connected to some other boundary point with a string
\item
Closed circles multiply the diagram by a factor of $\delta$

\end{enumerate}

We note that in our generic case $\delta\ge2$, this is a spherical $C^{*}$-planar algebra (see \cite{BHP}, \cite{Jo4} for definitions of spherical $C^{*}$-planar algebras).  We have $dim(P_{2n})=\frac{1}{n+1}\binom{2n}{n}$.  We remark that this is perhaps the most important example of a planar algebra since it appears in some form as a sub-algebra of an arbitrary planar algebra. It is usually presented as a shaded planar algebra in the subfactor context, and there exists many detailed expositions, see \cite{Jo2},\cite{Jo4}.  We can realize the category described above as $TLJ(\delta)=Proj(TL(\delta))$ (see \cite{BHP} for definition of the projection category of a planar algebra).  The object $k$ in $TLJ(\delta)$ corresponds to the $k^{th}$ Jones-Wenzl idempotent in the planar algebra $TL(\delta)$, denoted $f_{k}$.  These projections satisfy the property that applying a cap or cup to the top or bottom of $f_{k}$ results in $0$, called uncapability.  $f_{k}$ is a minimal projection in $TL_{k,k}$ and can be defined by an inductive formula, see \cite{Mor} or \cite{Jo2} for details.  

\ \ The affine annular representations of this planar algebra have been studied in detail by Jones, Jones-Reznikoff, and Reznikoff (see \cite{Jo3}, \cite{Jo4}, and \cite{R} respectively).  We will make use of these results to analyze the universal $C^{*}$-algebra structure on the centralizer algebras of the tube algebra of this category.  The beginning of this section can be deduced in its entirety from the references listed above.  We include these results here for the purpose of self-containment, and due to the slight differences in our setting.  We remark here that we use the categorical picture (Definition 3.2) for $ATL$ to fit with the perspective of Jones and Jones-Reznikoff.

As discussed in section $3$, a planar algebra $\Pl$ naturally provides an annular algebra $\AP$.  For details on this see \cite{Jo3}, \cite{Jo2}, \cite{DGG}.  The affine annular category $ATL$ is easy to describe. The weights will simply be natural numbers, and they will signify the number of strings on the boundaries of disks.  The object in $Proj(TL(\delta))$ corresponding to $k\in \N$ is $1_{k}\in TL_{k,k}$.  Then $ATL_{k,m}$ will consist of all $TL$ diagrams in an annulus with $k$ boundary points on the internal circle and $m$ on the external circle.  This means there are $\frac{k+m}{2}$ non-intersecting strings in the annulus, and each string touches precisely one boundary point (on either the inner or outer disk).  We consider these diagrams only up to affine annular isotopy.  That the set of affine annular pictures described here (isotopy classes of non-intersecting string diagrams) is really a basis for the annular category of the planar algebra follows from the analysis of \cite{DGG} and the fact that $TL(\delta)$ for $\delta\ge 2$ has no local skein relations except removing closed circles.  Composition is the obvious one, and homologicaly trivial circles in the annulus multiply the diagram by a factor of $\delta$.  For more details on this annular category in particular see \cite{Jo3}.

We consider here a subcategory of $Rep(ATL)$ consisting of all locally finite representations.  By this we mean the set of Hilbert representations  of $ATL$, $(\pi, V_{k})$ such that each $V_{k}$ is a finite dimensional Hilbert space, and $\pi: ATL_{k,m}\rightarrow B(V_{k}, V_{m})$ is a $*$-homomorphism.  This category is closed under finite direct sums.  In the literature, Hilbert representations of $A\Pl$ are called Hilbert $\Pl$ modules, and so we use these terms interchangeably in the planar algebra setting.

\begin{defi} A lowest weight $k$ Hilbert $TL$-module is a representation $(\pi, V_{m})$ such that $V_{m}=0$ for all $m< k$.  
\end{defi}

Irreducible representations of $ATL$ are representations which are irreducible as representations of the corresponding annular algebra.  It is straightforward to check that this implies each $V_k$ is irreducible as a representation of $ATL_{k,k}$.  Following the proof in \cite{Jo4}, one can show that every locally finite Hilbert $TL$-module is isomorphic to the direct sum of irreducible lowest weight $k$ modules.  It then becomes our task to classify and construct these.  

\ \ To do so we start by noting that $ATL_{0,0}$ is isomorphic to the fusion algebra $\C[\text{Irr}(TLJ(\delta)]$, which is abelian.  Thus an irreducible lowest weight $0$ module will be a $1$ dimensional representation of the fusion rules.  Let $v_{0}$ be a non-zero vector in the one dimensional space normalized so that $\langle v_{0}, v_{0}\rangle=1$.  We notice the identity object ($f_{0}$) must go to the identity and we may identify $\pi(f_{k})$ with some number (its eigenvalue on $v_{0}$).  But from the fusion rules, all these numbers are determined by $\pi(f_{1})$.  Since $f_{1}$ in $ATL_{0,0}$ is self dual and this must be a $*$-representation, we see that  $\pi(f_{1})$ (hence $\pi(f_{k})$ for all $k$) must be a real number.  Furthermore, by the bounds on the universal norm for the weight $0$ case (Corollary 4.7), we must have $|\pi(f_{1})|\le \delta$.  Let $t:=\pi(f_{1})\in [-\delta, \delta]$.  Then this parameter determines $\pi$ completely.  We still must see which of these extend to Hilbert $TL$-modules, but we will see that all of them will.

\ \ Now, consider $k>0$.  Let $ATL^{<k}_{k,k}$ be the ideal in $ATL_{k,k}$ spanned by diagrams with less than $k$ through strings.  We see that in a lowest weight $k$ representation, this ideal must act by $0$.  An irreducible lowest weight $k$ representation will then neccessarily be an irreducible representation of the algebra $ATL_{k,k}/ATL^{<k}_{k,k}$.  We define the element $\rho_{k}\in ATL_{k,k}$, known informally as ``rotation by one'', with the picture  \grc{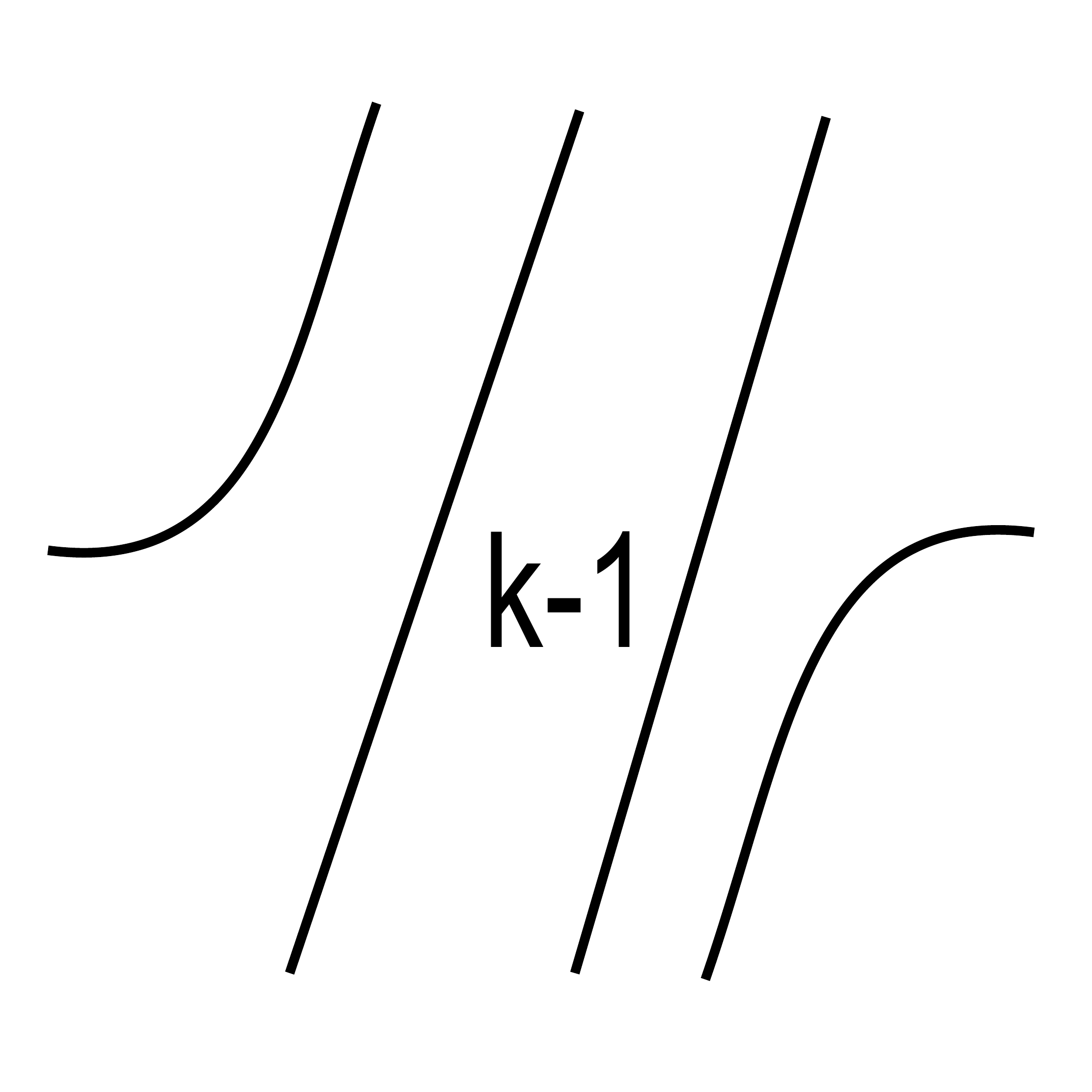}.  Here we use the graphical calculus conventions of annular algebras.  We see that this element is invertible in $ATL_{k,k}$, and we call its inverse $\rho^{-1}_{k}=\rho^{*}_{k}$ the ``left rotation by one".  The powers of $\rho_{k}$ form a subgroup of the algebra $ATL_{k,k}$ isomorphic to $\mathbb{Z}$, hence $ATL_{k,k}/ATL_{k,k}^{<k}\cong \C[\mathbb{Z}]$, which is abelian.  Therefore an irreducible lowest weight $k$ $*$-representation will be an irreducible unitary representation of $\mathbb{Z}$, hence determined by some $\omega\in S^{1}$.  

\ \ We have now found all candidates for irreducible lowest weight $m$ representations of $ATL$ for all $m$.  The question that remains is which of these representations of the fusion algebra and $\mathbb{Z}$ extend to a representation of the entire annular category, i.e. have a canonical extension.  Since all the spaces are finite dimensional (as we shall see), the annular actions are bounded, hence it suffices to demonstrate that the inner products of the canonical extension are positive semi-definite.  We follow the method prescribed in Corollary $4.8$, namely if we have a representation of $ATL_{m,m}/ATL_{m,m}^{< m}$ (or $ATL_{0,0}$) representation determined by the parameter $\alpha\in S^{1}$ (or in $[-\delta, \delta]$) on the one dimensional vector space $V^{\alpha}_{m}$, define $\hat{V}^{\alpha}_{n}:=ATL_{m, n}\otimes _{ATL_{m,m}} V^{\alpha}_{m}$.   If we let $s_{\alpha}\in V^{\alpha}_{m}$ be normalized, we can represent simple tensors in the vector space $\hat{V}^{\alpha}_{n}$ by

$f\otimes s_{\alpha}:=\grc{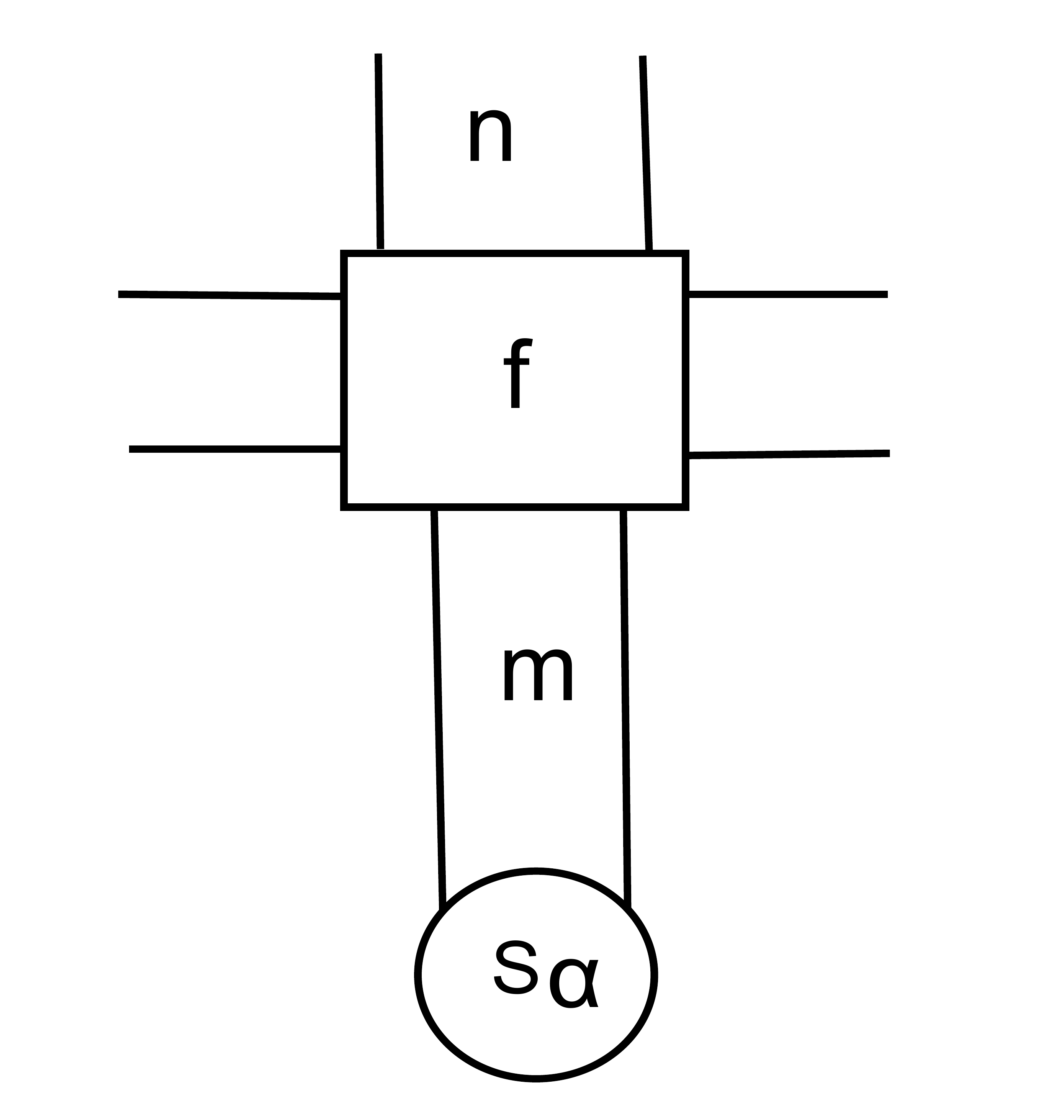}$

Connecting the bottom $m$ strings to the rotation eigenvector signifies that we are taking a relative tensor product over $ATL_{m,m}$.  Now, we can easily see that  $dim(f_{k}ATL_{m,k}\otimes _{ATL_{m,m}}V^{\alpha}_{m})$ is at most one.  To see this, we note that all the strings emanating from $s_{\alpha}$ must enter the $f_{k}$ consecutively, since apply a cap to $f_{k}$ results in $0$.  The remaining $k-m$ strings coming from $f_k$ that are not attached to $s_{\alpha}$ must be connected to each other somehow, but by uncapability of $f_k$, they must be connected ``around the bottom of the annulus''.  If $m-k$ is even, there is precisely one way to do this, and if $m-k$ is odd this is impossible.  In particular, $f_{k}\hat{V}^{\alpha}_{k}$ is spanned by the vector $g^{\alpha}_{m,k}:=\grc{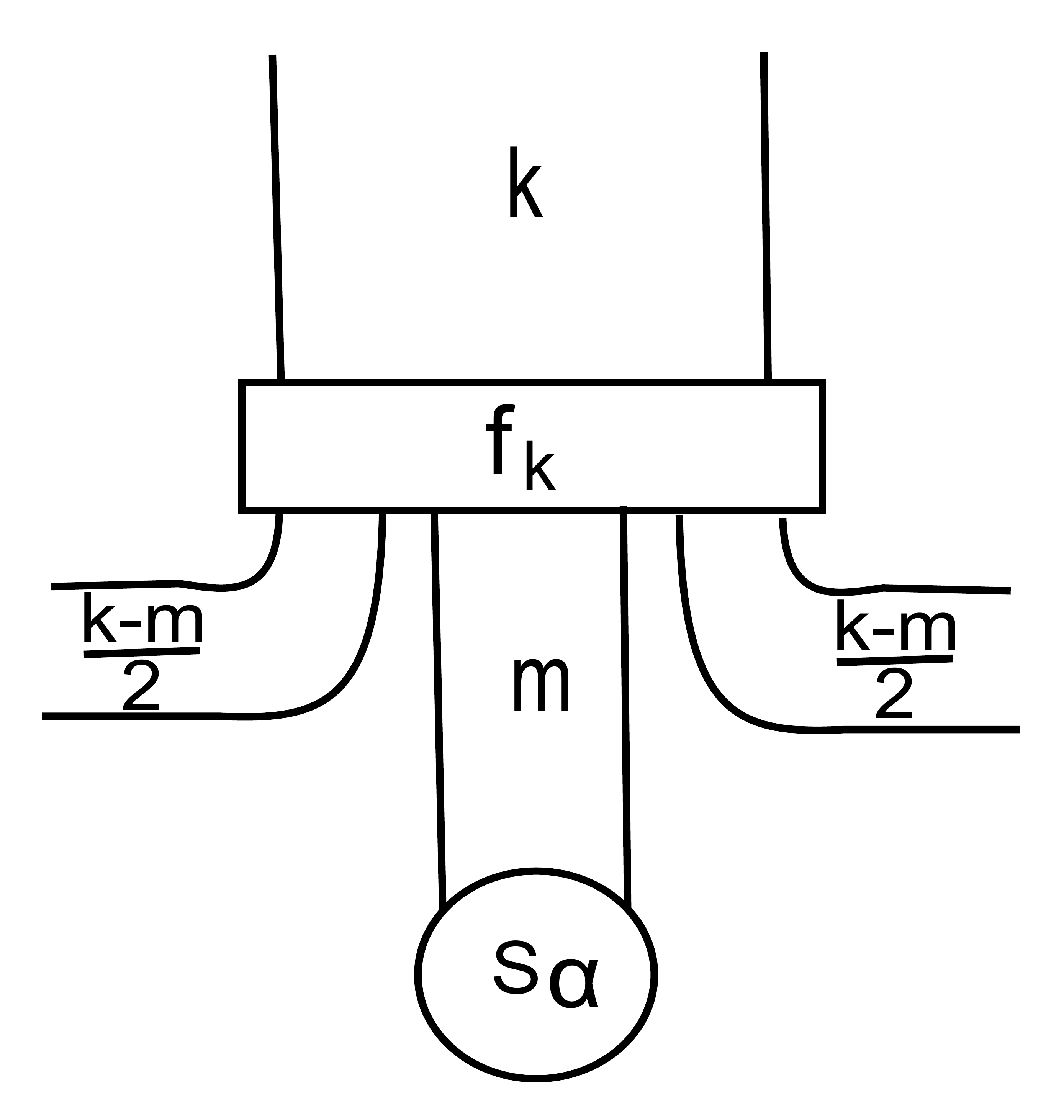}$.  We note that $g^{\alpha}_{m,k}=0$ for $k<m$.

To understand $\hat{V}^{\alpha}_{n}$, for each Jones-Wenzl idempotent $f_{n-2j}\prec n$, let $(f_{n-2j}, n)$ denote the linear space of planar algebra elements  $x\in P_{n-2j, n}$ such that $xf_{n-2j}=x$, for $0\le j\le \lfloor\frac{n}{2}\rfloor$.    $(f_{n-2j}, n)$ is precisely the space of morphisms in the projection category of $TL$ from $f_{n-2j}$ to $1_{n}$ (see \cite{BHP}).  It is clear that $\hat{V}^{\alpha}_{n}\cong \bigoplus_{0\le j\le \lfloor\frac{n}{2}\rfloor} (f_{n-2j}, n)\otimes g^{\alpha}_{m,n-2j}$.

With this nice decomposition, we want to see how the canonical inner product behaves.  First, we will do some diagrammatics that will allow us to clearly see the canonical inner product is positive semi-definite.  We closely follow the work of \cite{Jo3}.  Let $\alpha$ be the parameter of a lowest weight $m$ representation.  We define the numbers $B^{k}_{m,l}(\alpha)$ by the following:

$\grc{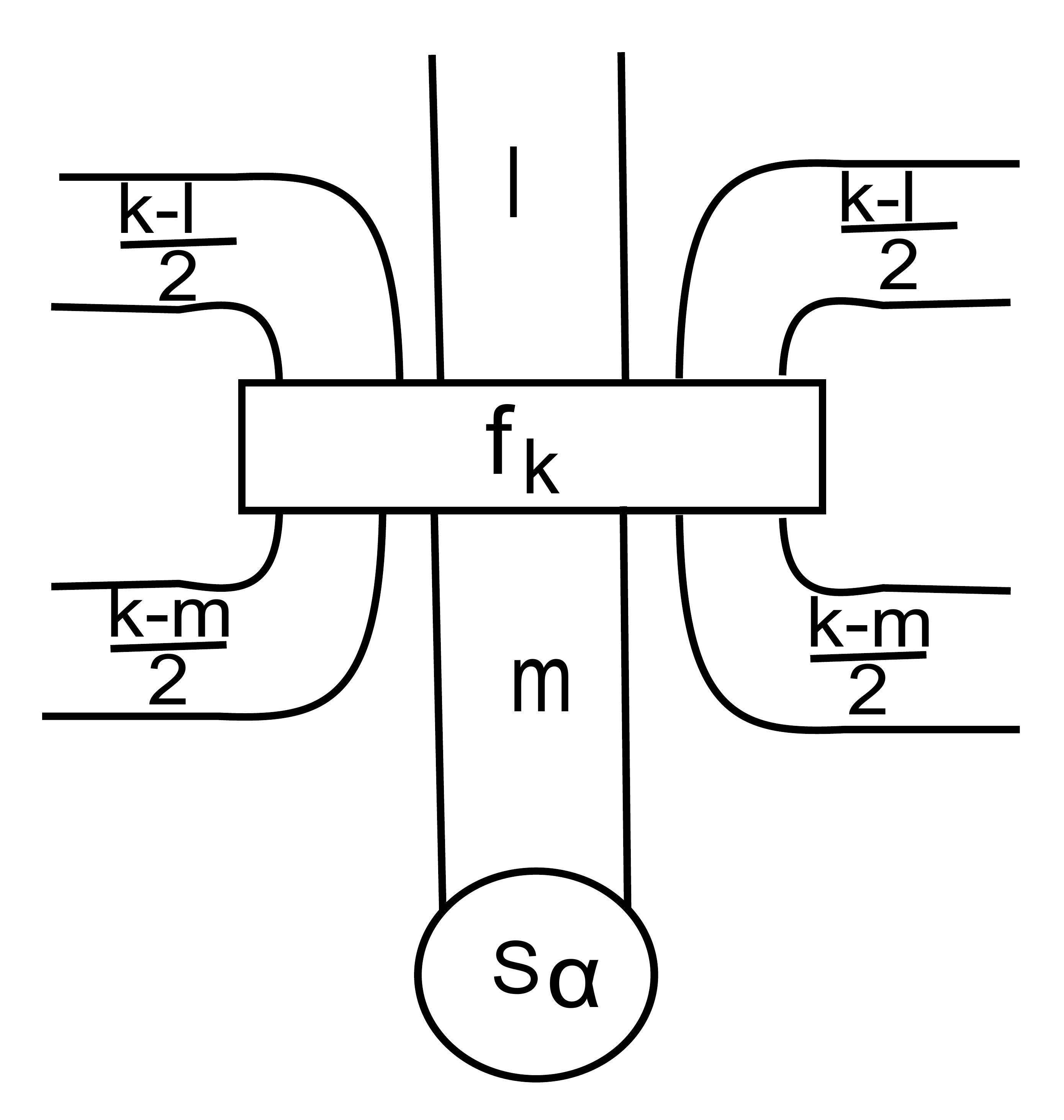}=B^{k}_{m,l}(\alpha)\grc{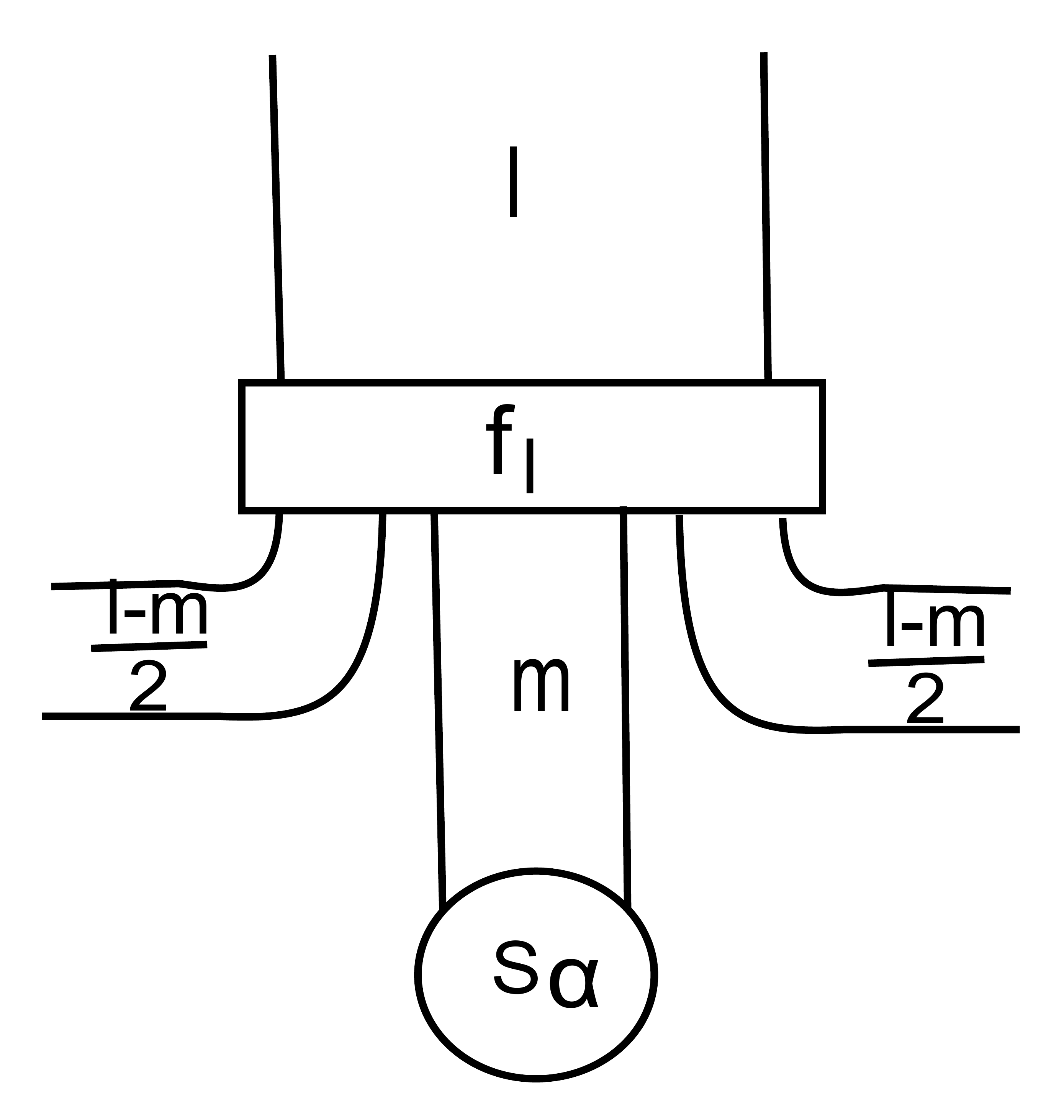}$.

Note that $B^{k}_{m,k}=1$ for all $k\ge m$.  As a matter of convention, we use $\alpha$ to represent an arbitrary irreducible representation parameter, while we use $t$ to represent a weight parameter (so that $t\in [-\delta, \delta]$) while we use $\omega\in S^{1}$ to represent a weight $>0$ parameter.

\begin{lemm} $[k]^{2}-\left[\frac{k-m}{2}\right]^{2}-\left[\frac{m+k}{2}\right]^{2}=(q^{k}+q^{-k})\left[\frac{k-m}{2}\right]\left[\frac{m+k}{2}\right]$
\end{lemm}

\begin{proof} Direct computation.
\end{proof}

Recall that $TL_{k,k}$ as a vector space is the linear span of all isotopy classes of rectangular diagrams and $k$ non-intersecting strings, with $k$ boundary points on the top and bottom of the rectangle, and each string is attached to exactly two of these boundary points.  Thus $f_{k}\in TL_{k,k}$ can be written as a linear combination of such diagrams. In general it is difficult to compute the coefficient of an arbitrary diagram in $f_{k}$, however there are several types of diagrams which have relatively easy coefficients. First, the coefficient of the identity diagram $1_{k}\in TL_{k,k}$ is one. The coefficient of the diagram $\grc{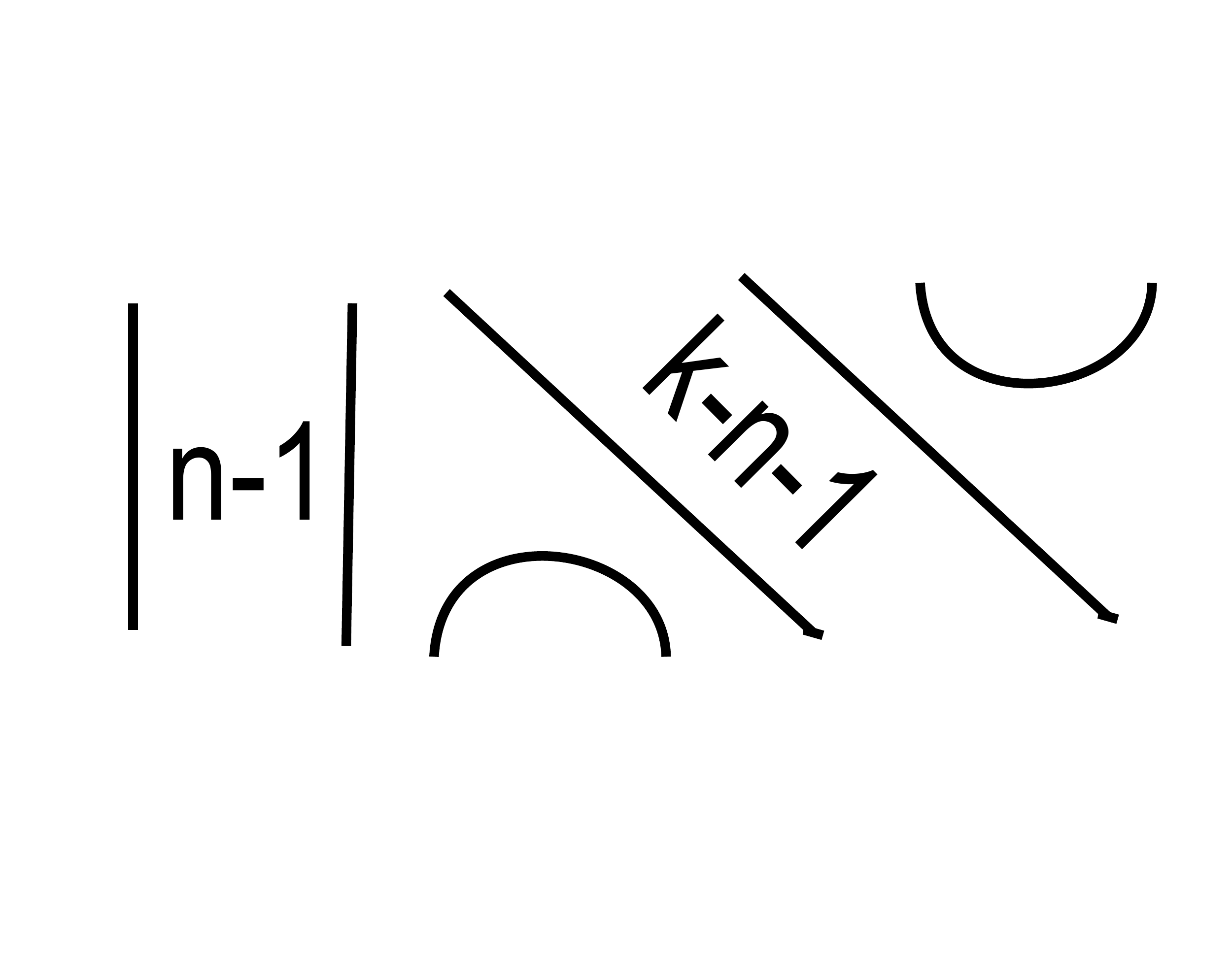}$ in $f_{k}$ is $(-1)^{n-k}\frac{[n]}{[k]}$.  For a proof of these formulas we refer the reader to Morrison's paper \cite{Mor}.  We note that the $f_{k}$  are invariant under vertical and horizontal reflection.  This implies diagrams obtained from one another by horizontal or vertical reflection will have the same coefficients in $f_{k}$. With these formulas in hand, we have the following proposition:

\begin{prop}

\ \ 

\begin{enumerate}

\item
For $m>0$, $\omega\in S^{1}$ and $k$ even, $B^{k}_{m,l}(\omega)=\frac{\left[\frac{k-m}{2}\right]\left[\frac{m+k}{2}\right]}{[k][k-1]}\left(q^{k}+q^{-k}-\omega^{2}-\omega^{-2}\right)B^{k-2}_{m,l}(\omega)$

\item
$B^{k}_{0,l}(t)=\frac{1}{[k][k-1]}\left([k]^{2}-t^{2}[\frac{k}{2}]^{2}\right)B^{k-2}_{0,l}$

\item
For $m>0$, $\omega\in S^{1}$ and $k$ odd, we have $B^{k}_{m,l}(\omega)=\frac{\left[\frac{k-m}{2}\right]\left[\frac{m+k}{2}\right]}{[k][k-1]}\left(q^{k}+q^{-k}-(i\omega)^{2}-(i\omega)^{-2}\right)B^{k-2}_{m,l}(\omega)$

\end{enumerate}
\end{prop}

\begin{proof}

First assume $m>0$.  Then we have $\grc{salpha2new}=\grc{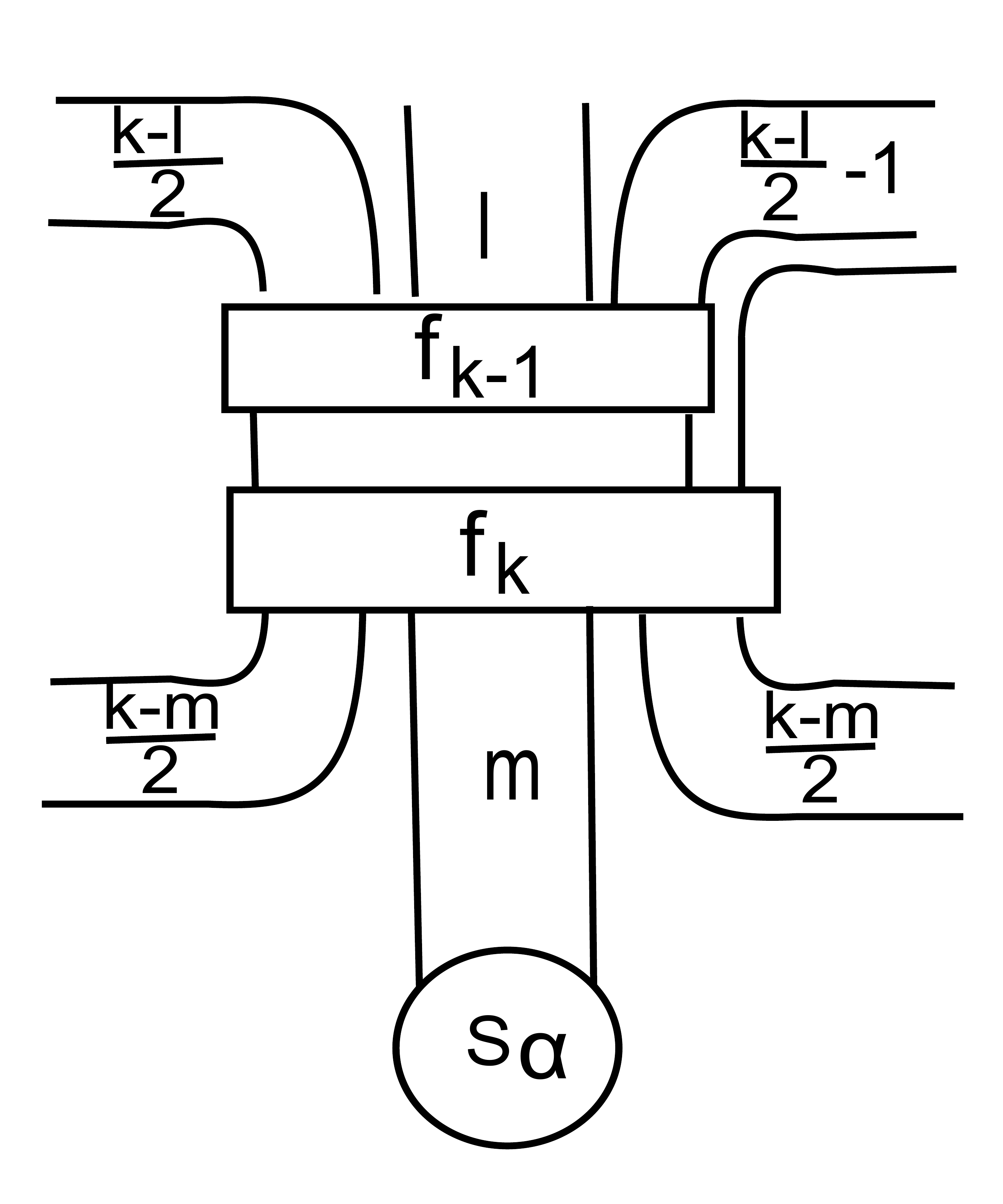}$

By uncapability of $f_k$, we see that there are precisely $3$ diagrams that can be inserted into the bottom $f_{k}$.  The identity diagram $1_k$, with no cups or caps on either the top or bottom, is the first.  There can only be one cup in the top, which must be on the top right. Such a diagram must have exactly one cap on the bottom, and it can be either at position $\frac{k-m}{2}$ or $\frac{k+m}{2}$.  As mentioned above, the coefficient of such a diagram is $(-1)^{\frac{k+m}{2}}\frac{[\frac{k-m}{2}]}{[k]}$ for the former and $(-1)^{\frac{k-m}{2}}\frac{[\frac{k+m}{2}]}{[k]}$ for the latter.  We see then pick up a value of $\omega^{-1}$ for the first diagram, and an $\omega$ for the second diagram.  Then the above is equal to 

$$(-1)^{\frac{k-m}{2}}\left(\frac{ (-1)^{m}[\frac{k-m}{2}]\omega^{-1}+[\frac{k+m}{2}]\omega}{[k]}\right)\grc{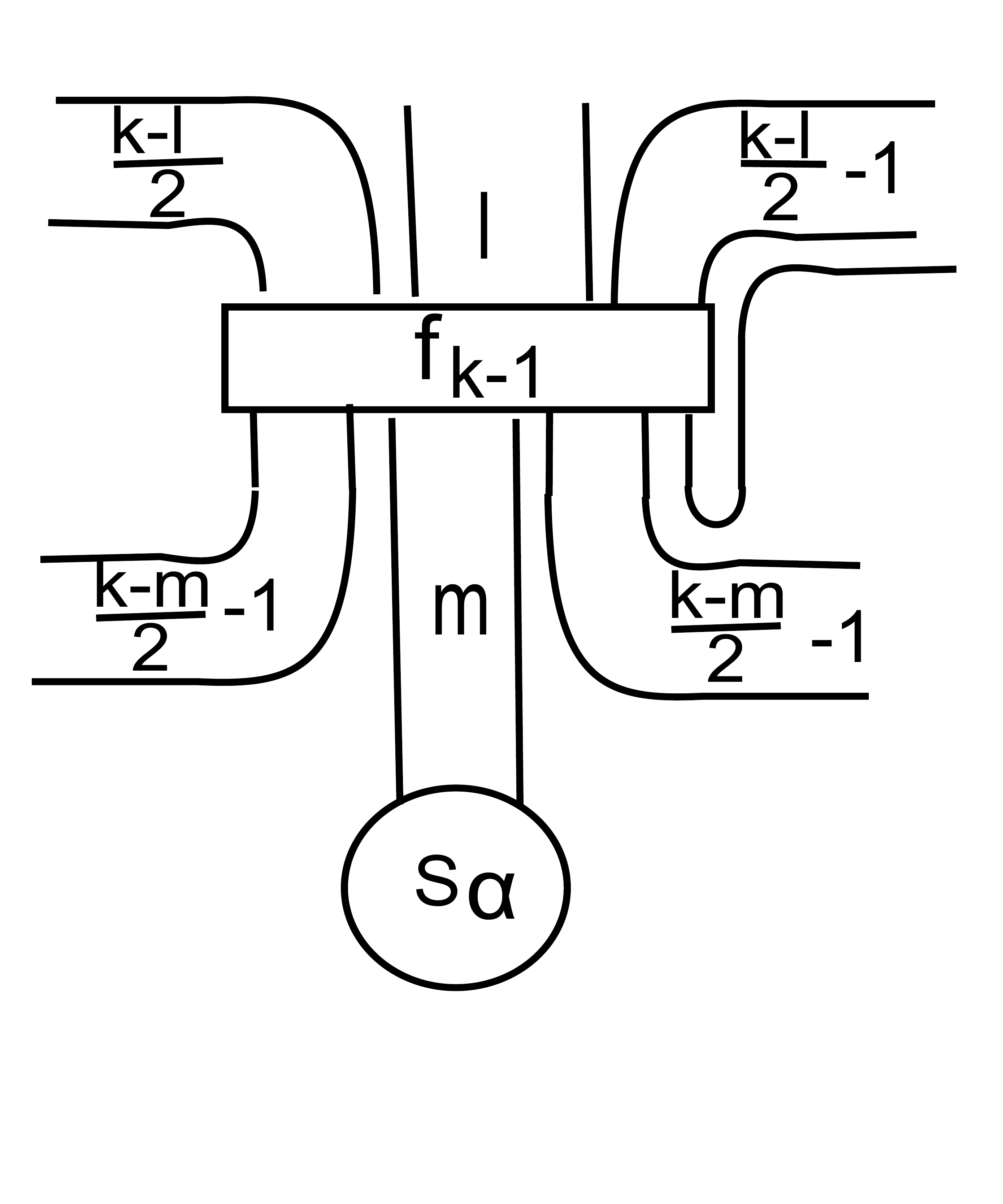}+\frac{[k]}{[k-1]}\grc{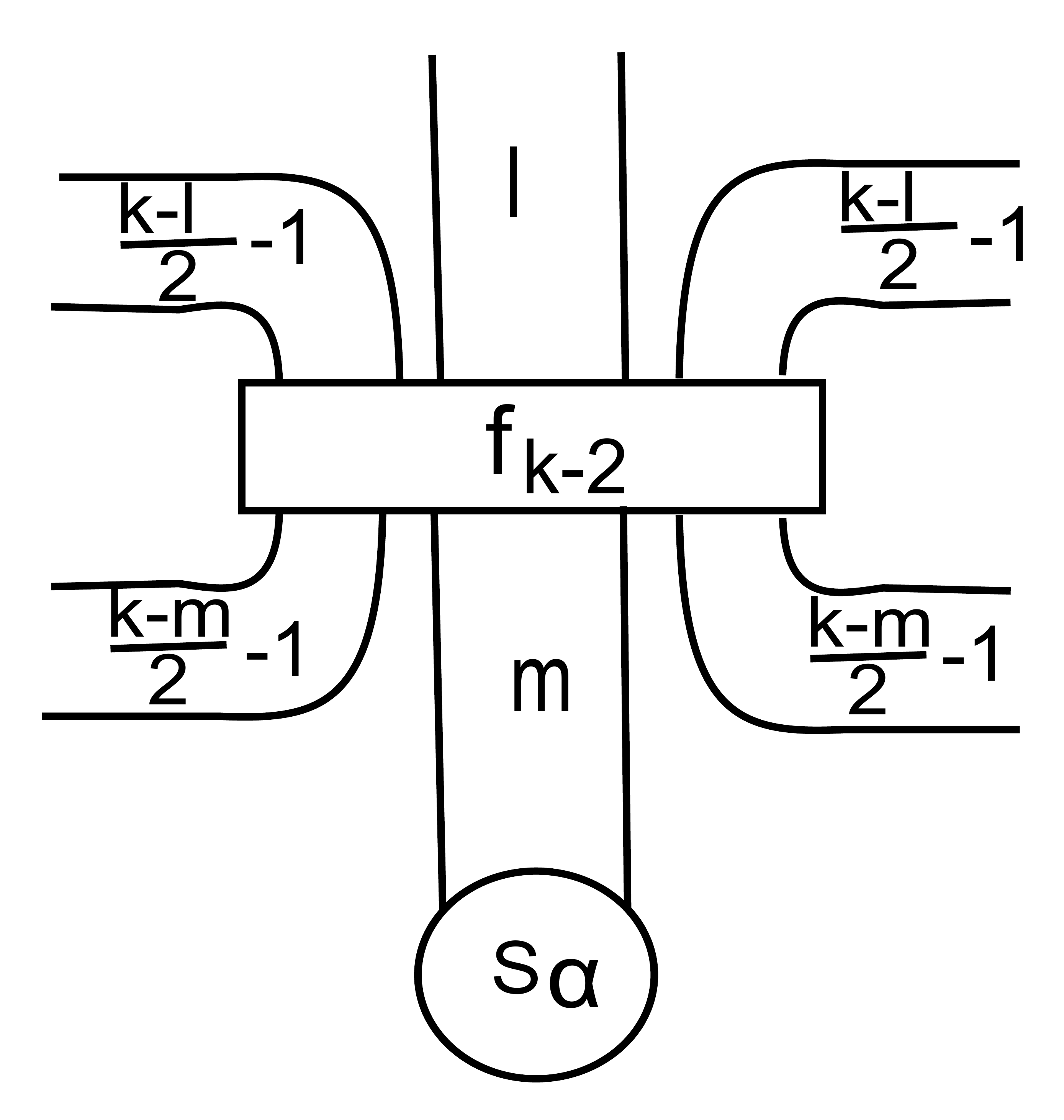}$$

The diagram on the right and its coefficient is obtained from plugging in $1_{k}$ in for $f_{k}$.  We see that applying annular relations introduces a copy of $f_{k-1}$, where the top left most string is attached to the bottom left most string around the left side.  This is nothing other than the left trace preserving conditional expectation $E_{L}:TL_{k-1, k-1}\rightarrow TL_{k-2, k-2}$ applied to $f_{k-1}$.  Since $E_{L}(f_{k-1})$ is uncapable on both the top and bottom, we have $E_{L}(f_{k-1})=c f_{k-2}$ for some scalar $c$.  Taking the trace on both sides givese us $Tr(f_{k-1})=cTr(f_{k-2})$, hence $c=\frac{[k]}{[k-1]}$. 

We want an expression just involving the diagram to the right in the sum, so we consider the left diagram in the sum and apply the same sort of argument.  We notice the diagram on the left is in fact equal to $\grc{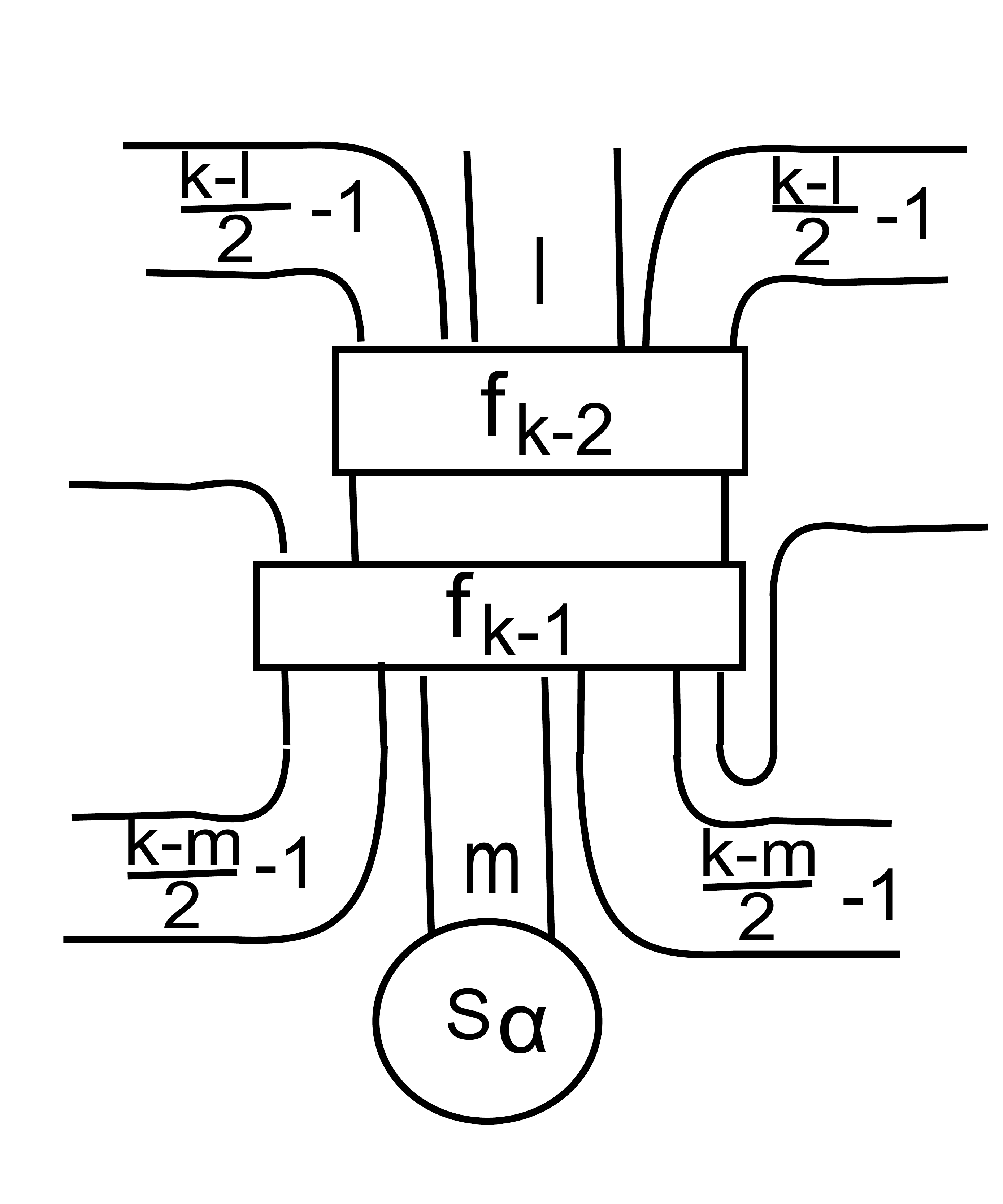}$.  Now here we see that the identity gives $0$.  There is only one possible place for a cup, and that is on the top left.  A careful consideration of caps and through strings shows a cap at the bottom right gives $0$.  There are precisely $2$ places for a bottom cap that give non-zero contributions, namely with positions at $\frac{k-m}{2}-1$ and  $\frac{k+m}{2}-1$.  The coefficients in $f_{k-1}$ of these in diagrams are $(-1)^{\frac{m-k}{2}+1}\frac{[\frac{k+m}{2}]}{[k-1]}$ and $(-1)^{\frac{m+k}{2}+1}\frac{[\frac{k-m}{2}]}{[k-1]}$ respectively.  Again the first coefficient picks up an $\omega^{-1}$ and the second picks up an $\omega$.  Thus this diagram is equal to 

$$-(-1)^{\frac{m-k}{2}}\left(\frac{ [\frac{k+m}{2}]\omega^{-1}+(-1)^{k}[\frac{k-m}{2}]\omega}{[k-1]}\right)\grc{d4new}$$

Putting everything together we end up with $$\frac{1}{[k][k-1]}\left( [k]^{2}- \left[\frac{k-m}{2}\right]^{2}-\left[\frac{k+m}{2}\right]^{2}- (-1)^{k}(\omega^{2}+\omega^{-2})\left[\frac{k-m}{2}\right]\left[\frac{k+m}{2}\right]\right)\grc{d4new}$$

 By the quantum number identity of Lemma $5.2$, $[k]^{2}-\left[\frac{k-m}{2}\right]^{2}-\left[\frac{m+k}{2}\right]^{2}=(q^{k}+q^{-k})\left[\frac{k-m}{2}\right]\left[\frac{m+k}{2}\right]$, so the above coefficient is $$\frac{\left[\frac{k-m}{2}\right]\left[\frac{m+k}{2}\right]}{[k][k-1]}\left(q^{k}+q^{-k}-(-1)^{k}(\omega^{2}+\omega^{-2})\right).$$  We immediately see the desired formulas for $k$ even and $k$ odd (for $k$ odd the i comes from the $(-1)^{k}=-1$, which we then bring inside the $(\omega^{2})$ as an $i$).

Now for $m=0$, we must have that $k, l$ are even.  We perform the same analysis:

$$\grc{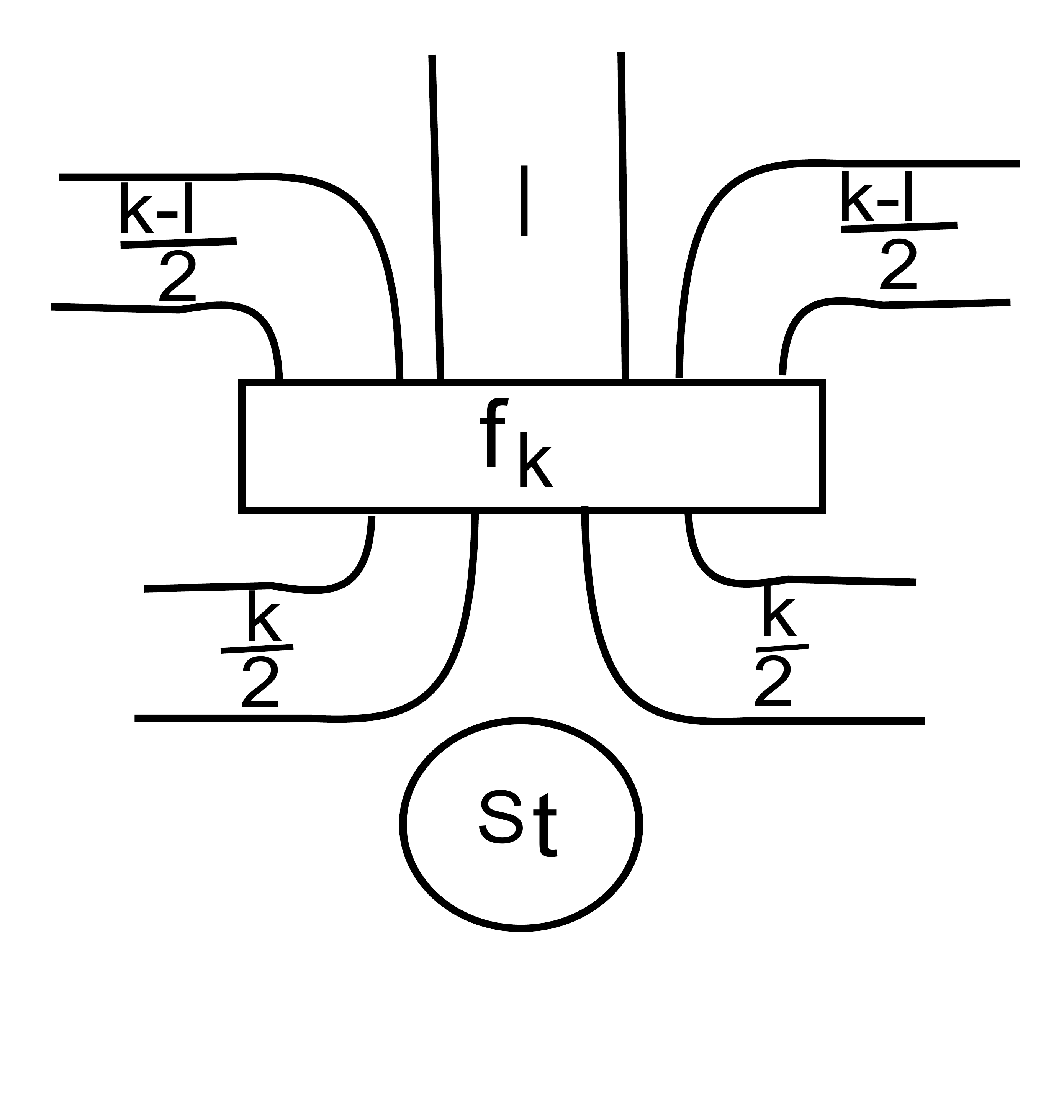}=\grc{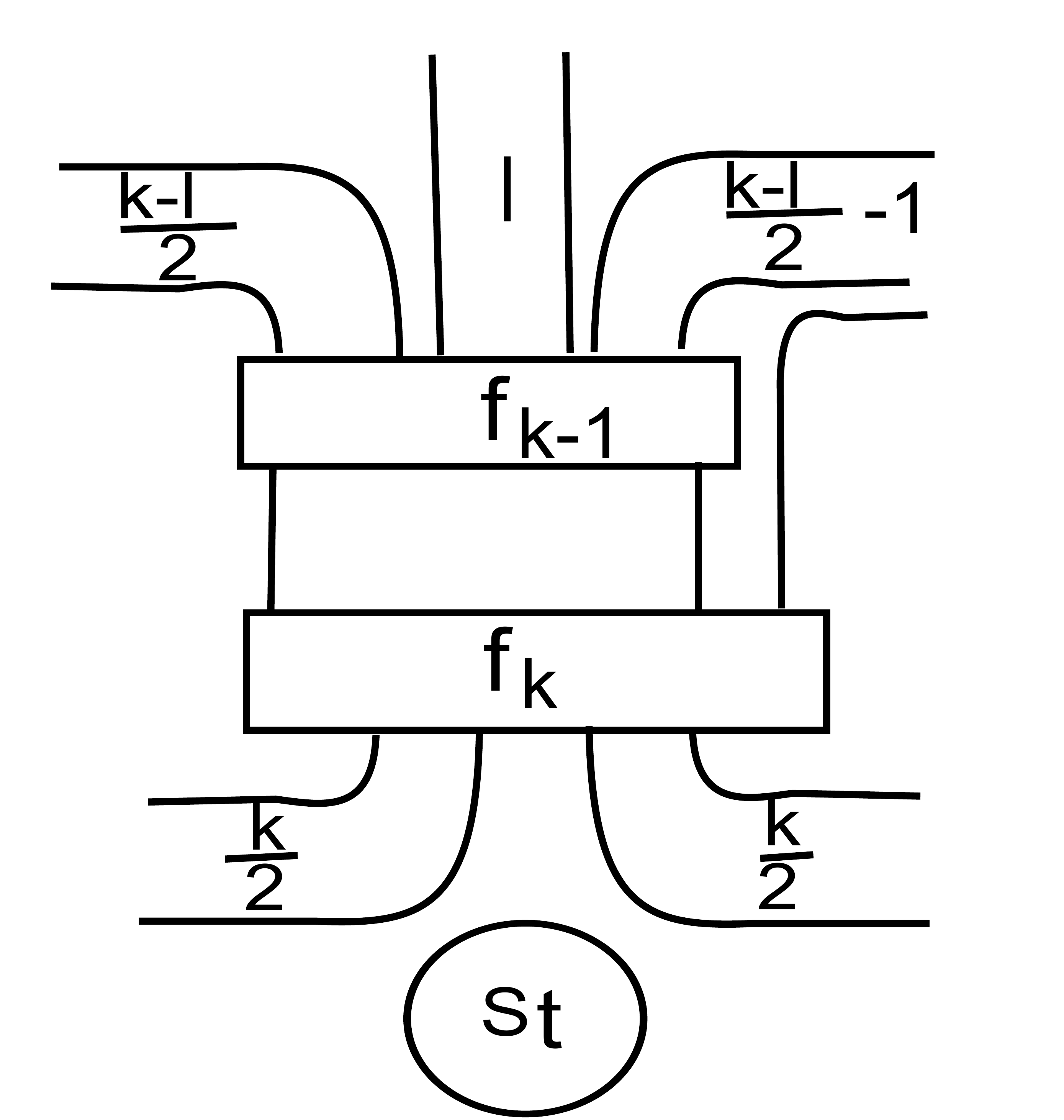}.$$

Evaluating the bottom $f_{k}$ with Temperley-Lieb diagrams, we see the identity $1_{k}$ yields $\frac{[k]}{[k-1]}\grc{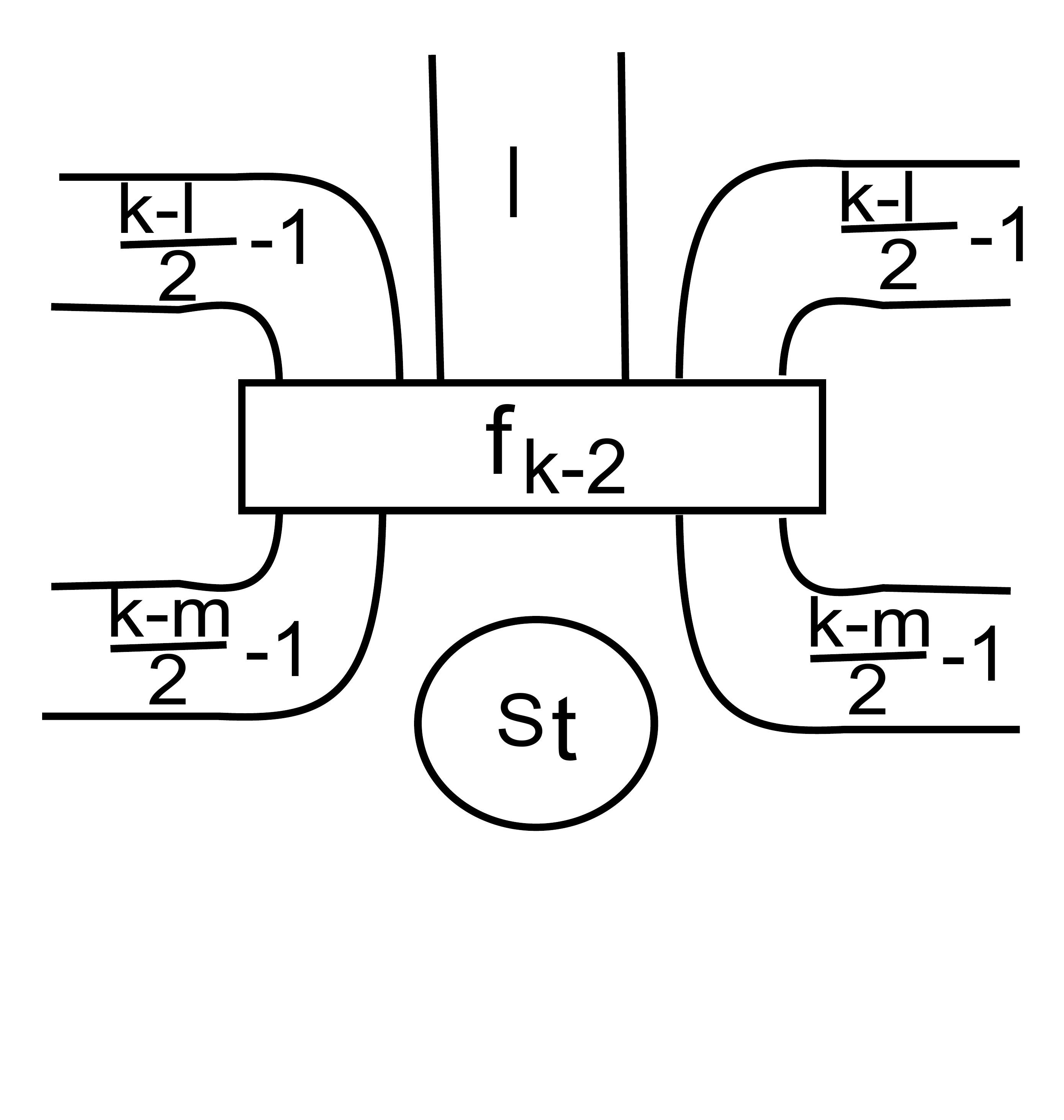}.$  Now there is only one possible non-zero cap location in the top (in the top right), and one possible cap on the bottom, at position $\frac{k}{2}$.  This diagram has coefficient $(-1)^{\frac{k}{2}}\frac{[\frac{k}{2}]}{[k]}$.  The cap at the bottom yields a factor of $t$ since it produces a homologically non-trivial circle around $s_{t}$, resulting in

$$(-1)^{\frac{k}{2}}t\frac{[\frac{k}{2}]}{[k]}\grc{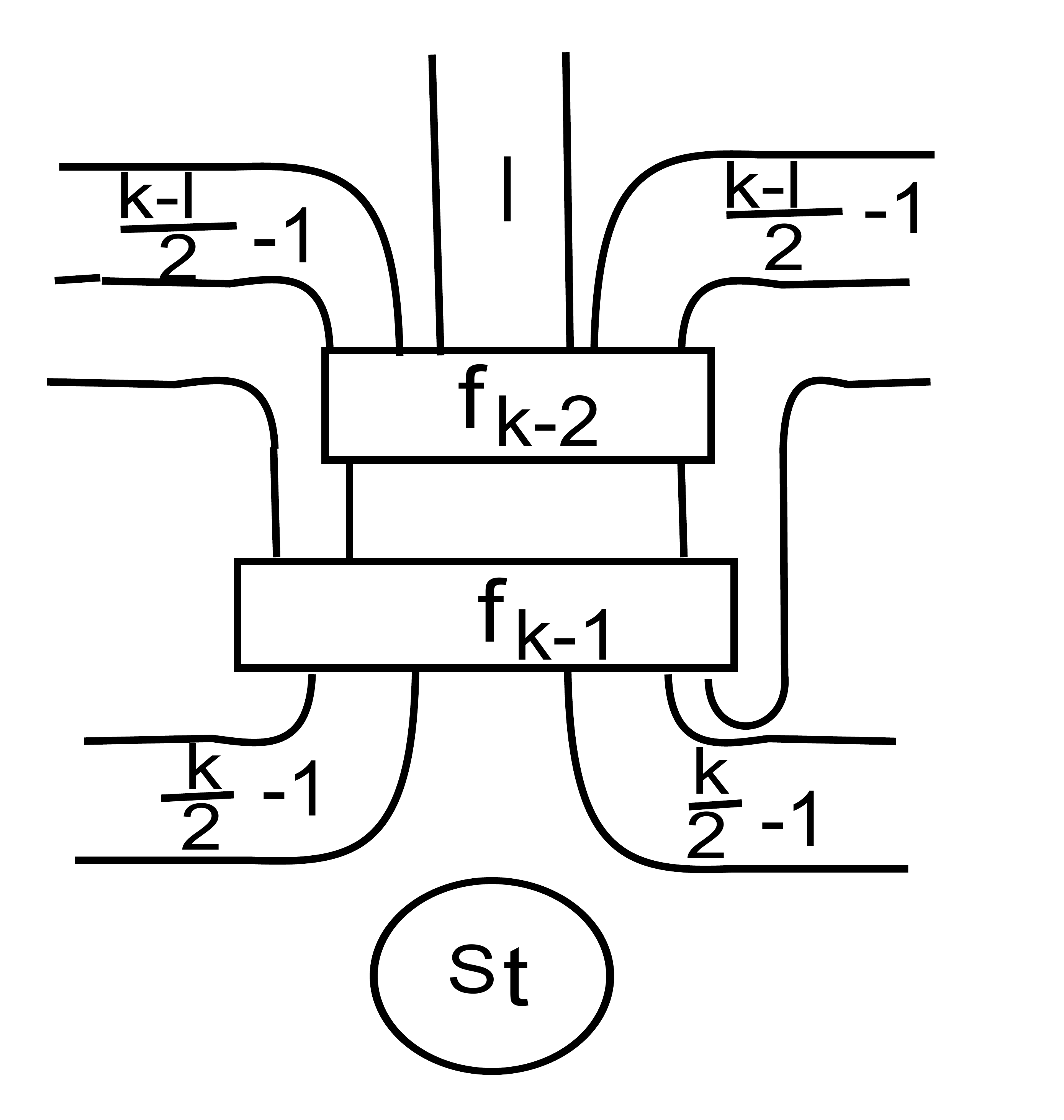}.$$

As in the case $m>0$, the identity $1_{k-1}$ yields $0$ at this step, and thus there is precisely one diagram which gives a non-zero contribution, with a cup in the upper left hand corner, and a cap on the bottom at position $\frac{k}{2}-1$.  The coefficient of this diagram in $f_{k-1}$ is $(-1)^{\frac{k}{2}-1}\frac{[\frac{k}{2}]}{[k-1]}$. Again a factor of $t$ pops out. Combining all the terms, we end up with our original expression equal to

$$\frac{1}{[k][k-1]}\left([k]^{2}-t^{2}\left[\frac{k}{2}\right]^{2}\right) \grc{d8}$$

This gives us the desired formula. 
\end{proof}

\medskip

\ \ As a corollary of Proposition 5.3, we can analyze the inner products on the spaces $\hat{V}^{\alpha}_{n}$ following Jones and Reznikoff \cite{Jo3}.  Let $\alpha$ be the parameter of a lowest weight $m$ representation.  $\hat{V}^{\alpha}_{n}\cong \bigoplus_{0\le j\le \lfloor\frac{n}{2}\rfloor} (f_{n-2j}, n)\otimes g^{\alpha}_{m,n-2j}$.  We see that this decomposition is orthogonal with respect to the sesquilinear form defined by our lowest weight $m$ representation.  If $x\otimes g^{\alpha}_{m, n-2j}, y\otimes g^{\alpha}_{m,n-2j}\in  (f_{n-2j}, n)\otimes g^{\alpha}_{m,n-2j}$, we see that $\langle x\otimes g^{\alpha}_{m, n-2j}, y\otimes g^{\alpha}_{m,n-2j}\rangle_{\alpha}=\langle x, y\rangle \langle g^{\alpha}_{m, n-2j},g^{\alpha}_{m, n-2j}\rangle_{\alpha}=\langle x,y\rangle B^{n-2j}_{m,m}(\alpha)$, where $\langle x,y\rangle$ denotes the positive definite inner product in the planar algebra.  An inspection of the formulas shows that $B^{n-2j}_{m,m}(\alpha)\ge 0$.  Thus our inner product is positive semidefinite, hence, taking the quotient by the kernel of our form, we obtain a sequence of finite dimensional Hilbert spaces $\{V^{\alpha}_{k}\}$ with $V^{\alpha}_{k}=0$ for $k<m$ (and 0 if the parity of $k$ is distinct from the parity of $m$).  We notice also that this inner product is uniquely determined by $\alpha$, thus for a given lowest weight $k$ and parameter $\alpha$, there is a unique Hilbert $TL$-module constructed as above.

\ \ In some cases, however, even with $k\ge m$, it may be that $g^{\alpha}_{m,k}=0$ in the quotient with respect to the positive semi-definite inner product.  This happens precisely when $B^{k}_{m,m}(\alpha)=0$.   Inspecting the coefficients as in \cite{Jo3}, we can determine when this happens.  For the weight $0$ case, we see that this happens precisely when $k>0$ and $t=\pm \delta$.  For $\delta>2$, all other $B^{k}_{m,m}(\alpha)$ are strictly positive for all $m$, $k\ge m$ and $\alpha$ .  When $\delta=2$ and hence $q=1$, the weight $0$ story is the same, but for higher weights we see that we run in to a problem in two places:  For $m$ even, $\omega=\pm 1$, $B^{k}_{m,m}(\pm 1)=0$ for all $k>m$.  For $m$ odd, we see that the problem occurs at $\omega=\pm i$, and $B^{k}_{m,m}(\pm i)=0$ for all $k>m$.  This will be relevant when we analyze the tube algebra representations of $TLJ(\delta)$, so we record the results in the following proposition.

\begin{prop} \cite{Jo3}, \cite{R}:  Irreducible lowest weight $m$ representations are classified as follows:  For a lowest weight $m$ representation with parameter $\alpha$, let $g^{\alpha}_{m,k}$ be the vector described above.  Recall that $g^{\alpha}_{m,k}=0$ if $k<m$.
\begin{enumerate}
\item
For $t\in [-\delta, \delta]$, there exists a unique irreducible lowest weight $0$ Hilbert $TL$-module $V^{t,0}:=\{V^{t}_{k}\ :\ k\ \text{is even}\}$.  For $t\in (-\delta, \delta)$, $g^{t}_{0,k}\ne 0$ for all even $k$.  $g^{\pm \delta}_{0,k}=0$ for all $k>0$.
\item
For $m>0$, $\omega \in S^{1}$, there exists a unique irreducible lowest weight $m$ Hilbert $TL$-module $V^{\omega, m}:=\{ V^{\omega, m}_{k}\ :\ m-k\ \text{is even} \}$.  For $\delta>2$, $g^{\omega}_{m,k}\ne 0$ for all $k\ge m$ with $k-m$ even.  For $\delta=2$, $k$ even, $g^{\pm1}_{m,m}=1$ and $g^{\pm 1}_{m,k}=0$ for all $k> m$.  For $\omega\ne \pm 1$, $g^{\omega}_{m,k}\ne 0$ for all even $k\ge m$.  If $m$ is odd, then $g^{\pm i}_{m,m}=1$ and $g^{\pm i}_{m,k}=0$ for all $k> m$, and for $\omega\ne \pm i$, $g^{\omega}_{m,k}\ne 0$ for all odd $k\ge m$. 
\item
Define the space $X^{+}_{\infty}:= [-\delta, \delta] \sqcup S_{1}\sqcup S_{1}\sqcup \dots$, with infinitely many copies of $S_{1}$, and $X^{-}_{\infty}:=S^{1}\sqcup S^{1}\sqcup \dots$.  Then irreducible representations in $Rep(ATL)$ are parameterized (as a set) by $X^{+}_{\infty}\sqcup X^{-}_{\infty}$.
\end{enumerate}
\end{prop}

We thank Makoto Yamashita for pointing out that the parametrization $(3)$ coincides with the parameterization of irreducible representations of the quantum Lorentz group $SL_{q}(2, \C)$, the Drinfeld double of $SU_{q}(2)$, determined by Pusz in \cite{Pu}.  However, as pointed out by the reviewer, Pusz only considers $q>0$, and we should expect $(3)$ to parameterize irreducible representations of $SL_{-q}(2, \C)$. 
 
We proceed to analyze the corners of the tube algebra of $TLJ(\delta)$.  We will denote the tube algebra for $TLJ(\delta)$ by $\A$.  Since the simple objects in our category are indexed by $k$, (namely, the $k^{th}$ Jones-Wenzl idempotent $f_k$), we let $k$ denote the equivalence class of $f_k$ as opposed to the identity in $TL_{k,k}$ from the planar algebra. To study the tube algebra we construct a nice basis for $\A_{k,k}$ which will allow us to exploit the planar algebra description of this category.  From the proof of Proposition 3.5, we see that $\A_{k,k}\cong f_{k}ATL_{k,k}f_{k}$.  In other words $\A_{k,k}$ is the cut down of the affine Temperley-Lieb $ATL_{k,k}$ space by the rectangular $k^{th}$ Jones-Wenzl projection $f_{k}$.  Thus we can construct a basis of $\A_{k,k}$ which consist of diagrams as follows:

For $k$ even and $j\in \mathbb{N}$, set $x^{k}_{0,j}:=\grc{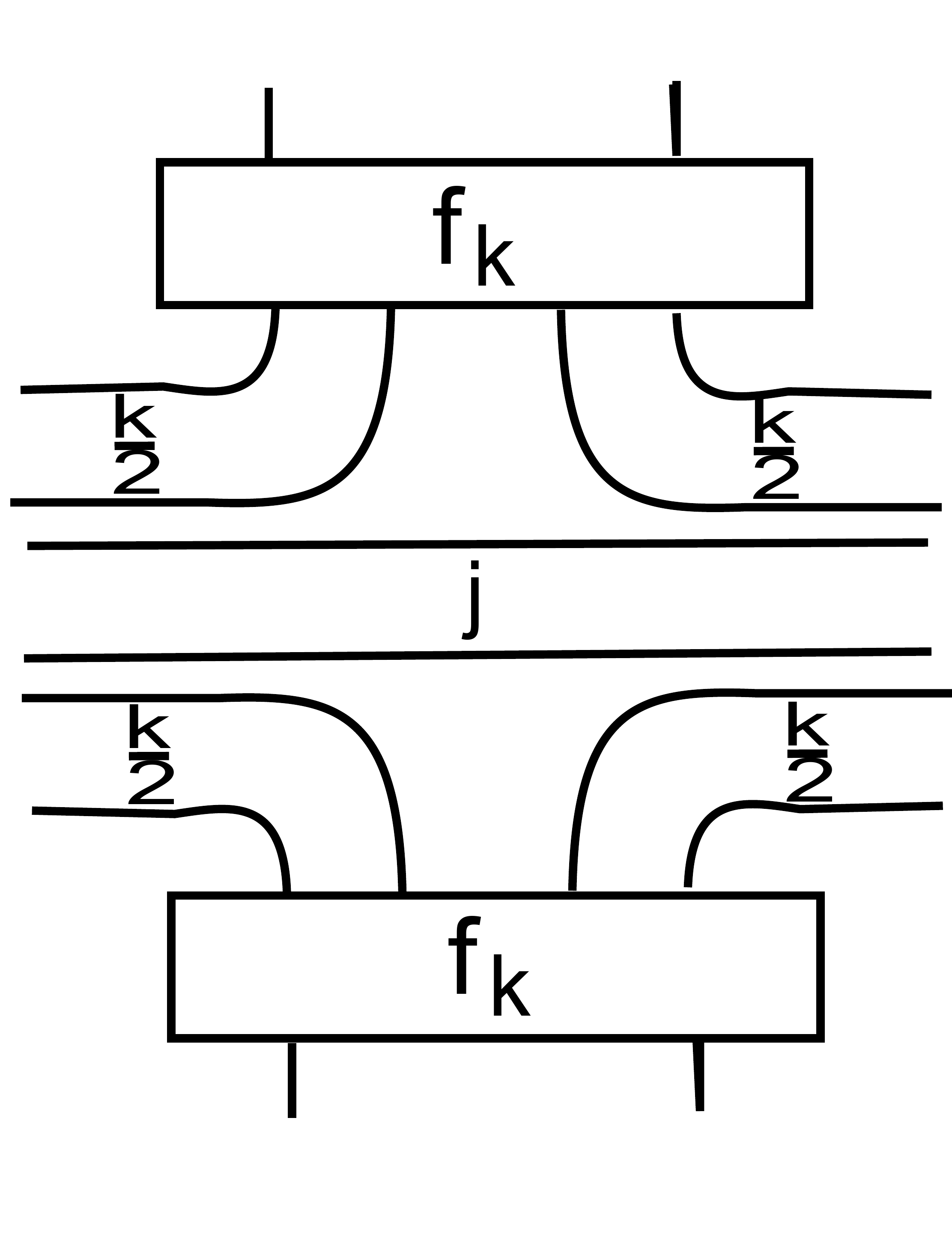}$.  For $n\in \mathbb{Z}$ and $0<m\le k$ with $k-m$ even, define $x^{k}_{m,n}:=\grc{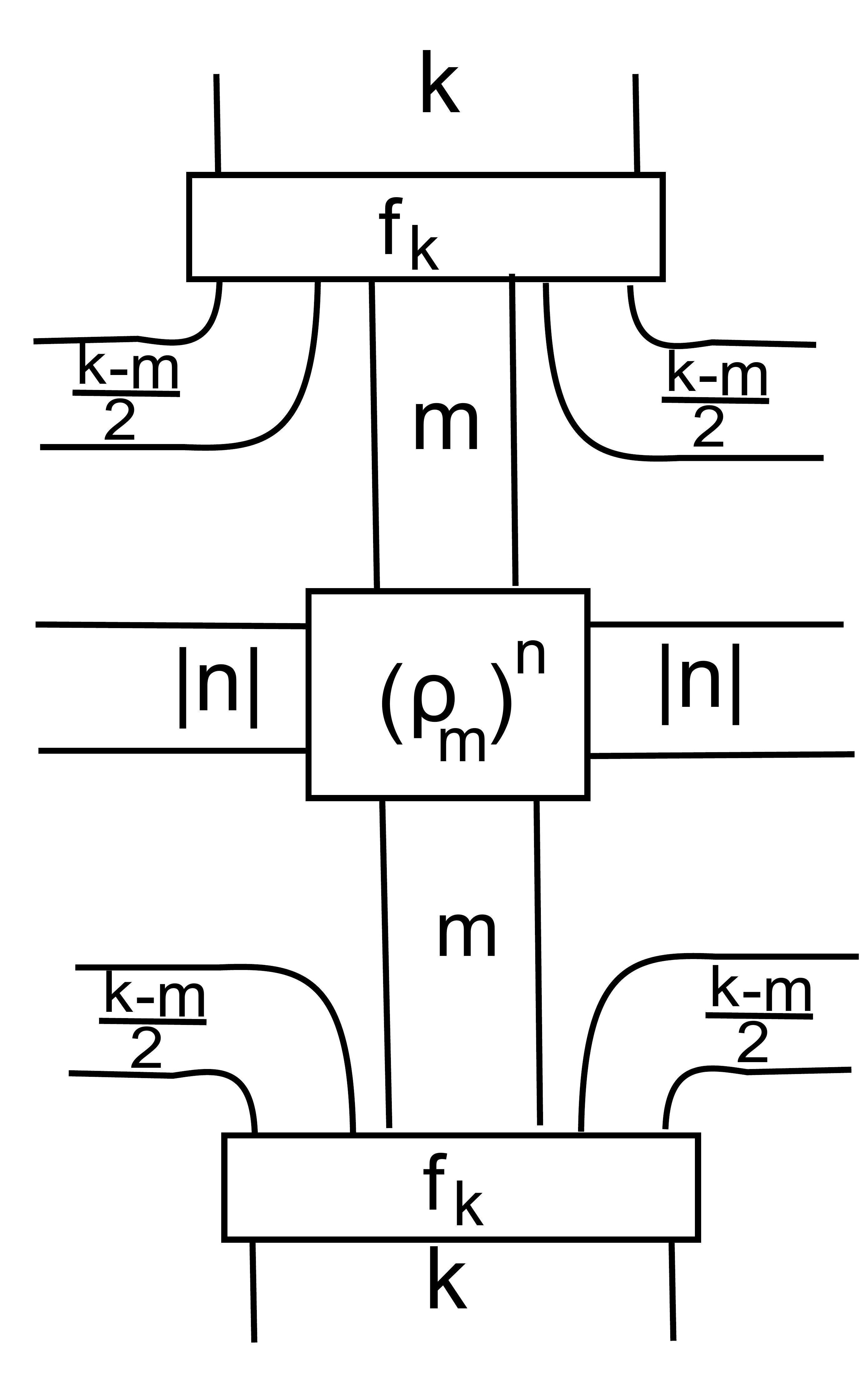}$.  Again, these pictures can and should be interpreted as representing annular tangles, with the strings on the left connecting to strings on the right around the bottom of an annulus.  In the center of the diagram $x^{k}_{m,n}$ is the $n^{th}$ power of the rotation $\rho_{m}$.  We define the rank of the diagram as $Rank(x^{k}_{m,n})=m$.  We see that the rank of a diagram in $\A_{k,k}$ must be the same parity as $k$. The rank corresponds to the number of strings starting from the bottom $f_{k}$ and going all the way to the top $f_{k}$. 

\begin{prop} Let $B:=\{x^{k}_{m,n}\ :\ m\in \mathbb{N},\ 0\le m\le k,\ k-m=0\ mod\ 2,\ n\in \mathbb{Z}\ \text{or}\ n\in \mathbb{N}\ \text{for}\ m=0\}$.  Then $B$ is a basis for $\A_{k,k}$.
\end{prop}

\begin{proof} Since $\A_{k,k}\cong f_{k}ATL_{k,k}f_{k}$, we see that the only diagrams that are not zero are in $B$, hence $B$ is a spanning set.  To see that these are linearly independent, we note that the diagrams listed above without the $f_{k}$ ( replacing each $f_{k}$ by $1_{k}\in ATL_{k,k}$) are linearly independent in $ATL_{k,k}$, since they correspond to distinct isotopy classes of diagrams.   We also note that these diagrams have no rectangular caps on their boundaries, which means that any cap on the top or bottom has to go ``around the bottom of the annulus''.   We have a bijective correspondence between $B$ and $ATL_{k,k}$ diagrams with no rectangular caps on the top and bottom boundaries, given by replacing the Jones-Wenzl idempotents in $x^{k}_{n,m}$ with the $1_{k}\in TL_{k,k}$.  We also note that by definition, the diagrams in $ATL_{k,k}$ with no rectangular caps on their boundaries must be linearly independent from the set of diagrams with some rectangular caps on their boundaries.  Suppose there exists some $\{b_{i}\}_{1\le i\le n} \subseteq B$ and $\lambda_{i}\in \C$ such that $\sum_{i}\lambda_{i}b_{i}=0$.  Let $\hat{b}_{i}\in ATL_{k,k}$ be the diagram obtained by replacing the top and bottom Jones-Wenzl idempotents in $b_{i}$ with the identity. Then evaluating the Jones-Wenzl idempotents at the top and the bottom of the diagrams in terms of $TL$ diagrams, we see that the only terms in both the top and bottom Jones-Wenzls that give no rectangular caps on the boundary are the identity diagrams $1_{k}\in TL$, and these have coefficient $1$ in $f_{k}$.  Since these diagrams are independent from the diagrams with caps, we notice that our equation implies $\sum_{i}\lambda_{i}\hat{b}_{i}=0$.  But our correspondence is bijective, and these are independent in $ATL_{k,k}$, hence there is no such collection of $\lambda_{i}$.
\end{proof}

\begin{prop} For every $k$, $\A_{k,k}$ is abelian.
\end{prop}

\begin{proof}    Define the map $r:\AC\rightarrow \AC$ given for $f\in \AC^{k}_{i,j}:=Mor(X_{k}\otimes X_{i}, X_{j}\otimes X_{k})$ by $r(f)=\overline{f}\in \AC^{\overline{k}}_{\overline{j}, \overline{i}}$.  $r$ is an anti-isomorphism with respect to annular multiplication. Then since $f_{k}=\overline{f_{k}}$, by the symmetry of our basis diagrams it is easy to see that  $r:\A_{k,k}\rightarrow \A_{k,k}$ given by a global rotation by $\pi$ is in fact the identity map on B, hence on all of $\AC_{k,k}$.  Then we have for any $x,y\in \A_{k,k}$,  $x\cdot y=r(x\cdot y)=r(y)\cdot r(x)=y\cdot x$.  Thus $\A_{k,k}$ is abelian.

\end{proof}

\medskip

This means $C^{*}(\A_{k,k})$ will be a unital, abelian $C^{*}$-algebra, hence isomorphic to the algebra of continuous complex valued functions on some compact Hausdorff space.  We describe these spaces below.

\begin{enumerate}
\item
Define $X_{0}:=[-\delta, \delta]$.
\item
For $k$ even, $k>0$, we define 

$$X_{k}:=\grd{spacexk}$$

\item
For $k$ odd, define 

$$X_{k}:=\grd{spacexkodd}$$
\end{enumerate}

  We will demonstrate the following:

\begin{theo} If $\delta>2$ then $C^{*}(\A_{k,k})\cong C(X_{k})$.

\end{theo}

For $\delta=2$, the situation is different.  As discovered in \cite{Jo3}, the annular representation theory of $ATL(2)$ is not generic.  In particular, there are some ``missing" one dimensional representations.  This will force us to identify points, resulting in some interesting topological spaces.

\begin{enumerate}
\item
Define $Y_{0}:=[-2, 2]$.
\item
For $k$ even, $k>0$ define

$$Y_{k}:=\grc{ykeven}$$
\item
For $k$ odd, define 

$$Y_{k}:=\grc{ykodd}$$
\end{enumerate}

\begin{theo} If $\delta=2$ then $C^{*}(\A_{k,k})\cong C(Y_{k})$.

\end{theo}

We note that in the case $k=0$, this essentially recovers a result of Popa and Vaes.  The only difference is that they use the even part of the $TLJ(\delta)$ category while we take the category as a whole, thus they have the ``square" of this interval, namely $[0,\delta^{2}]$ (see \cite{PV}).

\ \ To understand the one dimensional representations of $\A_{k,k}$ (which we often call characters) we note that (almost) every lowest weight $m$ representation with parameter $\alpha$ and $k-m$ even gives a one dimensional representation representation of $\A_{k,k}$.  We simply take the vector $g^{\alpha}_{k,m}$.  Then this will be an eigenvector of $\A_{k,k}$ viewed as a sub-algebra of $ATL_{k,k}$.  Thus if we understand the action of $\A_{k,k}$ on the vector $g^{\alpha}_{m,k}$ we will understand the characters.  There is a snag, however.  From the above proposition, some of these $g^{\alpha}_{m,k}$ are $0$ in the semi-simple quotient, hence do not produce characters on $\A_{k,k}$.  Furthermore, it is not a priori clear that every admissible representation of $\A_{k,k}$ comes from $ATL$ in the manner described here.  For example, it seems feasible that a one dimensional representation of $\A_{k,k}$ may have its canonical extension infinite dimensional in other weight spaces.  We will show that this is not the case.

\begin{lemm} For $\delta>2$, one dimensional representations of $\A_{k,k}$ are parameterized as a set by 
\begin{enumerate}
\item
If $k$ is even, $k>0$, the space $X_{k}:=(-\delta,\delta)\sqcup S^{1}\sqcup\dots\sqcup S^{1}$ if with $\frac{k}{2}$ copies of $S^{1}$
\item
If $k$ is odd, the  space $X_{k}:=S^{1}\sqcup\dots \sqcup S^{1}$ with $\frac{k+1}{2}$ copies of $S^{1}$
\item
If $k=0$, the space $X_{0}:=[-\delta,\delta]$.
\end{enumerate}

\end{lemm}

\begin{lemm} For $\delta=2$, one dimensional representations of $\AC_{k,k}$ are parameterized by:
\begin{enumerate}
\item
If $k>0$ is even, the space $Y_{k}:=(-2, 2)\sqcup \left(S^{1}-\{-1,1\}\right)\sqcup\left( S^{1}-\{-1,1\}\right)\sqcup\dots \sqcup S^{1}$ with $\frac{k}{2}-1$ copies of $S^{1}-\{-1,1\}$ and one copy of $S^{1}$.
\item
If $k$ is odd, $Y_{k}:=\left(S^{1}-\{-i, i\}\right)\sqcup \dots \sqcup S^{1}$ with $\frac{k+1}{2}-1$ copies of $S^{1}-\{-i, i\}$ and one copy of $S^{1}$.
\item
$Y_{0}:=[-\delta, \delta]$.
\end{enumerate}

\end{lemm}

\begin{proof}  This set produces characters by evaluating the action of $\A_{k,k}$ on the vectors $g^{\alpha}_{m,k}$ for $k-m\ge 0$ and $k-m$ even.  Now, the reason that $\pm \delta$ is missing in the interval $(-\delta,\delta)$ from all but $k=0$ is that the trivial representation of $\A_{0,0}$, does not extend to higher weight spaces by Proposition 5.4 (1), i.e. $g^{\pm \delta}_{0,k}=0$ in the semi-simple quotient of the canonical extension.  In the case $\delta=2$, we have from Proposition $5.4\ (2)$ that the characters corresponding to the lowest weight $k$ representations are ``missing", meaning that the corresponding $g^{\alpha}_{m,k}$ are $0$ for the parameters $\omega=\pm1$ for even $m>0$ and $\omega=\pm i$ for odd $m>0$.  Thus the sets listed describes all possible characters on $\AC_{k,k}$ coming from $ATL$, by \cite{Jo3}.  Applying Theorem $4.2$ we see that this yields all possible characters.  In particular, suppose we have a character $\alpha$ on $\A_{k,k}$.  Let $m$ be the smallest m such that the canonical extension to $\A_{m,m}$ is non-zero.  It is straightforward to check that since $\alpha$ is irreducible, the canonical extension to $\A_{m,m}$ is one dimensional.  Then when extended to an $ATL$ representation, this extends to an irreducible lowest weight $m$ representation, and we apply the classification of these described in the beginning of this section.

\end{proof}

\medskip

\ \ We can identify the points of the circles with characters of various weights, with each distinct circle corresponding to distinct weights.  The interior of the interval $(-\delta, \delta)$ corresponds to weight $0$ characters.  We know now that all characters must be given by $X_{k}$, but we do not yet know that distinct points in $X_{k}$ yield truly distinct characters on $\A_{k,k}$.  They yield distinct representations for $ATL$, but independent characters might become the same when restricted to the tube algebra.  In fact, we will show that they are distinct,  but first we see how to evaluate characters on a special subset of our basis, namely elements in $\A_{k,k}$ of the form $x^{k}_{m,1}$.

Let $t\in (-\delta, \delta)$, and $k$ even.  Then we see that $t(x^{k}_{0, j})=t^{j}B^{k}_{0,0}(t)$.  This is non-zero for $t\in (-\delta, \delta)$.  For $m>0$, and $\omega$ the eigenvalue for a lowest weight $m$ representation, we see that for $n\ge m$, $\omega(x^{k}_{n,1})g^{\omega}_{m,k}=x^{k}_{n,1}g^{\omega}_{m,k}=B^{k}_{m,n}(\omega)\omega(\rho_{n})g^{\omega}_{m,k}$, where here, we identify $\rho_{n}\in \A_{n,n}$ and the $\omega$ as a character on $\A_{n,n}$.  Thus to compute the value of $\omega(x^{k}_{n,1})$, we simply need to determine the value of $\omega(\rho_{n})$.  In pictures, we have to compute the scalar that pops out when we substitute $TL$ diagrams in the bottom $f_{n}$ of the picture $\grc{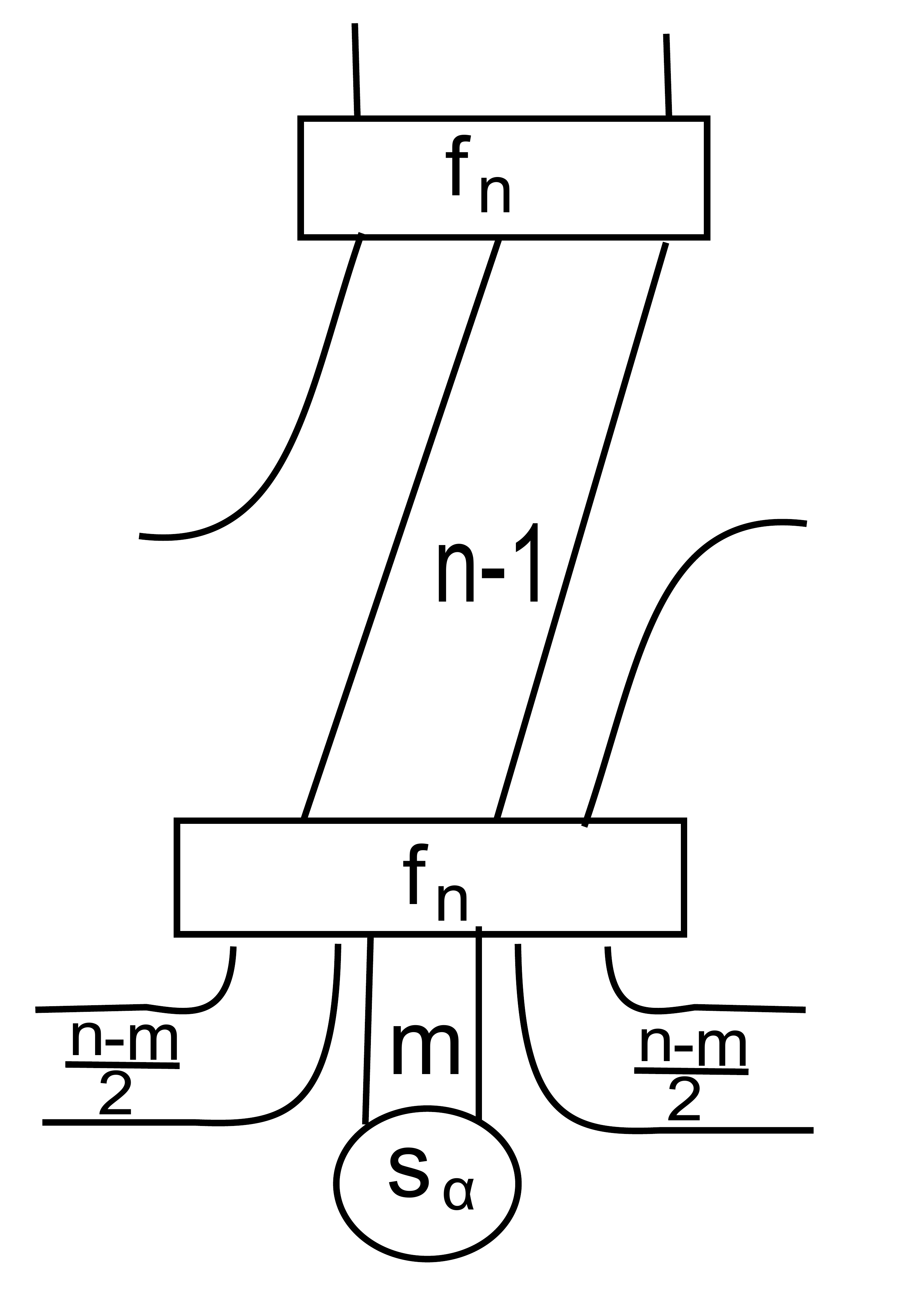}$.  If $n>m$, we see that there are precisely two diagrams which give non-zero contributions, a cup in the upper right hand corner, and a bottom cap at positions $\frac{n-m}{2}$ and $\frac{n+m}{2}$.  The coefficients of these diagrams in $f_{n}$ are given by $(-1)^{\frac{m+n}{2}}[\frac{n-m}{2}]/[n]$ and $(-1)^{\frac{m-n}{2}}[\frac{n+m}{2}]/[n]$ respectively.  The first diagram gives an eigenvalue of $\omega^{-1}$ and the second gives and eigenvalue of $\omega$, and thus we get $\omega(\rho_{n})=(-1)^{\frac{m-n}{2}}\left(\omega^{-1} (-1)^{n}[\frac{n-m}{2}]+ \omega[\frac{n+m}{2}]\right)/[n]$.  If $n=m$, we simply get $\omega$.  We apply the same procedure for $m=0$, which is even easier since there is only one $TL$ diagram to evaluate.  

\ \ We also notice that applying this same procedure to arbitrary basis diagrams, we see that an element of $\A_{k,k}$ evaluated at a character $\alpha$ will depend on $\alpha$ only as polynomial either in $\alpha$ and $\alpha^{-1}$ if $\alpha\in S^{1}$ or just in $\alpha$ if $\alpha\in (-\delta, \delta)$.   We record these results in the following lemma, which expresses our knowledge of how to evaluate characters:

\begin{lemm} Let $k>0$.
\begin{enumerate}
\item
For $k$ even, $t\in (-\delta,\delta)$, $k$ even, we have $t(x^{k}_{0,j})=t^{j}B^{k}_{0,0}(t)$.
\item
For $k$ even, $t\in (-\delta, \delta)$, $t(x^{k}_{n, 0})=B^{k}_{0,n}(t)$.  $t(x^{k}_{n,1})=(-1)^{\frac{n}{2}}t\frac{[\frac{n}{2}]}{[n]}B^{k}_{0,n}(t)$.
\item
For $\omega\in S^{1}$ of lowest weight $m>0$, for $k,n\ge m$,  $\omega(x^{k}_{n,1})= \frac{(-1)^{\frac{m-n}{2}}}{[n]}\left ((-1)^{n}[\frac{n-m}{2}]\omega^{-1}+[\frac{n+m}{2}]\omega\right) B^{k}_{m,n}(\omega)$, where here $[0]=0$.
\item
If $\omega\in S^{1}$, then $\omega(x^{k}_{j,m})\in \C[\omega, \omega^{-1}]$, and if $t\in (-\delta, \delta)$, then $t(x^{k}_{j,m})\in \C[t]$.
\end{enumerate}
\end{lemm}

\begin{lemm} For $\delta\ge 2$, and $X_{k}$ as above, $\A_{k,k}$ separates the points of $X_{k}$.

\end{lemm}

\begin{proof}

For each pair of distinct characters $\alpha_{1}, \alpha_{2}\in X_{k}$, we must show that there exists $f\in \A_{k,k}$ such that $\alpha_{1}(f)\ne \alpha_{2}(f)$.

First consider the $k=0$ case.  Then $t(x^{0}_{0,1})=t$ separates all points in $[-\delta, \delta]$.  

Now suppose $k>0$.  If $\alpha_{1}$ and $\alpha_{2}$ correspond to different weights, assume without loss of generality that the weight of $\alpha_{1}$ is strictly less than the weight of $\alpha_{2}$.  Then suppose the weight of $\alpha_{1}$ is m. Then we pick the diagram $x^{k}_{m,0}$.  Then from the above proposition, we have that $\alpha_{1}(x^{k}_{m,0})=B^{k}_{m,m}(\alpha_{1})\ne 0$, while $\alpha_{2}(x^{k}_{m,0})=0$ since $x^{k}_{m,0}$ has rank $m$.  Thus we can separate characters of different weights, and only need to show that we can seperate characters of the same weight.

Consider the case when $\delta>2$, and $k$ even.

\ \ Suppose $\alpha_{1}, \alpha_{2} \in (-\delta, \delta)$.  Then we have $B^{k}_{0,0}(\alpha_{1})$, $B^{k}_{0,0}(\alpha_{2})\ne 0$, and thus if $\alpha_{1}(x^{k}_{0,0})=B^{k}_{0,0}(\alpha_{1})\ne B^{k}_{\alpha_{2}}=\alpha_{2}(x^{k}_{0,0})$ we are done.  If $B^{k}_{0,0}(\alpha_{1})=B^{k}_{0,0}(\alpha_{2})\ne 0$, then $\alpha_{1}(x^{k}_{0,1})=\alpha_{1}B^{k}_{0,0}(\alpha_{1})\ne \alpha_{2}B^{k}_{0,0}(\alpha_{2})=\alpha_{2}B^{k}_{0,0}(\alpha_{2})$.  Thus we can separate the weight $0$ characters with $\A_{k,k}$.

\ \ Now suppose $\alpha_{1}, \alpha_{2} \in X_{k}$ are of the same weight $m>0$ but $\alpha_1\ne \alpha_{2}$.  Then $\alpha_{1}(x^{k}_{m,0})=B^{k}_{m,m}(\alpha_{1})$, and $\alpha_{2}(x^{k}_{m,0})=B^{k}_{m,m}(\alpha_{2})$.  If $ B^{k}_{m,m}(\alpha_{1})\ne B^{k}_{m,m}(\alpha_{2})$ we are done.  Suppose these are equal.  They are not $0$ by Proposition 5.4 (2).    Then $\alpha_{1}(x^{k}_{m,1})=\alpha_{1}B^{k}_{m,m}(\alpha_{1})$ while $\alpha_{2}(x^{k}_{m,1})=\alpha_{2}B^{k}_{m,m}(\alpha_{2})$.  Since $\alpha_{1}\ne \alpha_{2}$ we are finished.

The other cases are the same.  For $\delta=2$, we simply remove the points in the domain where $B^{k}_{m,m}=0$, and the above proof applies.

\end{proof}

\medskip

Now, we know that $C^{*}(\A_{k,k})$ will be a unital ($f_{k}$ is the unit) abelian $C^{*}$-algebra thus it must be isomorphic to the continuous functions on some compact Hausdorff space.  Since the characters evaluated on $\A_{k,k}$ are simply polynomials in the parameters of $X_{k}$ (Lemma 5.11 (4)), away from the ``missing'' points ($\pm \delta$, and when $\delta=2$, the points corresponding to $\pm 1$ on the even circles and $\pm i$ on the odd circles), the topology on the set of characters precisely agrees with the natural topology on the spaces.  Let us now consider the case when $\delta>2$.  The only ``missing" points are $t=\pm \delta$.  In other words, since our character space is compact and the topology on $X_{k}$ as characters agrees with the natural topology on $(-\delta,\delta)$, if we have a sequence of characters $t_{n}\subseteq(-\delta,\delta)$, such that $t_{n}\rightarrow \pm \delta$, this sequence must be converging to some other character in $X_{k}$.  Thus to identify the topology on $X_{k}$ as the space of characters, we must identify which character such a sequence $t_{n}$ converges to.  It must live in $X_{k}$ since $X_{k}$ contains all characters.

\begin{lemm} Let $\delta>2$, and let $k=2n$ be even.  Let $\omega_{-1}$ be the point $-1\in S^{1}\subseteq X_{k}$ corresponding to the weight $2$ copy of $S^{1}$  and similarly, $\omega_{1}$ the point in the same circle corresponding to $1$. Then for any $f\in \A_{k,k}$, if $\{t_{n}\}\subseteq (-\delta,\delta)$ is a sequence such that $t_{n}\rightarrow \delta$,  $t_{n}(f)\rightarrow \omega_{-1}(f)$.  If $t_n\rightarrow -\delta$, then $t_{n}(f)\rightarrow \omega_{1}(f)$.
\end{lemm}

\begin{proof} First from the list of coefficients above $B^{k}_{0,0}\rightarrow 0$ as $t_{j}\rightarrow \pm \delta$, , and thus $t_{n}(x^{k}_{0,j})=t^{j}B^{k}_{0,0}\rightarrow 0$ $t_{j}\rightarrow \pm \delta$, $t_{j}(x_{l})\rightarrow 0$ for all $l$.  Thus the limit of $t_{j}$ must be some higher weight character.  We see that $$B^{k}_{0,2}(t_{j})=\prod_{1\le i\le \frac{k}{2}}\frac{[2i]^{2}-t^{2}_{j}[i]^{2}}{[2i-1][2i]}\rightarrow \prod_{1\le i\le \frac{k}{2}}\frac{[2i]^{2}-[2]^{2}[i]^{2}}{[2i-1][2i]}$$  On the other hand using our formula for the $B$'s and Lemma 5.2, $$B^{k}_{2,2}(\pm1)=\prod_{i=2}^{n}\frac{[2i]^{2}-[i-1]^{2}-[i+1]^{2}-2[i+1][i-1]}{[2i][2i-1]}$$

\ \ Using the fact that $[2][i]=[i+1]+[i-1]$, and comparing each term in the product with the same denominator, we see that the term in the limit of the $t_{j}$ is 
$[2i]-[2]^{2}[i]^{2}=[2i]-([i+1]+[i-1])^{2}=[2i]-[i+1]^{2}-[i-1]^{2}-2[i+1][i-1]$, which is precisely the term in $B^{k}_{2,2}(\pm1)$.  Therefore we see that $\lim\ t_{n}$ must be a lowest weight $2$ character, and it must be $\omega_{\pm}$.  The problem is, we do not know which it is.  To determine this, we notice that $\alpha(x^{k}_{2,1})=\alpha B^{k}_{0,2}$.

\ \  Therefore, as $t_{n}\rightarrow \delta$, $t_{n}(x^{k}_{2,1})\rightarrow -B^{k}_{0,2}(-1)=\omega_{-1}(x^{k}_{2,1})$.  Since $x^{k}_{2,1}$ separates points, we see that $\lim_{t\rightarrow \delta} t=\omega_{-1}$.  Similarly, $\lim _{t\rightarrow -\delta} t=\omega_{1}.$

\end{proof}

\medskip

\begin{lemm} Let $k>0$, $\delta=2$.  
\begin{enumerate}
\item
Suppose $k$ is even.  Let $\omega_{\pm 1}$ be the characters on the weight $k$ circle.  If $\omega_{n}$ is a subset of the lowest weight $m$ circle for some $m\le k$ such that $\omega_{n}\rightarrow \pm 1$, then $\omega_{n}(f)\rightarrow \omega_{\pm (-1)^{\frac{k-m}{2}}}(f)$.
\item
Let $k$ be odd and $\omega_{\pm i}$ be the characters on the weight $k$ circle corresponding to $\pm i$.  If $\omega_{n}$ is a subset of a weight $m$ circle for some $m\le k$ such that $\omega_{n}\rightarrow \pm i$, then $\omega_{n}(f)\rightarrow \omega_{\mp (-1)^{\frac{m-2}{2}} i}$.
\end{enumerate}

\end{lemm}

\begin{proof} If $\omega_{n}\rightarrow \pm 1$ by examining coefficients, we see that $B^{k}_{m, n}(\omega_{n})\rightarrow 0$.  Since this coefficient occurs in the evaluation of $\omega_{n}(x)$ for all diagrams of rank $<k$, we see that $\omega_{n}$ must be converging to a lowest weight $k$ character.  To determine which one, we note that $x^{k}_{k,1}=\rho_{n}$, and compute $\omega_{n}(x^{k}_{k,1})=(-1)^{\frac{k-m}{2}}\frac{1}{k}\left( (-1)^{k} \frac{k-m}{2}\omega^{-1}_{n}+\frac{k+m}{2}\omega_{n}\right)$.

\ \ If $k$ is even, then as $\omega_{n}\rightarrow\pm 1$, $\omega_{n}(x^{k}_{k,1})\rightarrow \pm (-1)^{\frac{k-m}{2}}$.  Since $x^{k}_{k,1}$ separates lowest weight $k$ representations, we are done.

\ \ If $k$ is odd, then as $\omega_{n}\rightarrow \pm i$, $\omega_{n}(x^{k}_{k,1})\rightarrow \mp (-1)^{\frac{k-m}{2}} i$.  Since $x^{k}_{k,1}$ separates lowest weight $k$ representations, we are done.
\end{proof}

\textit{Proof of Theorems 5.7 and 5.8 }
The above lemmas have identified the appropriate topology on the sets $X_{k}$and $Y_{k}$, and it agrees with the pictures we have drawn. 

 In particular, consider $\delta>2$.  For $k$ odd, we don't even need the lemmas, since there are no ``missing" points.  For $k$ even, we must identify the points $\pm \delta$ on the interval $[-\delta,\delta]$ with the points $\mp 1$ respectively, on the weight $2$ circle.  Thus we have that $C^{*}(\A_{k,k})$ is an abelian $C^{*}$ algebra whose spectrum is the compact Hausdorff space $X_{k}$ .  

Now assume $\delta=2$.  For $k$ even, by the above lemma, the weight $m$ circle for $m>0$ will be glued on to the weight $k$ circle at the points $\pm 1$, and it alternates which endpoint goes to which endpoint on the circle as $\frac{k-m}{2}$ changes parity. We know by the above lemma that the interval is glued with its endpoints attached to the points $\pm 1$ on the weight $2$ circle which in turn is glued to the points $\pm1$ on the weight $k$ circle, resulting in the space pictured as $Y_{k}$.  For $k$ odd, we glue the points $\pm i$ to the highest weight circle in an alternating fashion as described in the above lemma.  Topologically, we obtain the space $Y_{k}$ pictured $\square$.
 
\medskip

\medskip

\textbf{Remark.}   It would be interesting to understand the tensor structure on $Rep(ATL)$ in terms of this topological characterization.  We also have not characterized $\A_{k,m}$ for $m\ne n$.  We notice that its completion is a Hilbert $C^{*}$ bimodule over $C(X_{m})-C(X_{k})$, and thus $\A_{k,m}$ should have some characterization as vector fields on $X_{k}\times X_{m}$, which may be interesting.  Understanding this structure will help determine the tensor structure on the category $Rep(ATL)$.

\end{subsection}

\end{section}

\begin{section}{Approximation and Rigidity Properties}

\begin{subsection}{Weight $0$ representations and analytical properties.}

In a remarkable paper \cite{PV}, Popa and Vaes defined representation theory for rigid $C^{*}$-tensor categories. They introduce the concept of cp-multipliers for $\ca$, which are a class of functions in $\ell^{\infty}(\Irr)$.  Normalizing these functions provide positive linear functionals on the fusion algebra $\Fus$.  A \textit{admissible representation} of $\Fus$ is a $*$-representation such that every vector state is a certain normalization of a cp-multiplier (see Definition 6.4).  The class of admissible representations of the fusion algebra provides a good notion for the representation theory for $\ca$, generalizing unitary representations of a discrete group if $\ca$ is equivalent to $G-Vec$.  In this context, they define approximation and rigidity properties, generalizing the definitions from the world of discrete groups.  They show that if $\ca$ is equivalent to the category  of $M$-$M$ bimodules in the standard invariant of a finite index inclusion $N\subseteq M$ of $II_{1}$ factors, then the definitions of approximation and rigidity properties given via cp-multipliers are equivalent to the definitions defined via the symmetric enveloping algebra for the subfactor $N\subseteq M$ given by Popa.

We will show in this section that admissible representations of the fusion algebra in the sense of Popa and Vaes exactly coincide with weight $0$ admissible representations of $\AC$.  Thus the admissible representation theory of Popa and Vaes is the restriction of ordinary representation theory of the tube algebra.  In a recent paper of Neshveyev and Yamashita \cite{NY2}, given an object of $Z(\text{ind-}\ca)$ they construct a representation of the fusion algebra.  They then show that the class of representations of the fusion algebra that arises in this way is exactly the class identified by Popa and Vaes.  Thus the equivalence of $Z(\text{ind-}\ca)$ and $Rep(\AC)$ observed by Vaes implies our result, however we feel the direct proof given here is instructive.

We now assume for the rest of the paper that $\W$ contains the strict tensor unit indexed by $0$, so that $X_{0}=id$.  From Proposition 3.1 we see that $\AW_{0,0}$ is $*$-isomorphic to the fusion algebra of $\ca$.  If $\phi$ is a function on $\Irr$, it defines a functional on $\Fus$ by sending $f=\sum_{k} f_{k} \in \AW_{0,0}$ (where $f_{k}\in \AW^{k}_{0,0}$) to 

$$\phi(f)=\sum_{k}\frac{\phi(X_k)}{d(X_k)}\grc{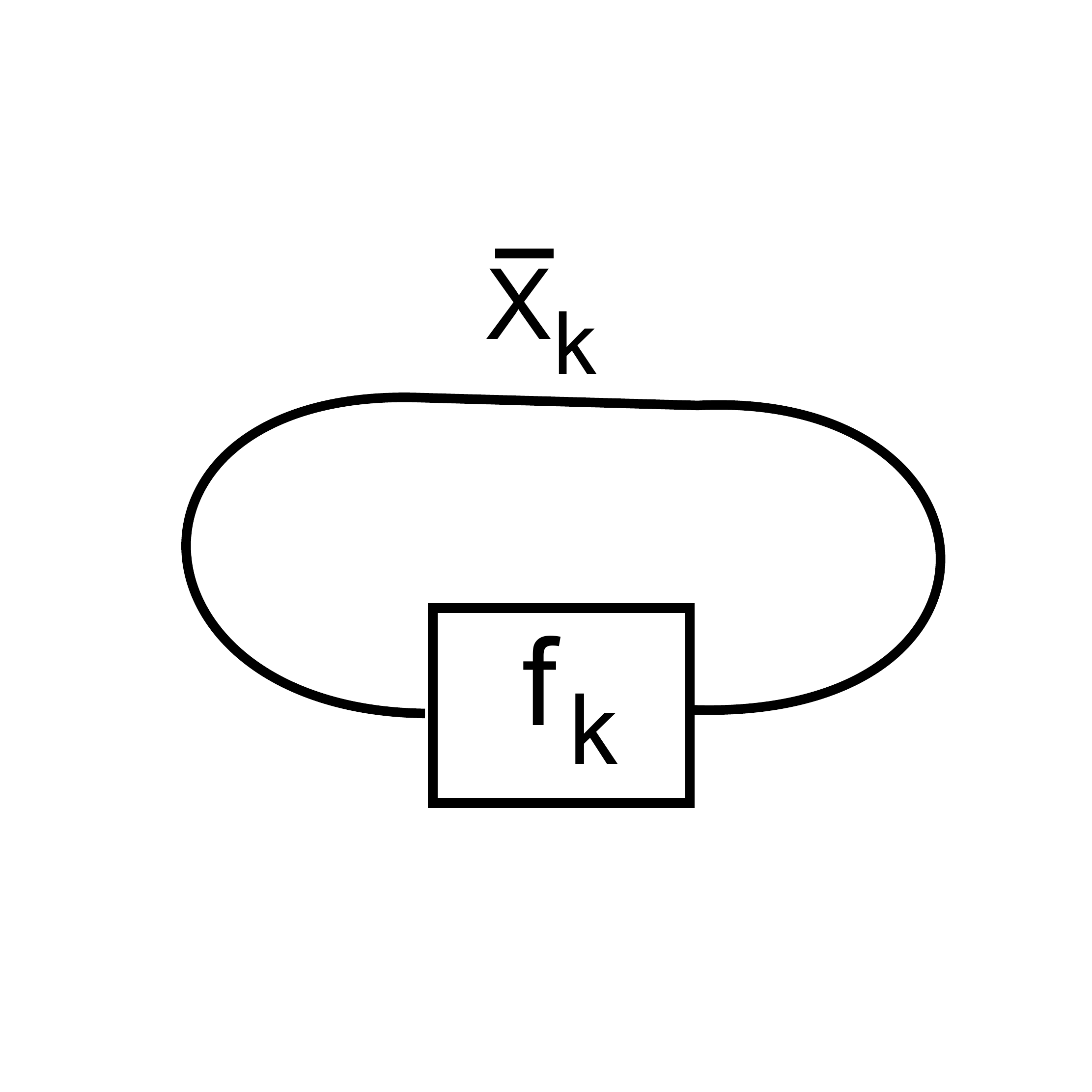}$$

\ \ This is because $f_{k}$ is really a scalar times the single string labeled $X_k$.  Now since any annular algebra has $\AW_{0,0}\cong \Fus $, we can naturally identify the algebraic duals $\widehat{(\AC_{0,0})}$ and $\widehat{(\AW_{0,0})}$, both as functions $\phi:\Irr \rightarrow \mathbb{C}$.  Recall from Definition 4.3 that $\Phi_{k}$ denotes the set of weight $k$ annular states on $\AC_{k,k}$, while $\Phi \W_{k}$ denotes the weight $k$ annular states on $\AW_{k,k}$.  We have the following lemma:

\begin{lemm} If $\phi:\Irr \rightarrow \mathbb{C}$, then for any full $\W$, $\phi\in \Phi\W_{0}$ if and only if $\phi\in \Phi_{0}$.
\end{lemm}

\begin{proof} Since we can embed $\AC$ as a subalgebra of $\AW$ as in the proof of Proposition 3.5, it is clear that $\phi\in \Phi\W_{0}\Rightarrow \phi\in \Phi_{0}$.  For the converse, suppose $\phi\in \Phi_{0}$. Let $f=\sum_{k} f^{k}_{0,m}\in \AW_{0,m}$, where $f^{k}_{0,m}\in \AW^{k}_{0,m}$.  Then we have 

$$\phi(f^{\#}\cdot f)= \sum_{j\in \Irr} \sum_{V\in onb(m, X_{j})} \phi(\left(f^{\#}(1\otimes V^{*})\right)\cdot \left((V\otimes 1)f\right))$$

But each $(V\otimes 1)f\in \AC_{0,j}$, and thus each term in the right hand sum is positive.  Therefore $\phi(f^{\#}\cdot f)\ge 0$ for all $m\in \W$.

\end{proof}

\begin{lemm} If $\phi: \Irr \rightarrow \C$.  Define $\phi^{op}:\Irr \rightarrow \C$  by $\phi^{op}(X_k)=\phi(\overline{X}_{k})$.  Then $\phi$ is an annular state if and only if $\phi^{op}$ is an annular state.

\end{lemm}

\begin{proof}  We only need to check the positivity condition.  Suppose $\phi\in \AC_{0,0}$.  Recall the map $r:\AC\rightarrow \AC$ introduced in the proof of Proposition 5.6, defined for $f\in \AC^{k}_{i,j}:=Mor(X_{k}\otimes X_{i}, X_{j}\otimes X_{k}$ by $r(f)=\overline{f}\in \AC^{\overline{k}}_{\overline{j}, \overline{i}}$.  $r$ is an anti-isomorphism with respect to annular multiplication. Then if $f\in \AW_{0,m}$, $\phi^{op}(f^{\#}\cdot f)=\phi(r(f^{\#}\cdot f))=\phi(r(f)\cdot r(f)^{\#})$. 
\end{proof}

\ \  Now we recall several definitions from \cite{PV}.  Let $\mathcal{C}$ be a rigid $C^{*}$-tensor category and let $Irr(\mathcal{C})$ be the set of simple objects.  

\begin{defi}\ A \textit{multiplier} on a rigid $C^{*}$ tensor category is a family of linear maps $\Theta_{\alpha, \beta}: End(\alpha\otimes \beta)\rightarrow End(\alpha\otimes \beta)\ \text{for all}\ \alpha, \beta\in Obj(\mathcal{C)}$ such that
\begin{enumerate}
\item
Each $\Theta_{\alpha,\beta}$ is $End(\alpha)\otimes End(\beta)$-bimodular 
\item
$\Theta_{\alpha_{1}\otimes \alpha_{2}, \beta_{1}\otimes \beta_{2}}(1\otimes X\otimes 1)=1\otimes\Theta_{\alpha_{2}, \beta_{1}}(X)\otimes 1\ \text{for all}\ \alpha_{i},\beta_{i}\in \mathcal{C}, X\in End(\alpha_{2}\otimes \beta_{2})$
\end{enumerate}

A multiplier is a \textit{cp-multiplier} if each $\Theta_{\alpha, \beta}$ is completely positive. 

\end{defi}

In \cite{PV}, Proposition 3.6, it is shown that multipliers are in one-one correspondence with functions $\phi: \Irr \rightarrow \C$.  For such a $\phi$, we define a multiplier  $\Theta^{\phi}_{\alpha, \beta}$ as follows:

For an object $\alpha\in \mathcal{C}$, and for $k\in \Irr$ with $X_{k}\prec \alpha \overline{\alpha}$, define the central projection in $End(\alpha\otimes\overline{\alpha})$

$$P^{k}_{\alpha \overline{\alpha}}:=\sum_{W\in onb(\alpha \overline{\alpha}, X_{k})} W^{*}W$$

Then for $x\in End(\alpha \otimes \beta)$, 

$$\Theta^{\phi}_{\alpha, \beta}(x)=\sum_{k\in \Irr} \phi(X_k) \grc{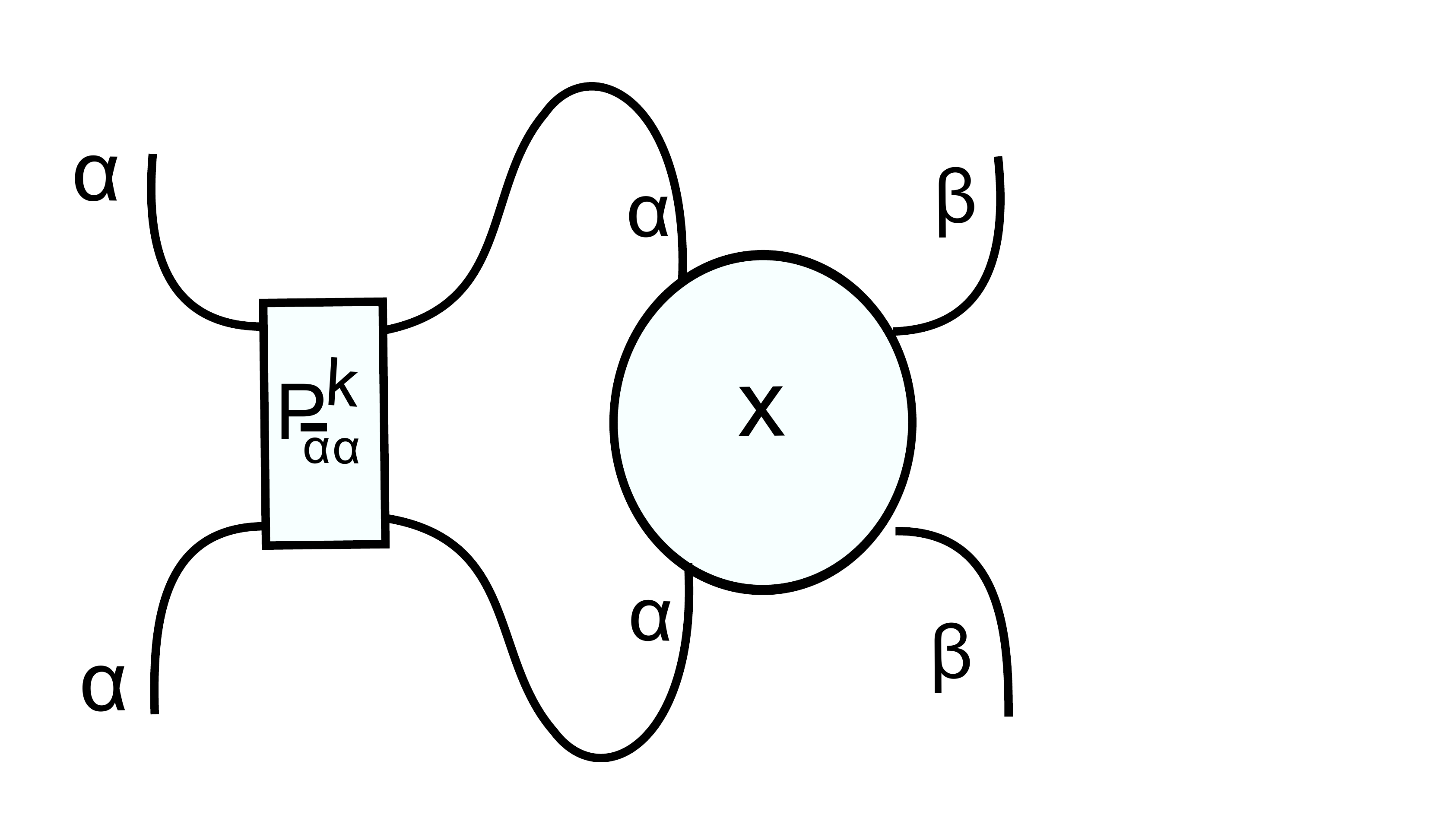}=\sum_{k\in \Irr}\phi(X_k)\grc{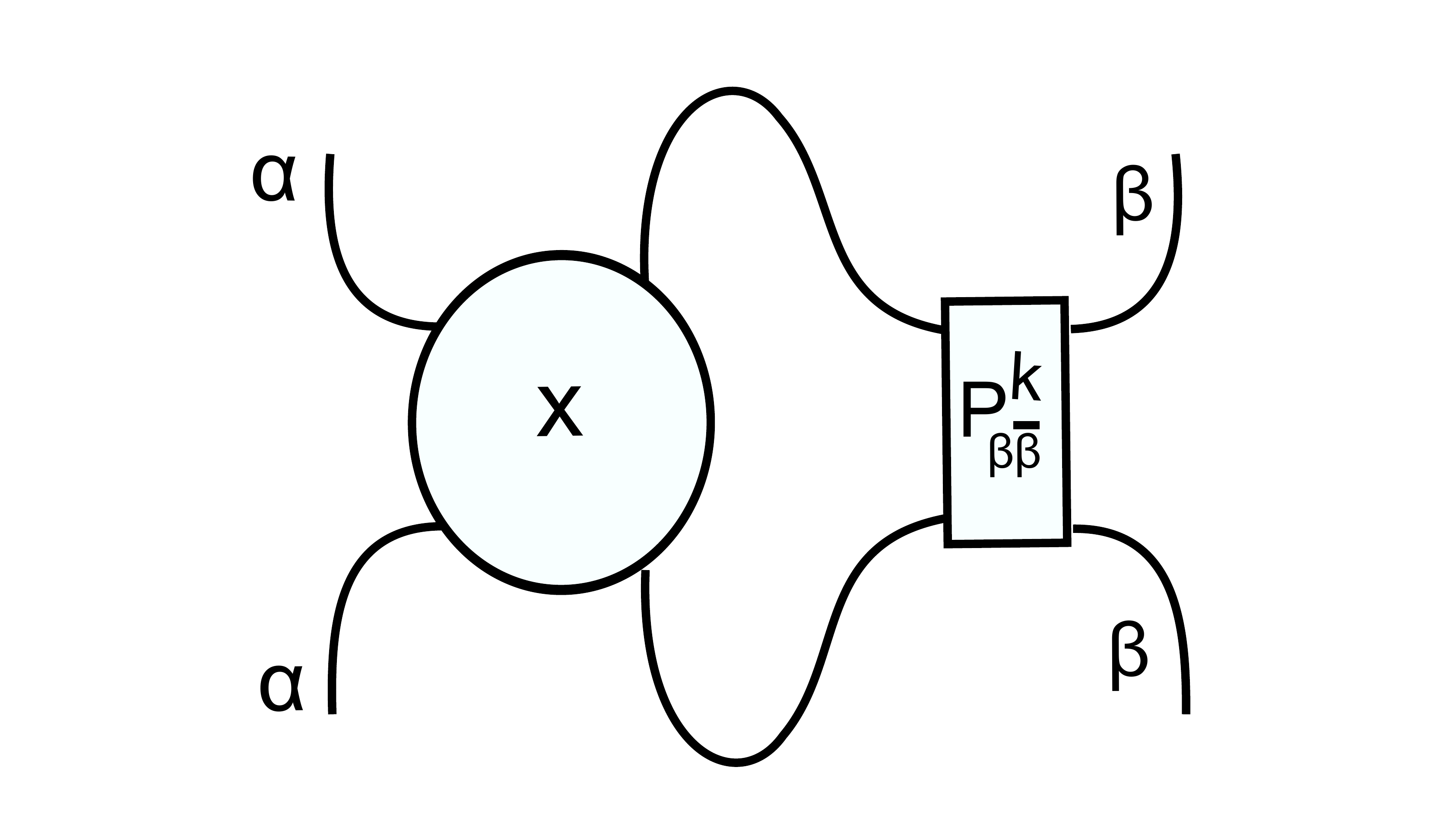}$$

  We note this sum is finite.  In the above pictures, we apply our conventions for horizontal strings from section 3 locally.  Popa and Vaes show every multiplier is of this form.  We abuse notation, and say a function $\phi:\Irr \rightarrow \mathbb{C}$ is a cp-multiplier if $\Theta^{\phi}$ is a cp-multiplier.  It is shown in \cite{PV} that if $\phi:\Irr \rightarrow \mathbb{C}$ is a $cp$-multiplier, then $d(.)\phi(.):\mathbb{C}[\Irr] \rightarrow \mathbb{C}$ is a state on the fusion algebra.  
 
\begin{defi} \begin{enumerate}
\item
A function $\phi: \Irr \rightarrow \mathbb{C}$ is called an \textit{admissible state} if $\frac{\phi(.)}{d(.)}$ is a cp-multiplier.
\item
A (non-degenerate) $*$-representation $\pi$ of $\AW_{0,0}\cong \C[\Irr]$ is called \textit{admissible} if every vector state in the representation is admissible.
\item
Define $\displaystyle \| \ \|_{u}:=\sup_{\pi\ \text{admissible}} \| \ \|_{\pi}$ on $\AW_{0,0}\cong \Fus $.  $C^{*}(\ca)$ is defined as the completion of $\AW_{0,0}\cong \Fus$ with respect to this universal norm.  It is shown in \cite{PV} that this is finite and a $C^{*}$-norm.
\end{enumerate}
\end{defi}

We will show that admissible states are exactly the same as weight $0$ annular states.  First, a lemma due to Popa and Vaes:

\begin{lemm}(\cite{PV}, Lemma 3.7)  Let $\mathcal{C}$ be a rigid $C^{*}$-tensor category and $\Theta$ a multiplier on $\mathcal{C}$.  Then the following are equivalent:
\begin{enumerate}
\item
For all $\alpha, \beta\in \mathcal{C}$, the map $\Theta_{\alpha, \beta}: End(\alpha\otimes \beta)\rightarrow End(\alpha\otimes \beta)$ is completely positive.
\item
For all $\alpha, \beta\in \mathcal{C}$, the map $\Theta_{\alpha,\beta}: End(\alpha\otimes \beta)\rightarrow End(\alpha\otimes\beta)$ is positive.
\item
For all $\alpha\in \mathcal{C}$ we have that $\Theta_{\alpha, \overline{\alpha}}(\overline{R}_{\alpha}\overline{R}^{*}_{\alpha})$ is positive.
\end{enumerate}
\end{lemm}

Note that $\overline{R}_{\alpha}\overline{R}^{*}_{\alpha}$ is given in pictures by \grc{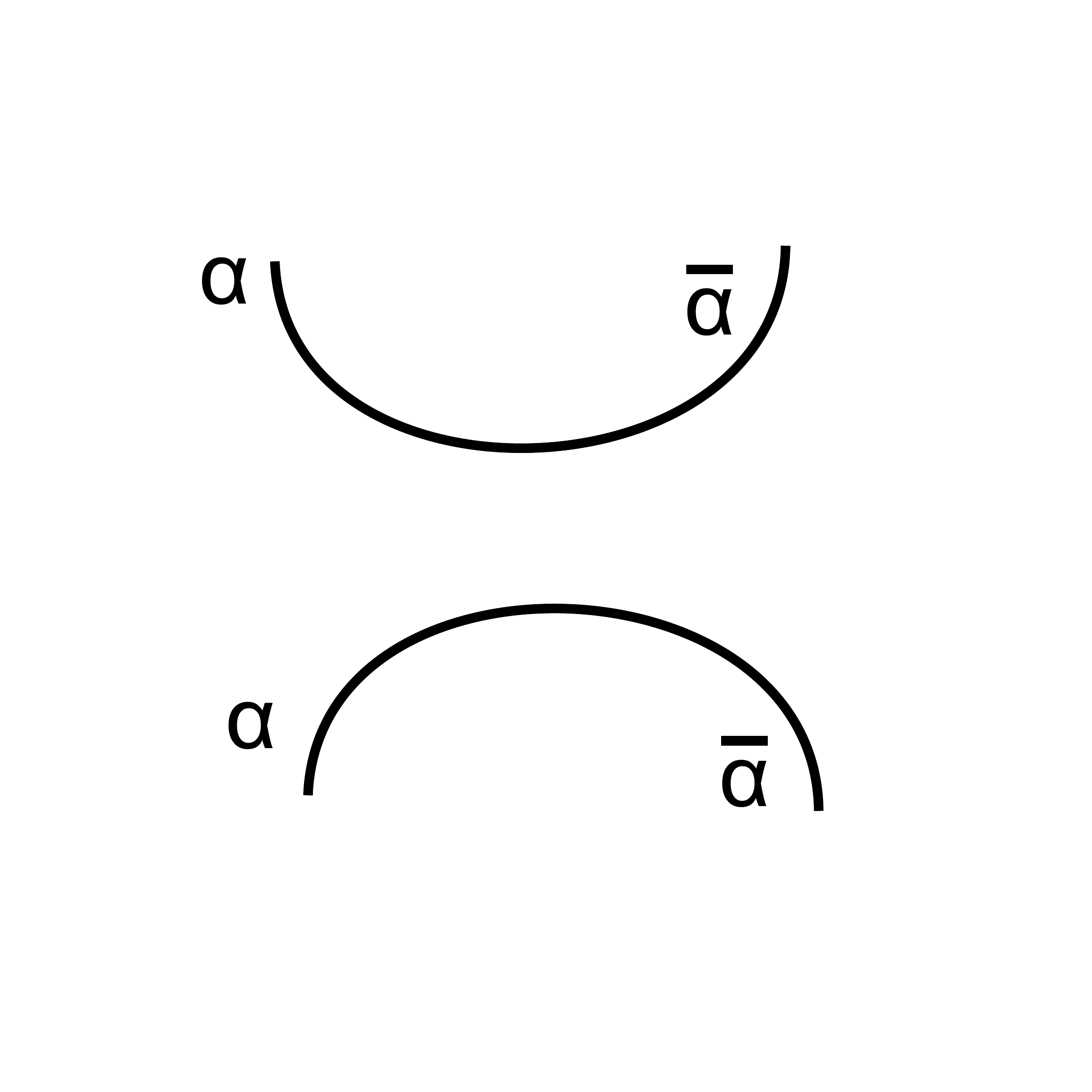}.

\begin{theo} $\phi$ is a weight $0$ annular state if and only if $\phi$ is admissible in the sense of Definition $6.4$.
\end{theo}

\begin{proof} \ First let $\phi:\Irr\rightarrow \C$ be an arbitrary function.   We define the multiplier $\Theta^{\psi}$ associated to the function $\psi(.):=\frac{\phi(.)}{d(.)}$ as above.  Take any vector $v\in End(\alpha \otimes \overline{\alpha})$.  Then we have $$\langle \Theta^{\psi}_{\alpha, \overline{\alpha}}(\overline{R}_{\alpha}\overline{R}^{*}_{\alpha})v,v\rangle =Tr(\Theta^{\psi}_{\alpha, \overline{\alpha}}(\overline{R}_{\alpha}\overline{R}^{*}_{\alpha}) v v^{*})$$
$$=\sum_{k\in \Irr} \frac{\phi(X_k)}{d(X_k)}\grd{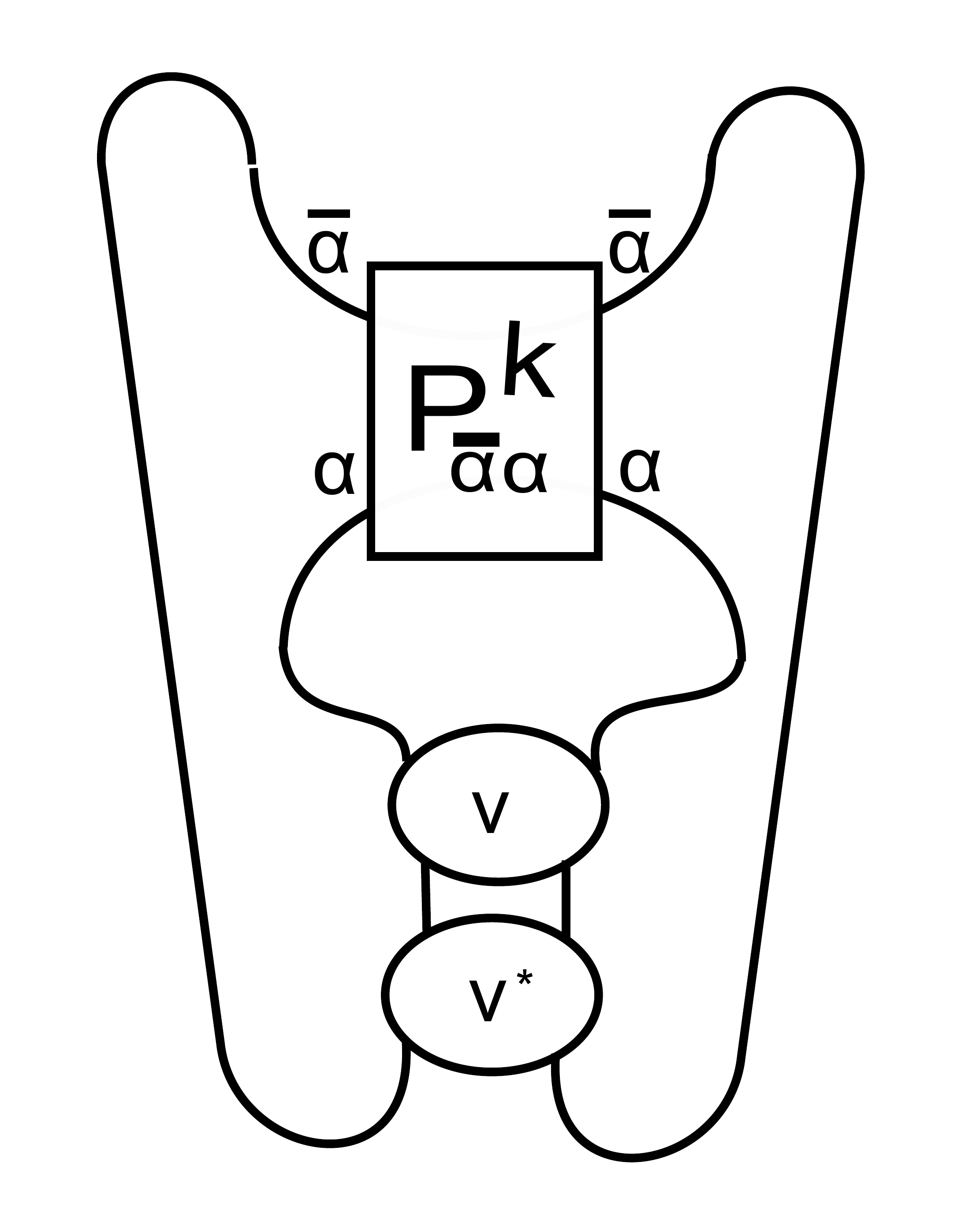}$$
$$= \sum_{k\in \Irr} \sum_{W\in onb(\overline{\alpha}\alpha, X_k)} \frac{\phi(X_k)}{d(X_k)} \grd{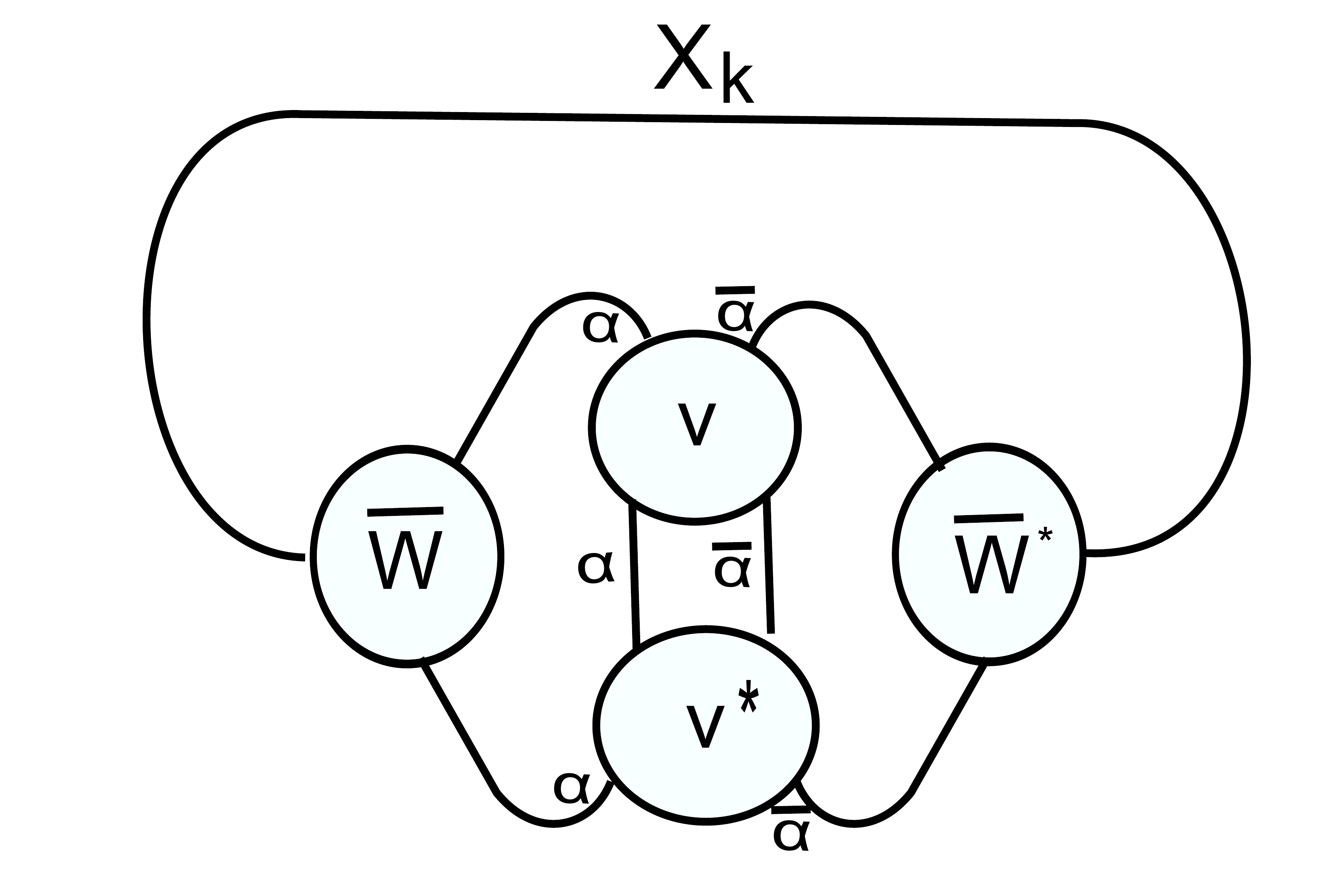}$$
$$=\phi^{op}\circ\Psi^{\overline{\alpha}\alpha} \left(\grc{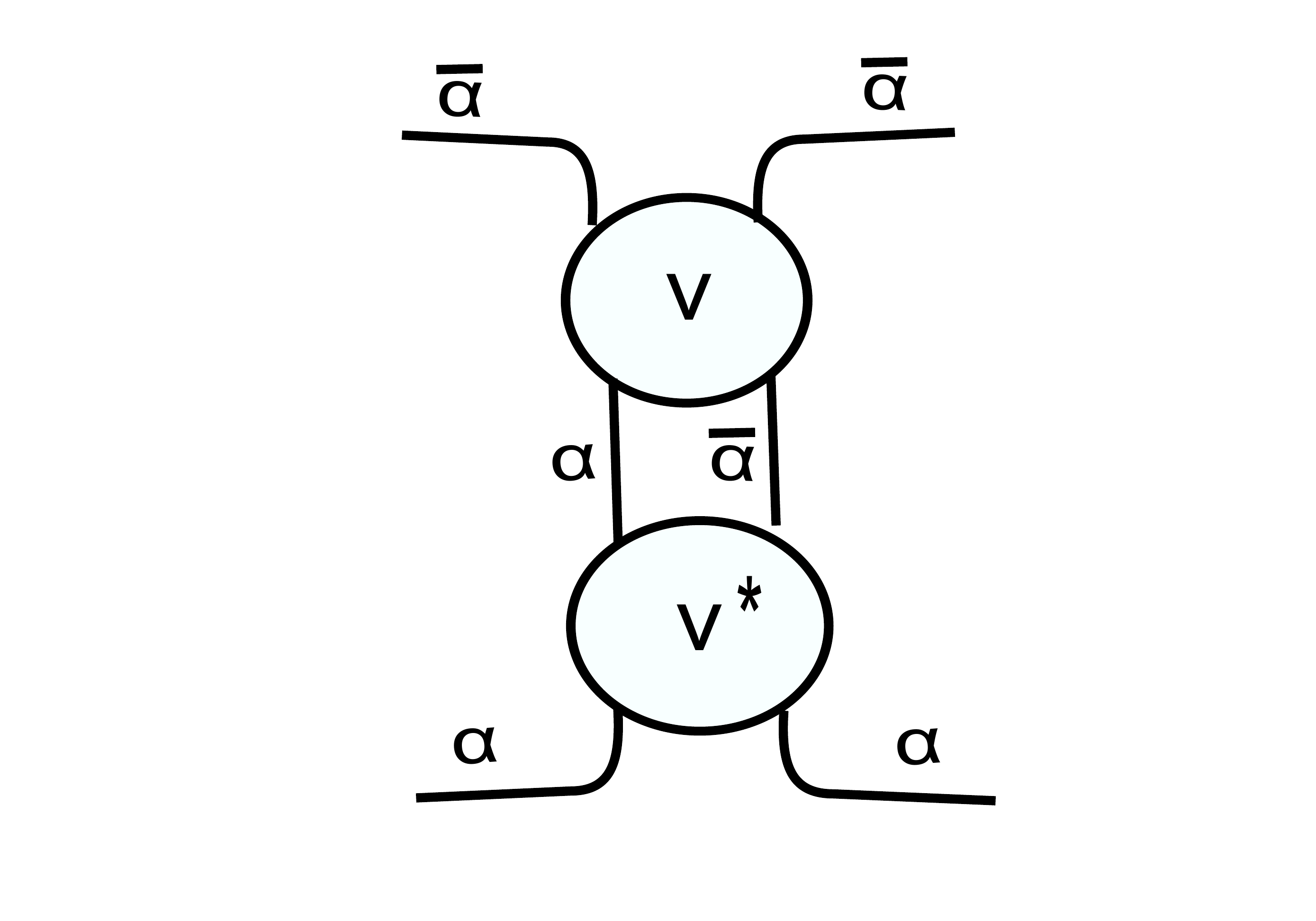}\right)$$

Here we use our graphical calculus conventions for side strings, and $\overline{W}$ represents the image of of the morphism $W$ under the appropriate duality functor.  In the last equality, we view $\phi^{op}$ as a functional on $\Fus$, as in the discussion preceeding Lemma 6.1.

Now, since weight $0$ annular states are the same for all full annular categories, without loss of generality we set $\W=[Obj(\ca)]$.    Let $x:= (v^{*}\otimes 1_{\alpha})\circ (1_{\alpha}\otimes R_{\alpha})\in \widetilde{\AW}^{\alpha}_{0, \alpha \overline{\alpha}}$.  Then the last term in the above equality can be interpreted as

$$\phi^{op}\circ \Psi^{\overline{\alpha}\alpha}(x^{\#}\cdot x)=\phi^{op}(\Psi^{\alpha}(x)^{\#}\cdot \Psi^{\alpha}(x)).$$

If $\phi\in \Phi\W_{0}$, then by Lemma $6.2$ the above expression is non-negative for all $v, \alpha$, hence $\Theta^{\psi}$ is a cp-multiplier by Lemma 6.5 (3).  

Conversely, if $\Theta^{\psi}$ is a cp-multiplier we need to show that $\phi$ is an annular state, and it suffices to show $\phi^{op}\in \Phi\W_{0}$.  But by Lemma $6.1$ it suffices to check this for the tube algebra.  Let $f=\sum_{k\in \Irr} f^{k}_{0,m}\in \AC_{0,m}$.  Set $\alpha:=\oplus X_{k}$ where $k$ appears in the description of $f$.  Then since $f^{k}_{0,m}\ne 0$,  $X_{m}\prec \alpha \overline{\alpha}$.  Then set 
$$v^{*}:=\sum_{k\in \Irr}\sum_{W\in onb(X_{m}, \alpha \overline{\alpha})} (1_{\alpha \overline{\alpha}}\otimes\overline{R}^{*}_{k})(W\otimes 1_{k}\otimes 1_{\overline{k}}) (f^{k}_{0,m}\otimes 1_{\overline{k}})\in End(\alpha \overline{\alpha})$$

Then since $\Theta^{\psi}$ is a cp-multiplier,  $$\phi^{op}(f^{\#}\cdot f)=\langle \Theta^{\psi}_{\alpha,\overline{\alpha}}(\overline{R}_{\alpha}\overline{R}^{*}_{\alpha})v, v\rangle\ge 0$$

Thus $\phi$ is an annular state.

\end{proof}

\begin{cor} $(\pi, H)$ be a $*$ representation of the fusion algebra $\Fus$.  Then the following are equivalent:

\begin{enumerate}
\item
$(\pi, H)$ is admissible in the sense of Definition 4.9, namely, there exists a non-degenerate $*$-representation of $\AC$ which restricted to $\AC_{0,0}$ is unitarily equivalent to $(\pi, H)$.
\item
$(\pi, H)$ is admissible in the sense of Popa and Vaes, Definition $6.4.$
\end{enumerate}
\end{cor}

\begin{cor}
$C^{*}(\ca)\cong C^{*}(\AC_{0,0})$
\end{cor}

We consider the affine state $1_{\mathcal{C}}$ corresponding to the trivial representation, given by $1_{\mathcal{C}}(X)=d(X)$ for each $X\in Irr(\mathcal{C})$.  We note that if $\phi\in \Phi \W_{0}$, then $\phi$ is a state on $C^{*}(A\W_{0,0})$.  Furthermore, for each simple object $X\in Irr(\ca)$, $||X||_{u}=d(X)$, since we have $||X||\le d(X)$ by Corollary 4.7, and this value is realized in the trivial representation.  Furthermore, $C^{*}(A\W_{0,0})$ contains the one dimensional subspace $\hat{X}\cong \C[X]$ for $X\in Irr(\ca)$.  Hence for an annular state $\phi\in \Phi \W_{0}$, when viewed as a state on $C^{*}(A\W_{0,0})$, $||\phi|_{\hat{X}}||=|\frac{\phi(X)}{d(X)}|$.  Hence the numbers $|\frac{\phi(X)}{d(X)}|$ are ``local norms" of the state $\phi$.   Now, we recall the definitions of approximation and rigidity properties given by Popa and Vaes, but present them translated into our annular language.

\begin{defi} \cite{PV} A rigid $C^{*}$ tensor category (with $\Irr$ countable) is said

\begin{enumerate}
\item
to be \textit{amenable} if there exists a sequence of finitely supported weight $0$ annular states $\phi_{n}$ that converges to $1_{\mathcal{C}}$ pointwise on $Irr(\mathcal{C})$.
\item
to have \textit{property (T)} if for every sequence of annular states $\phi_{n}$ which converges pointwise to $1_{\mathcal{C}}$, the sequence of functions $\frac{\phi_{n}(\ .\ )}{d(\ .\ )}$ converges uniformly to $1$ on $Irr(\ca)$.
\item
to have the \textit{Haagerup property} if there exists a sequence of annular states $\phi_{n}$ each of which vanish at $\infty$ (for every $\epsilon$, there exists a finite subset $K\subseteq Irr(\ca)$ such that $|\frac{\phi(X)}{d(X)}|<\epsilon$ for all $X\in K^{c}$), which converge to $1_{\mathcal{C}}$ pointwise.

\end{enumerate}

\end{defi}

The statements above can thus be interpreted as statements about the convergence of states in the ``local norms" of annular states.

\ \ It is shown in \cite{PV} that these definitions are equivalent to the usual ones given in terms of symmetric enveloping algebras in the case where $\mathcal{C}$ is the even part of some subfactor standard invariant.  Popa and Vaes also give several very interesting examples of categories with each of these approximation properties.  We recall the following results of Popa and Vaes, which can be seen in our setting:

\begin{prop} (Popa-Vaes) \cite{PV}
\begin{enumerate}
\item
For a discrete group $G$ and $\ca\cong G-Vec$, $\ca$ has an analytical property if and only if $G$ has the corresponding property as a discrete group.
\item
$TLJ(2)$ is amenable.
\item
$TLJ(\delta)$ has the Haagerup property for $\delta\ge 2$.
\end{enumerate}
\end{prop}

 The first two items in the above proposition are due to Popa and have been known for some time.   We direct the reader to the paper of Popa and Vaes \cite{PV} for more details on the third item.  They use results about the corresponding quantum groups $SU_{q}(2)$ obtained in \cite{DeC}.  We note that while Popa and Vaes consider only the even part of $TLJ(\delta)$, their proofs work more generally.  One could deduce these results from the analysis of our examples above, by choosing weight $0$ characters $t\in (-\delta, \delta)$ such that $t\rightarrow \delta$.  That these annular states are $c_{0}$ is straightforward.   We have not emphasized this because the arguments are exactly the same as in \cite{PV}.  We remark that in \cite{Bro}, Brothier and Jones give a proof that $TLJ(\delta)$ has the Haagerup property using annular representations to directly construct bimodules over a symmetric enveloping inclusion of a subfactor whose standard invariant is $TLJ(\delta)$.  We also remark that using the results of Arano \cite{Ar}, Popa and Vaes show that the categories $Rep(SU_{q}(N))$ for $N\ge 3$ odd have property (T).

\end{subsection}
\begin{subsection}{Higher Weights and Approximation Properties}

If we try to define approximation and rigidity properties for the higher weights (i.e. $k\in \W$ with $k\ne 0$) in the vein of Popa and Vaes, we do not have the notion of a weight $k$ trivial representation, since the higher weight spaces in the trivial representation are $\{0\}$.  However, we do have the notions of annular states and universal norms.  In \cite{BG}, Brown and Guentner have characterized the three analytical properties described in the previous section in terms of $C^{*}$-algebra completions.  We will recast their work in the setting of annular algebras, and briefly show that definitions given here for weight $0$ agree with the definitions of Popa and Vaes from the previous section.  This will allow us to define approximation and rigidity properties for higher weight spaces in these terms.  We would like to thank Ben Hayes for pointing us in this direction.

\ \ First we define a ``point-wise product" of annular states and show that they are again affine states. Our definition requires some knowledge of the tensor functor on $Rep(\AW)$ found in \cite{DGG}, Section 4, however the final result will be easy to understand. Let $\phi\in \Phi \W_{k}$ and $\psi\in \Phi\W_{0}$.  We define the tensor product of the two states as follows. Let $\xi_{\phi}$ be a vector in a Hilbert representation $H_{\phi}$ realizing $\phi$, i.e. $\phi(x)=\langle \pi(x)\xi_{\phi}, \xi_{\phi}\rangle$ and similarly for $\psi$.  Then we define $\phi\boxtimes \psi\in \Phi \W_{k}$ as the vector state in the tensor product representation corresponding to the vector $ \xi_{\phi}\otimes 1_{k}\otimes\xi_{\psi}$ in the weight $k$ space of $H_{\phi}\boxtimes H_{\psi}$ (see \cite{DGG}, Section 4).  This vector state is an annular state by the definition of the inner product on this Hilbert space and Proposition $4.4$ of \cite{DGG}, Section 4. 

\begin{lemm} Let $\displaystyle x=\sum_{m\in \Irr} x_{m}\in \AW_{k,k}$, where each $x_{m}\in \AW^{m}_{k,k}$.  Then $ \displaystyle \phi\boxtimes\psi(x)=\sum_{m\in \Irr}\frac{\psi(X_m)}{d(X_m)}\phi(x_{m})$.
\end{lemm}

\begin{proof} To compute the value of this state, let $(c^{\phi}_{n})_{n\in \Irr}$ and $(c^{\psi}_{n})_{n\in \Irr}$ be the commutativity constraints corresponding to $\phi$ and $\psi$ respectively.  Then we have from the definition in \cite{DGG} before Lemma $4.2$, $$\phi\boxtimes \psi(x_m)=\grc{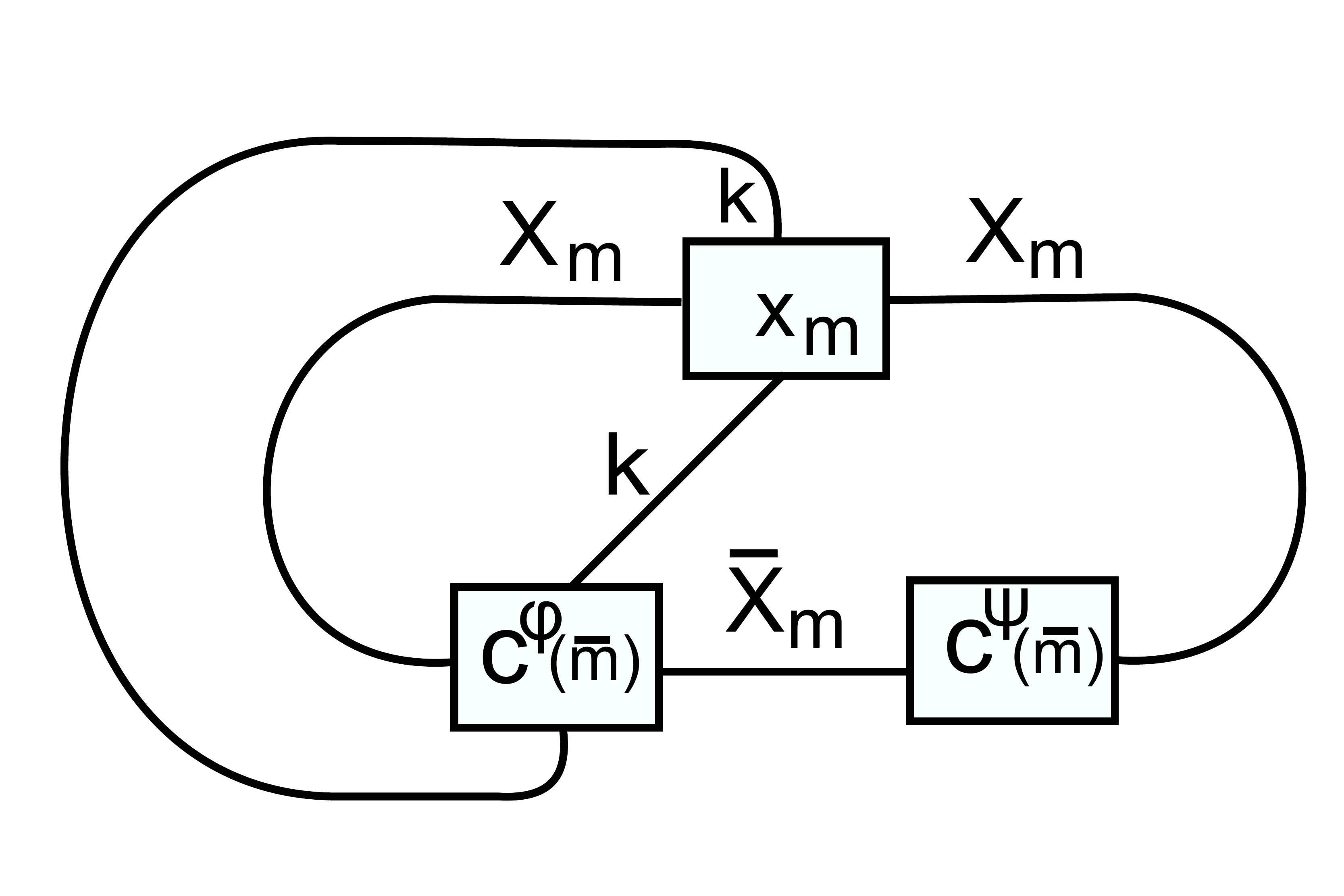}$$ 
$$=\sum_{l\in \Irr}\grc{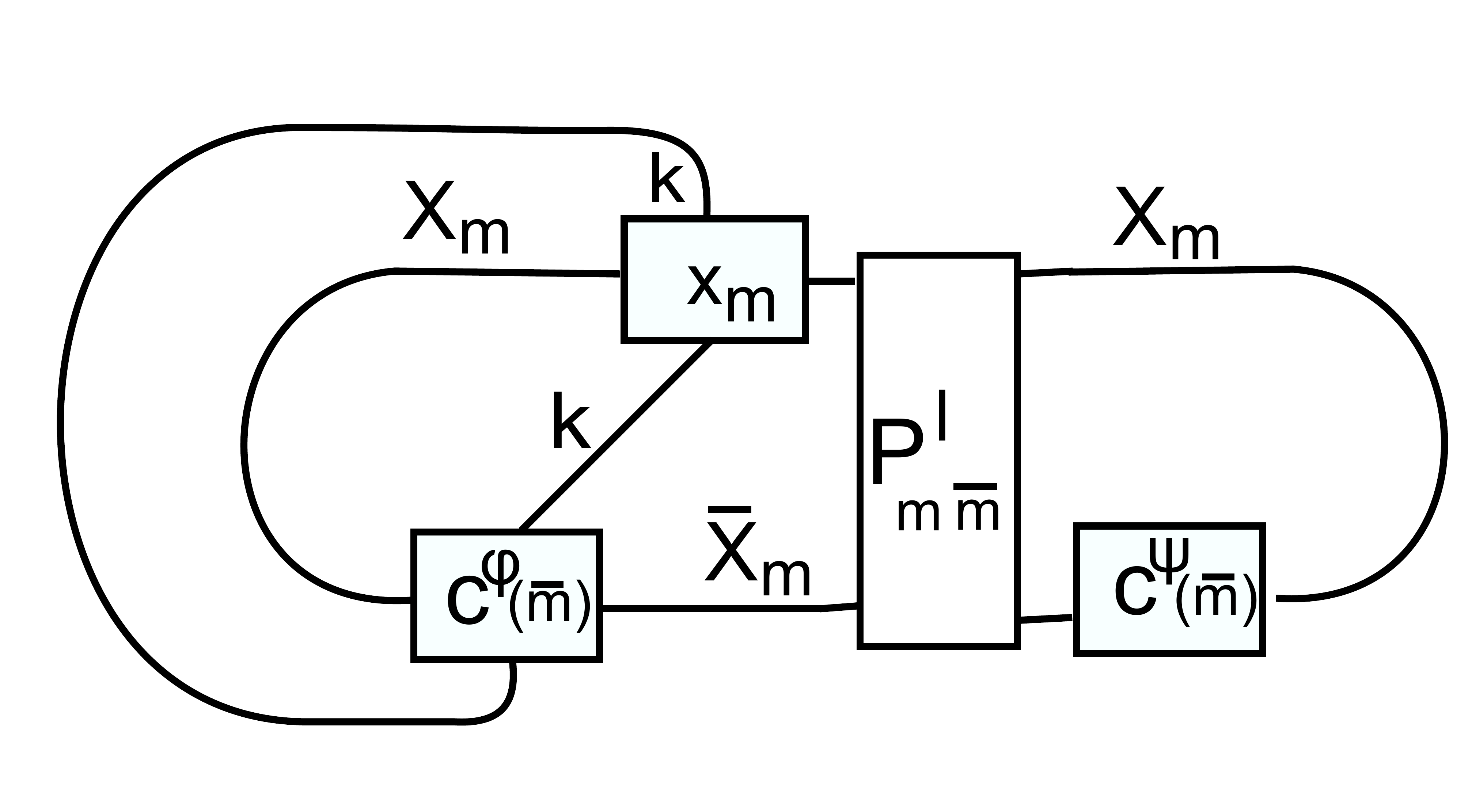},$$

where $P^{l}_{m\overline{m}}\in End(X_m\overline{X_m})$ is the central projection of $X_m \overline{X}_m$ onto the summand of sub-objects isomorphic to $X_l$.  Reading the diagram sideways as in our graphical calculus convention, we see that we have a hom from identity to identity factoring through the  $P^{l}_{m\overline{m}}$, but since each $X_l$ is simple, all these terms are $0$ except when $l=0$.  But $P^{0}_{m\overline{m}}$ is the $X_m$ Jones projection, and thus the above is equal to

$$\frac{1}{d(X_m)}\grc{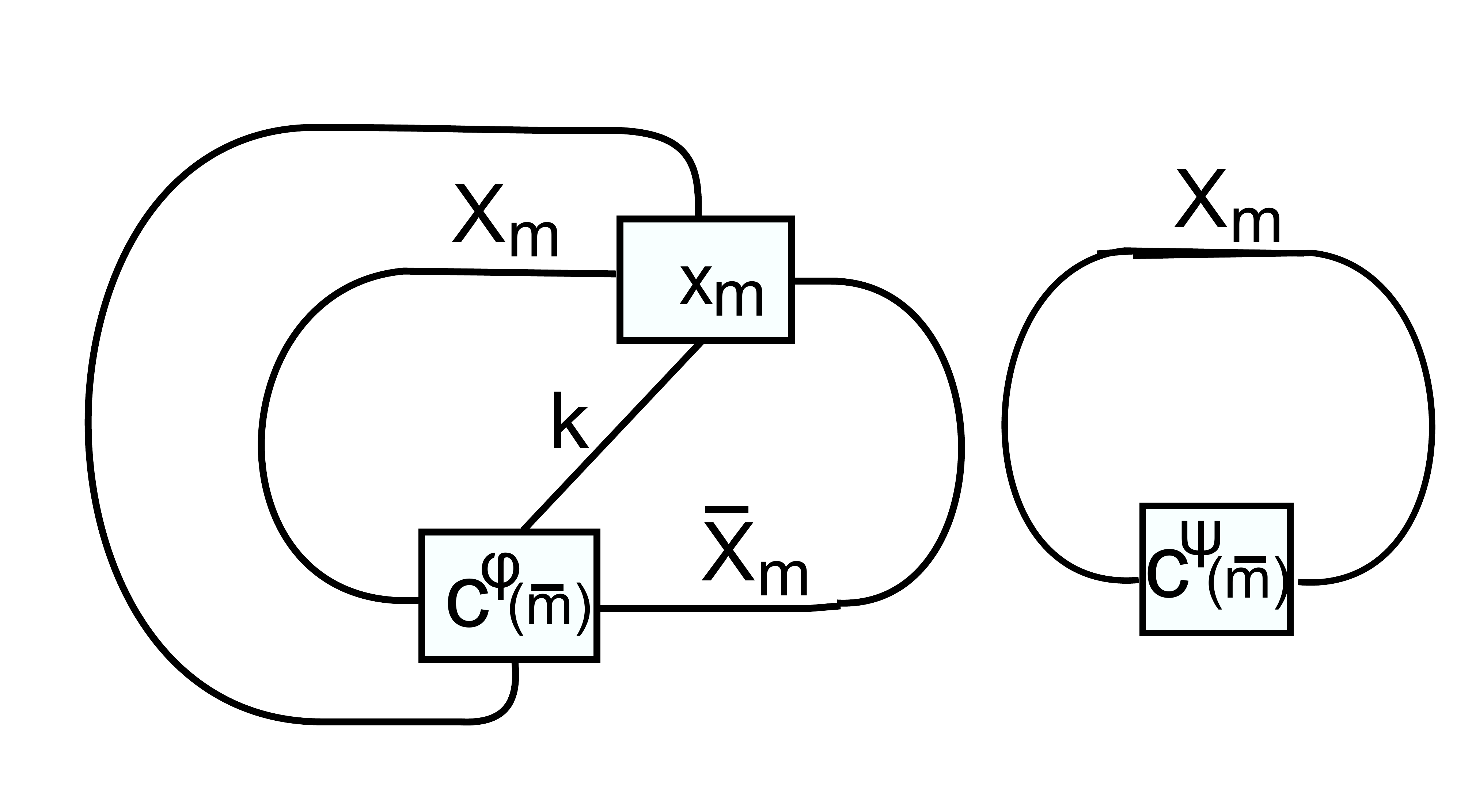}=\frac{\psi(X_m)}{d(X_m)}\phi(x_m)$$

Extending by linearity, we see that for $x=\sum_{W} x_{w}\in A\W_{k,k}$, where each $x_{w}\in A\W^{W}_{k,k}$, 

$$\phi\boxtimes\psi(x)=\sum_{m\in \Irr}\frac{\psi(X_m)}{d(X_m)}\phi(x_{m}).$$

\end{proof}

Although the fact that this is an annular state in general requires \cite{DGG}, the reader unfamiliar with this work can simply use the above lemma for the rest of this section.  We note that in the weight $0$ case this pointwise product can be understood as resulting from the composition of cp-multipliers as in the proof of Proposition 5.3 in \cite{PV}.  We proceed to define universal norms with respect to certain ideals of $\ell^{\infty}(\Irr)$, following \cite{BG}. 

Consider the annular centralizer algebra of weight $k\in \W$, $\displaystyle \AW_{k,k}\cong\bigoplus_{m\in \Irr} \AW^{m}_{k,k}$.  We equip each finite dimensional vector space $\AW^{m}_{k,k}$ with the restriction of the universal norm $||\ .\ ||_{u}$.

\begin{defi}  Let $D\triangleleft \ell^{\infty}(\Irr)$ be an algebraic ideal.  If $\phi\in (C^{*}(\AW_{k,k}))^{\#}$ (the space of continuous linear functionals on $C^{*}(\AW_{k,k})$ ), then we say $\phi$ is $D$-class if the  function $\hat{\phi}:\Irr\rightarrow \mathbb{R}_{+}$ defined by $\hat{\phi}(X_m):=\| \phi|_{\AW^{m}_{k,k}} \| $ is in $D$.
\end{defi}

We notice that $\hat{\phi}\in \ell^{\infty}(\Irr)$ since $\| \phi|_{\AW^{m}_{k,k}} \| \le \| \phi \|$, so the definition above makes sense.

\begin{defi}  An admissible representation $(\pi, V)$ of $\AW_{k,k}$ is $D$-class if there is a dense subspace $\{\eta_{\alpha}\}\subseteq V$ such that all the matrix coefficients $\langle \pi(.)\eta_{\alpha},\eta_{\beta}\rangle$ are $D$-class. 
\end{defi}

Since we only require dense subspaces, this class is closed under arbitrary direct sums.

\begin{defi} We define a $C^{*}$ semi-norm on $\AW_{k,k}$ by $||f||_{D}:=sup\{ ||f||_{V}\ :\ V\ \text{is}\ D\text{-class}\}$.  
\end{defi}

We remark that this semi-norm is finite since it is bounded by the universal norm, and a $C^{*}$ semi-norm since arbitrary direct sums of D-class representations are $D$-class, hence we may take such a direct sum over all $D$-class representations to realize this semi-norm.  We call this representation $\pi_{D}$.  Let $Ker_{D}\triangleleft \AW_{k,k}$ denote the kernel of this representation.  Then we define the $C^{*}$-algebra  $C^{*}_{D}(\AW_{k,k}):=\overline{\AW_{k,k}/Ker_{D}}^{||\ .\ ||_{D}}$.  We have a natural homomorphism $\gamma_{D}:C^{*}(\AW_{k,k})\rightarrow C^{*}_{D}(\AW_{k,k})$, and we notice that if $D=\ell^{\infty}(\Irr)$, then the natural homomorphism is an isomorphism so that $C^{*}(\AW_{k,k})\cong C^{*}_{\ell^{\infty}}(\AW_{k,k})$.

 \begin{defi}\ We say an ideal $D\subseteq \ell^{\infty}(\Irr)$ is \textit{$k$-translation invariant} if for all $\phi\in (\AW_{k,k})^{\#}$ such that $\phi$ is $D$-class, then for any $x, y\in \AW_{k,k}$, $\phi(x\cdot (\ .\ ) \cdot y)$ is $D$-class.
\end{defi}

 For $k=0$, this is equivalent to $D$ being invariant under left and right actions of the fusion algebras.  We present the following lemma, following \cite{BG}, to underline the role of translation invariance for the rest of the section:

\begin{lemm} If $D$ is $k$-translation invariant , $\phi \in \Phi \W_{k}$, and $\phi$ is $D$-class, then the GNS representation of $C^{*}(\AW_{k,k})$ with respect to $\phi$ is $D$-class.
\end{lemm}

\begin{proof} Let $\Omega_{\phi}$ be a cyclic vector for $\phi$.  Then for $f,g,h\in \AW_{k,k}$ we have $\langle \pi_{\phi}(f)g\Omega_{\phi}, h \Omega_{\phi} \rangle=\phi(h^{\#}\cdot f\cdot g)$.  By translation invariance, this implies $\langle\ \pi_{\phi}(\ .\ )g \Omega_{\phi}, h \Omega_{\phi} \rangle$ is $D$-class.  Since $\{g \Omega_{\phi}\ :g\in \AW_{k,k}\}$ is dense in the $GNS$ representation, this representation $D$-class.
\end{proof} 

\begin{cor} If $D$ is $k$-translation invariant, $f\in \AW_{k,k}$, then $\displaystyle \| f \|^{2}_{D}=\sup_{\phi\in \Phi \W_{k}\bigcap D\text{-class}} \phi(f^{\#}\cdot f)$.
\end{cor}

\begin{lemm} If $\phi\in \Phi\W_{0}$, then $\phi$ is $D$-class if and only if  $|\frac{\phi(\ .\ )}{d(\ .\  )}|\in D$.
\end{lemm}

\begin{proof} We note that $\phi$ is $D$-class if and only if the function from $\Irr \rightarrow \C$ given by $\hat{\phi}(X)=||\phi|_{\AW^{X}_{0,0}}||$ is $D$-class.  But $\AW^{X}_{0,0}\cong \mathbb{C} X$, and for $\lambda\in \mathbb{C}$, $\phi(\lambda X)=\lambda\phi(X)$.  But $||X||_{u}=d(X)$, and thus $\hat{\phi}(X)=|\frac{\phi(X)}{d(X)}|$.

\end{proof}

\medskip

  The following lemma is a direct adaptation of \cite{BG}, Theorem 3.2.  We use an almost identical proof  with the exception that we are now using annular states and $\boxtimes$ defined above, instead of positive definite functions and pointwise product.

\begin{lemm} If $D\subseteq \ell^{\infty}(\Irr)$ is a $0$-translation invariant ideal, then the canonical homomorphism $\gamma_{D}:C^{*}(\AW_{0,0})\rightarrow C^{*}_{D}(\AW_{0,0})$ is an isomorphism if and only if there exists a sequence $\{\phi_{n}\}\subseteq \Phi \W_{0}\bigcap D\text{-class}$ such that  $\frac{\phi_{n}(\ .\ )}{d(\ . \ )}\rightarrow 1$  point-wise.
\end{lemm}

\begin{proof}  First suppose the canonical map $C^{*}(\AW_{0,0})\rightarrow C^{*}_{D}(\AW_{0,0})$  is an isomorphism.  Then there exists a faithful $D$-class representation $\pi$ of $C^{*}(\AW_{0,0})$.  Taking infinite direct sums if necessary, we can assume that $\pi(C^{*}(\AW_{0,0}))$ contains no compact operators.  Then by Glimm's Lemma (see, for example \cite{BO}, Lemma 1.4.11 ), for any state $\phi$ of $C^{*}(\AW_{0,0})$ there exists a sequence of vector states $\omega_{\eta_{n}}\rightarrow \phi$.  By definition of $D$-class,  there is a dense subspace of vectors whose vector states are $D$-class. We can thus approximate the vector states with $D$-class vector states.  Setting $\phi=1_{\ca}$, the trivial representation vector state described in the previous section,  we have one direction of our lemma.
 
Now suppose that there exists a sequence of functions $\phi_{n}\in \Phi \W_{0}$ such that $|\frac{\phi_{n}(\ .\ )}{d(\ .\ )}|\in D$, and $\frac{\phi_{n}(\ .\ )}{d(\ .\ )}\rightarrow 1$ point-wise. By the above corollary, we simply need to show that the collection of $D$-class annular states is weak- $*$ dense in $\Phi \W_{0}$. Let $\psi\in \Phi \W_{0}$ be arbitrary.  Then $\psi \boxtimes \phi_{n} :\Irr\rightarrow \C$ is $D$-class since $|\frac{\psi \boxtimes \phi_{n}(\ . \ )}{d(\ .\ )}|=|\frac{\psi(\ .\ )}{d(\ .\ )}\frac{\phi_{n}(\ .\ )}{d(\ .\ )}|$.  Since $|\frac{\phi_{n}(\ . \ )}{d(\ .\ )}|\in D$ by Lemma 6.18, $|\frac{\psi(\ . \ )}{d(\ . \ )}\frac{\phi_{n}(\ .\ )}{d(\ . \ )}|\in D$ since $D$ is an ideal.  Now $\psi\boxtimes \phi_{n}(X)=\psi(X)\frac{\phi_{n}(X)}{d(X)}\rightarrow \psi(X)$ for all $X\in \Irr$ by hypothesis, and thus $\psi\boxtimes \phi_{n}\rightarrow \psi$ in the weak-$*$ topology on $\Phi \W_{0}$.  

\end{proof}

  We let $c_{c}\subseteq \ell^{\infty}(\Irr)$ be the algebraic ideal of finitely supported functions, and $c_{0}$ the ideal of functions vanishing at $\infty$. 

\begin{lemm} $c_{c}$ and $c_{0}$ are $k$-translation invariant ideals for all $k\in \W$.
\end{lemm}

\begin{proof} First we remark that to check translation invariance, it suffices to check that $\phi(f\ .\ )$ and $\phi(\ .\ f)$ are $D$-class for $\phi\in D$-class $\cap \Phi\W_{k}$ independently.  Furthermore by linearity it suffices to check for $f\in \AW^{j}_{k,k}$ for $j\in \Irr$.  

First we claim that for a fixed simple object $X$, $|\{ Y\in \Irr \ :\ N^{Y}_{XZ}\ne 0\}|\le d(X)^{2}$ for all $Z$.  To see this, note that if $Y\prec XZ$ by Frobenius reciprocity, $\overline{Z}\prec \overline{Y}X$, and thus $d(Z)=d(\overline{Z})\le d(\overline{Y})d(X)=d(X)d(Y)$, hence $1\le \frac{d(X)d(Y)}{d(Z)}$.  But we have  $\displaystyle | \{ Y\in Irr(\ca)\ :\ N^{Y}_{XZ}\ne 0\}| \le \sum_{Y\prec XZ} N^{Y}_{XZ}\le \sum_{Y\prec XZ} \frac{d(X)d(Y)}{d(Z)}N^{Y}_{XZ}=d(X)^{2}$.  

\ \ To show $c_{0}$ is $k$-translation invariant, we must show that if $\hat{\phi}\in c_{0}$ then the functional that maps $m\in \Irr$ to $\hat{\phi}(f\cdot (\ .\ ) )(m):=\| \phi(f\ .\ )|_{\AW^{m}_{k,k}} \| \in c_{0}$, where $f\in \AW^{j}_{k,k}$.  For $\epsilon>0$, there exists a finite subset $K\subset \Irr$ such that $\| \phi|_{\AW^{i}_{k,k}}\| <\frac{\epsilon}{d(X_j)^{2}\| f\| }$ for all $X_i\notin K$.  Now, for $Y\in K$, define $K^{\prime}_{Y}=\{X_s \in \Irr \ :\ Y\prec X_j X_s\}$.  This set is clearly finite by Frobenius reciprocity.  Thus $K^{\prime}:=\bigcup_{Y\in K} K^{\prime}_{Y}$ is finite.  For $X_t\notin K^{\prime}$, since $\displaystyle f\cdot AP^{t}_{k,k}\subseteq \bigoplus_{X_s\prec X_j X_t}\AW^{s}_{k,k}$ we have 
$$\| \phi(f\ .\ )|_{\AW^{t}_{k,k}}\| \le \| f\| \left(\sum_{X_s\prec X_j X_t}\| \phi(\ .\ )|_{\AW^{s}_{k,k}}\|\right) \le |\{X_s\prec X_j X_t\ :\ s\in \Irr \}| \frac{\epsilon}{d(X_j)^{2}}<\epsilon$$

The proof for right invariance works exactly the same.  Putting $\epsilon=0$ and carrying out the same argument gives the $c_{c}$ case.

\end{proof}

\medskip

We are now ready for the generalization of Popa and Vaes 's approximation properties to higher weights $k\in \W$

\begin{defi} Let $\AW$ be an annular algebra.  Let $k\in \W$.  Then $\AW$:

\begin{enumerate}
\item
is \textit{k-amenable} if $C^{*}(\AW_{k,k})=C^{*}_{c_{c}}(\AW_{k,k})$.
\item
has the \textit{k-Haagerup} property if $C^{*}(\AW_{k,k})=C^{*}_{c_{0}}(\AW_{k,k})$
\end{enumerate}

\end{defi}

\begin{cor} The definitions of the approximation properties (amenability and Haagerup) above for $k=0$ are equivalent to the definitions of Popa and Vaes.
\end{cor}

\begin{proof}
This follows easily from Lemma $6.18$.
\end{proof}

\begin{theo} Let $\AW$ be an annular algebra
\begin{enumerate}
\item
If $\ca$ is amenable, then it is $k$-amenable for all $k\in \W$.
\item
If $\ca$ has the Haagerup property, it has the $k$-Haagerup property for all $k\in \W$.
\end{enumerate}
\end{theo}

\begin{proof} Let $D$ be either $c_{c}(\Irr)$ or $c_{0}(\Irr)$.  By the above corollary, if $C^{*}_{D}(\AW_{0,0})=C^{*}_{\AW_{0,0}}$, then there exists $\phi_{n}\in \Phi \W_{0}$ with $|\frac{\phi_{n}(\ .\ )}{d(\ .\ )}|\in D$  such that $\frac{\phi_{n}(X)}{d(X)}$ converges to $1$ for each $X\in \Irr$.  Now,  if $\psi\in \Phi\W_{k}$,  we see that the function defined on $\Irr$ by $h_{n}(X):=||\psi\boxtimes\phi_{n}|_{\AW^{X}_{k,k}}||=|\frac{\phi_{n}(x)}{d(X)}|\ ||\psi|_{\AW^{X}_{k,k}}||\in D$.  Then $\psi\boxtimes\phi_{n}$ is a $D$-class annular state, and for every $f\in \AW_{k,k}$ with $\displaystyle f=\sum_{m\in \Irr}f_{m}$ where $f_{m}\in \AW^{m}_{k,k}$, $\displaystyle\psi\boxtimes \phi_{n}(f)=\sum_{m\in \Irr}\psi(f_{m})\frac{\phi_{n}(X_m)}{d(X_m)}\rightarrow \psi(f)$.  Thus the set of $D$-class states is weak-$*$ dense in the set of all states of $C^{*}(\AW_{k,k})$, hence $C^{*}_{D}(\AW_{k,k})=C^{*}(\AW_{k,k})$.

\end{proof}

\medskip

\ \ Weight $0$ approximation properties imply the corresponding approximation properties for all higher weights, thus supporting the notion that weight $0$ is certainly the ``right" place to define these properties for the whole category.  Using the $G-Vec$ example, it is easy to find examples with higher approximation properties but not the corresponding weight $0$ property. Simply find a group which is not amenable, or does not have the Haagerup property, but some centralizer subgroup does.   For example, if we take $G=F_{2}:=<a,b>$, then $G$ is not amenable but the centralizer subgroup of the element $a$ is isomorphic to $\mathbb{Z}$, so is amenable.

\medskip

\begin{cor}
\begin{enumerate}
\item
If $\ca\cong G-Vec$ for $G$ a discrete group, then for $X\in G\cong \Irr$, $\ca$ has a $X$- property (amenability, Haagerup) if and only if $Z_{G}(X)$ has the corresponding property as a discrete group.
\item
For all weight sets $\W$ and all $k\in \W$, $TLJ(2)$ is $k$-amenable.
\item
For all weight sets $\W$ and all $k\in \W$, $TLJ(\delta)$ has the $k$-Haagerup property for $\delta\ge 2$.
\end{enumerate}
\end{cor}

\begin{proof}
The first item follows from the fact that $C^{*}(\AC_{X,X})\cong C^{*}_{u}(Z_{G}(X))$.  The second two follow from Proposition $6.10$ and Theorem $6.23$.

\end{proof}

We conclude this section with the remark that rigidity properties do not enjoy the same ``globalness'' as approximation properties.  There is a natural definition of weight $k$ property (T) generalizing the notion for weight $0$, but unlike approximation properties, there seems to be no correspondence with the weight $0$ notion and the higher weight notions.  In the group case, the higher weight centralizer algebras are always subgroups of $G$, so this could simply reflect the fact that $(T)$ does not pass to subgroups.  However, property (T) appears to behave in unexpected ways for categories.  For example, it is shown in \cite{PV} using the work of Arano that categories with abelian fusion rules can have property (T) (\cite{Ar}).  This obscures the hope that property (T) von Neumann algebras can be constructed from property (T) categories, and also suggests that property (T) will not be as interesting for tensor categories in general as it is for groups.  This leads us to speculate whether a notion of (T) which does lead to von Neumann algebras with (T) can be formulated using the higher weights of the tube algebra.

\end{subsection}

\end{section}

\end{document}